\documentclass[11pt]{article}
\usepackage[top=1in, bottom=1in, left=1in, right=1in]{geometry}

\usepackage[pagebackref, letterpaper=true, colorlinks=true, urlcolor=blue, linkcolor=blue, citecolor=OliveGreen,]{hyperref}              
\usepackage{enumerate}                        
\usepackage{float}                            

\usepackage[ruled,vlined,linesnumbered]{algorithm2e}
\usepackage{environ}

\usepackage[dvipsnames]{xcolor}               
\usepackage{amsmath, amssymb}
\usepackage{graphicx}                         
\usepackage{booktabs}                         
\usepackage{mathtools}                        
\usepackage{amsthm}
\usepackage{thmtools}
\usepackage{thm-restate}
\usepackage{appendix}
\usepackage{subfig}                           
\usepackage{bm}                               
\usepackage{amsfonts}
\usepackage{bbm}                              
\usepackage[english]{babel}
\usepackage[utf8]{inputenc}
\usepackage[T1]{fontenc}
\usepackage[capitalise,nameinlink]{cleveref}
\usepackage{microtype}

	\newcounter{parentAlgoLine}
	\SetKwBlock{BEGIN}{}{}
	
	\makeatletter
	\NewEnviron{substeps}{%
		\refstepcounter{AlgoLine}%
		\protected@edef\theparentequation{\theequation}%
		\setcounter{parentAlgoLine}{\value{AlgoLine}-1}%
		\setcounter{AlgoLine}{0}%
		\BEGIN{\BODY}%
		\setcounter{AlgoLine}{\value{parentAlgoLine}}%
		\ignorespacesafterend
	}
	\makeatother
	
	\renewcommand{\epsilon}{\varepsilon}
	\renewcommand{\geq}{\geqslant}
	\renewcommand{\leq}{\leqslant}
	
	\newcommand\numberthis{\refstepcounter{equation}\tag{\theequation}}
	
	\newtheorem{theorem}{Theorem}[section]
	\newtheorem{lemma}[theorem]{Lemma}
	\newtheorem{corollary}[theorem]{Corollary}
	
	\theoremstyle{definition}
	\newtheorem{definition}[theorem]{Definition}
	\newtheorem{model}[theorem]{Model}
	\newtheorem{assumption}[theorem]{Assumption}
	
	\theoremstyle{remark}
	\newtheorem{remark}[theorem]{Remark}
	
	\crefname{lemma}{Lemma}{Lemmas}
	\crefname{theorem}{Theorem}{Theorems}
	\crefname{corollary}{Corollary}{Corollaries}
	\crefname{definition}{Definition}{Definitions}
	\crefname{model}{Model}{Models}
	\crefname{assumption}{Assumption}{Assumptions}
	
	\usepackage[createShortEnv, commandRef=Cref]{proof-at-the-end}
	\pgfkeys{/prAtEnd/global custom defaults/.style={
			text link=}
	}
	
	\makeatletter
	\def\thmt@innercounters{section,equation,theorem}
	\makeatother

	\newcommand{\norm}[1]{\left\lVert#1\right\rVert}
	\newcommand{\normtwo}[1]{\norm{#1}_{2}}
	\newcommand{\normone}[1]{\norm{#1}_1}
	\newcommand{\norminfty}[1]{\norm{#1}_\infty}
	\newcommand{\inner}[2]{\left\langle #1,#2 \right\rangle}
	\newcommand*{\normop}[1]{\norm{#1}_{\mathrm{op}}}
	\newcommand{\restr}[2]{\left.#1\right|_{#2}}
	\newcommand{\ceil}[1]{\left\lceil #1 \right\rceil}
	
	\newcommand{\EX}{\mathbb{E}}
	\newcommand{\probP}{\text{I\kern-0.15em P}}
	\newcommand\iid{\overset{\mathrm{i.i.d.}}{\sim} }
	\newcommand{\indep}{\perp \!\!\! \perp}
	
	\newcommand{\maximize}{\mathop{\textrm{maximize}}}
	\newcommand{\minimize}{\mathop{\textrm{minimize}}}
	\DeclareMathOperator*{\argmin}{\textbf{arg}\,\textbf{min}}
	
	\DeclareMathOperator{\Tr}{Tr}
	\DeclareMathOperator{\RE}{RE}
	\DeclareMathOperator{\OPT}{OPT}
	\DeclareMathOperator{\opt}{opt}
	\DeclareMathOperator{\conv}{conv}
	\DeclareMathOperator{\vectorize}{vec}
	\DeclareMathOperator{\rank}{rank}
	\newcommand{\err}{\mathrm{Err}}
	\newcommand{\Cov}{\mathrm{Cov}}
	\DeclareMathOperator{\diag}{diag}
	\DeclareMathOperator{\diam}{diam}
	
	\newcommand{\outlier}{\text{out}}
	\newcommand{\inlier}{\text{in}}
	\newcommand{\interval}{\text{interval}}
	\newcommand{\tol}{\text{tol}}
	\newcommand{\prev}{\text{prev}}
	\newcommand{\grad}{\text{grad}}
	\newcommand{\param}{\text{param}}
	\newcommand{\dgmm}{\text{DGMM}}
	\newcommand{\gmm}{\text{GMM}}
	\newcommand{\sgr}{\text{SGR}}
	\newcommand{\sgrgmm}{\text{SGR-GMM}}
	\newcommand{\mm}{\text{MM}}
	\newcommand{\lbfgs}{\text{L-BFGS}}
	
	\newcommand{\cB}{\mathcal B}
	\newcommand{\cD}{\mathcal D}
	\newcommand{\cI}{\mathcal I}
	\newcommand{\cM}{\mathcal M}
	\newcommand{\cN}{\mathcal N}
	\newcommand{\cO}{\mathcal O}
	\newcommand{\cX}{\mathcal X}
	
	\newcommand{\1}{\mathbbm{1}}
	\newcommand{\N}{\mathbb N}
	\newcommand{\R}{\mathbb R}
	
	\newcommand{\ba}{\mathbf{a}}
	\newcommand{\bg}{\mathbf{g}}
	\newcommand{\bw}{\mathbf{w}}
	\newcommand{\by}{\mathbf{y}}
	\newcommand{\bz}{\mathbf{z}}
	\newcommand{\bX}{\mathbf{X}}
	\newcommand{\bY}{\mathbf{Y}}
	
	\newcommand{\bpi}{\bm{\pi}}
	\newcommand{\btheta}{\bm{\theta}}
	\newcommand{\bkappa}{\bm{\kappa}}
	\newcommand{\bmu}{\bm{\mu}}
	\newcommand{\bxi}{\bm{\xi}}
	\newcommand{\bDelta}{\bm{\Delta}}
	
	\newcommand{\frakD}{\mathfrak{D}}
	
	\newlength{\dhatheight}
	\newcommand{\doublewidehat}[1]{%
		\settoheight{\dhatheight}{\ensuremath{\widehat{#1}}}%
		\addtolength{\dhatheight}{-0.2ex}%
		\widehat{\vphantom{\rule{1pt}{\dhatheight}}%
			\smash{\widehat{#1}}}}
	
	\allowdisplaybreaks
\begin{document}

\title{Robust Moment-Based Estimation \\ via Spectral Gradient Reweighting}
\author{Liu Zhang\thanks{Program in Applied and Computational Mathematics, Princeton University, Princeton, NJ 08544 USA (\href{mailto:lz1619@princeton.edu}{lz1619@princeton.edu}).}
	\and Amit Singer\thanks{Department of Mathematics and Program in Applied and Computational Mathematics, Princeton University, Princeton, NJ 08544 USA (\href{mailto:amits@math.princeton.edu}{amits@math.princeton.edu}).}
}
\date{}
\maketitle
\begin{abstract}
	Moment-based estimation is a theoretically attractive approach to parametric inference, especially when likelihood-based estimation is unavailable, misspecified, or computationally inconvenient. However, the moment equations involve sample averages, which makes moment-based estimation sensitive to outliers. We propose the SGR-GMM algorithm, a robust generalized method of moments (GMM) procedure that uses a spectral gradient reweighting (SGR) primitive to soft-reweight the per-observation gradients during the moment-matching optimization.  Our analysis has three layers. First, for a fixed center, the SGR primitive is formulated as an entropy-regularized spectral game between a sample-weight player and a density-matrix player, which is analyzed using classical multiplicative-weights and matrix-multiplicative-weights regret bounds. Second, we establish explicit convergence radius and finite termination bound for the fixed-center updates in the SGR primitive. Third, we prove a local finite-sample parameter estimation error bound with explicit dependence on the contamination fraction, inlier gradient stability, local GMM identification strength, and optimization accuracy. We further specialize the SGR-GMM algorithm to obtain a robust diagonally-weighted GMM (DGMM) estimator for estimating heteroscedastic low-rank Gaussian mixtures observed under additive Gaussian noise and strong contamination. In the numerical experiments, the SGR primitive produces nearly-oracle gradient estimation and the robust DGMM specialization substantially improves over non-robust moment baselines. The code and data are available at \url{https://github.com/liu-lzhang/sgr-gmm}.
\end{abstract}

\tableofcontents

\section{Introduction}
\label{sec:intro}

\subsection{Motivations}

Moment-based estimation is one of the standard approaches to obtain a computable estimator from a parametric statistical model. This includes the classical method of moments \cite{pearson1894contributions} and its subsequent extension, the generalized method of moments (GMM) \cite{hansen1982large}. In the GMM framework, the true parameter $\btheta^\star$ is identified as the solution to a system of polynomial equations in $\btheta$, given by the population moment conditions:
\begin{equation}
	\label{eq:intro-population-moment}
	\EX\left[g(\btheta; \bY)\right] = \bm{0} \in \R^q,
\end{equation}
where $g(\btheta; \bY) \in \R^q$ is a suitable moment function in $\btheta \in \Theta\subset\R^{p}$ (the parameter) and $\bY \in \R^d$ (the observed random variable). Given access to a set of observations $\{\by_n\}_{n=1}^N$, the GMM estimator is obtained by replacing the population moment conditions in \cref{eq:intro-population-moment} by their sample analogues and solving a weighted moment-matching optimization problem. One of the key theoretical advantages of moment-based estimation is that one can specify the estimating equations without needing to compute the full likelihood, which is especially useful when likelihood-based estimation is unavailable, misspecified, or computationally inconvenient. However, the sample moment conditions are statistical averages, which are sensitive to even a small fraction of outliers.

Such sensitivity to outliers has long been formalized in robust statistics, dating back to Huber's seminal work \cite{huber1964robust}, which introduced M-estimators for univariate location estimation. Subsequent progress includes influence functions and breakdown point \cite{hampel1968contributions}, robust testing \cite{huber1965robusttest, ronchetti1979robusttest}, and finite-sample breakdown \cite{donohohuber1983breakdown}. For a comprehensive review on robust statistics, see \cite{loh2025review}. More recently, a line of works on algorithmic robust statistics extends the classical theory by emphasizing that robust estimators should be computable without exponential search, in addition to tolerating outliers. Among the results in this direction, e.g., \cite{lai2016agnostic, diakonikolas2019robust, diakonikolas2019sever, dalalyan2022allinone, diakonikolas2023algorithmic}, a central algorithmic principle is that if an adversarial subset changes the mean substantially, then it must create a large direction in the empirical covariance matrix. 

In this paper, we apply this spectral principle to the \emph{per-observation gradients} of the GMM moment-matching optimization. This choice is intentional and motivated by two reasons inherent to GMM estimation: first, the observations and moment conditions need not encode the local parameter sensitivity of the estimating equations, whereas the gradients directly determine how contamination enters the first-order numerical optimization; second, in parameter estimation problems, the parameter space dimension $p$ (and hence the dimension of the per-observation gradients) is usually much smaller than the sample space dimension $d$, and even more so than the number of moment conditions $q = d + d^2 + \cdots + d^L$, where $L$ is the highest moment order. Thus, applying the spectral reweighting to the per-observation gradients is much less computationally intensive compared to working directly on the sample space or on the full moment condition vector. This viewpoint is conceptually consistent with previous theoretical results showing the advantages of robust gradient estimation, including \cite{prasad2020robustgradient, cheng2020high, zhu2022quasigradients}. 

We work under the strong $\epsilon$-contamination model, which is the same paradigm used in modern high-dimensional robust statistics; see, e.g., \cite[Definition 1.1]{cheng2020high}, \cite[Definition 1]{dalalyan2022allinone}, \cite[Section 3]{rohatgi2022robust}:
\begin{model}[Strong $\epsilon$-contamination model]
	\label{model:strong-contamination}
	Fix the contamination fraction $0 \leq \epsilon < 1$.  We say that a target parametric distribution $\cD_{\btheta^\star}$ is observed in the presence of \emph{\textbf{strong $\epsilon$-contamination}} if there is an adversary that inspects the full clean sample $\{\by_n\}_{n=1}^N \subset \R^d \iid \cD_{\btheta^\star}$ and replaces at most  $\epsilon N$ observations by arbitrary points in $\R^d$. The data generating process is the following:
	\begin{align}
		\check{\by}_n = \begin{cases*}
			\by_n \in \R^d, & (if $n \in \cI_{\inlier}$)\\
			\ba_n \in \R^d, & (if $n \in  \cI_{\outlier}$)
		\end{cases*},
		\quad [N] = \cI_{\inlier} \sqcup \cI_{\outlier}, \quad |\cI_{\outlier}| \leq \epsilon N.
	\end{align}
\end{model}

\subsection{Contributions}

Our contributions cover algorithmic, theoretical, and numerical aspects:

\begin{enumerate}
	\item \textbf{\cref{alg:spectral-reweighting}: Spectral gradient reweighting (SGR) primitive with regret bound analysis and convergence analysis.}
	\cref{alg:spectral-reweighting} is a soft-reweighting primitive for a given gradient cloud \(\{\check\bg_n^{(k)}\}_{n=1}^N\).  For a fixed center, \cref{thm:spectral-norm-after-mw-mmw} proves the regret bound 
	\begin{equation}
		\gamma(\widehat\bw^{[s]};\widehat\bmu^{[s]})-\OPT(\widehat\bmu^{[s]})
		\leq 4\nu\left\{\sqrt{\frac{\log(1/(1-\epsilon))}{T}}+\sqrt{\frac{\log p}{T}}\right\},
	\end{equation}
	where \(\nu\) is the squared diameter of the gradient cloud.  This is the same mirror-descent scale as the multiplicative-weights and online-convex-optimization bounds in  \cite[Theorem~5.1]{arora2012multiplicative}, \cite[Theorem~4.2]{bubeck_convex_2015}, and \cite[Chapter~2]{hazan_introduction_2023}. Under a deterministic inlier stability condition \cref{ass:inlier-stability}, \cref{thm:outer-loop-convergence} proves the convergence of the fixed-center updates
	\begin{equation}
		e^{[s+1]}
		\le
		\alpha_\epsilon e^{[s]}+R_{\epsilon,T},
		\qquad
		\alpha_\epsilon=\sqrt{\frac{\epsilon}{1-2\epsilon}},
		\label{eq:intro-contraction}
	\end{equation}
	which explains the theoretical threshold $\epsilon<1/3$. \cref{thm:sgr-termination-error} gives an explicit finite outer-loop termination bound and the corresponding robust gradient mean estimation error.
	
	\item \textbf{\cref{alg:robust-sgr-gmm}: Robust SGR-based GMM (SGR-GMM) algorithm with local finite-sample analysis.}
	\cref{alg:robust-sgr-gmm} uses the SGR primitive to ``robustify'' the per-observation moment gradients in the GMM moment-matching optimizer.  Under standard GMM local identification conditions \cref{ass:local-identification}, the high-probability inlier stability conditions \cref{ass:inlier-stability-high-prob}, and the numerical optimizer conditions \cref{ass:numerical-optimizer-conditions}, \cref{thm:finite-sample-parameter} proves
	\begin{align}
		\label{eq:intro-main-error-bound}
		\normtwo{\widehat\btheta^{(\sgrgmm)}-\btheta^\star}
		\leq
		\frac{2}{\lambda^\star}
		\left(
		\underbrace{\sum_{k=1}^L a_k\{\delta_{\mu,k}(\zeta)+\alpha_\epsilon\sqrt{C_k}\}}_{\text{SGR error}}
		+
		\underbrace{\delta_{\opt}}_{\text{optimizer error}}
		\right),
		\qquad
		\alpha_\epsilon=\sqrt{\frac{\epsilon}{1-2\epsilon}}.
	\end{align}
	At a high level, the local finite-sample parameter estimation error of \cref{alg:robust-sgr-gmm} decomposes into the robust gradient-estimation error from \cref{alg:spectral-reweighting} and the numerical-optimization residual. More explicitly, the final error depends on the contamination fraction, inlier gradient stability, local GMM identification strength, and optimization accuracy. 
	
	\item \textbf{\cref{alg:robust-dgmm}: Robust DGMM specialization for Gaussian mixture modeling. }
	\cref{alg:robust-dgmm} specializes SGR-GMM to the diagonally-weighted GMM (DGMM) estimator introduced in \cite{zhang2025diagonally} for heteroscedastic low-rank Gaussian mixtures observed under additive Gaussian noise and strong $\epsilon$-contamination. This specialization builds on the DGMM framework of \cite{zhang2025diagonally}, uses the robust SGR-weighted per-observation gradients, and updates the order-wise weights using the robust objective.
	
	\item To verify the algorithms and their theoretical analyses, we implement numerical experiments and observe that the primitive \cref{alg:spectral-reweighting} produces nearly-oracle gradient estimation and the specialization \cref{alg:robust-dgmm} substantially improves over non-robust moment baselines.
\end{enumerate}

\subsection{Related works}
\label{subsec:related}

\paragraph{Classical GMM theory.}
The local identification and monotonicity argument in \cref{sec:robust-sgr-gmm} is a finite-sample analogue of the rank and differentiability conditions used in classical GMM theory, see e.g., \cite[Sections~2-3]{hansen1982large}, \cite[Theorems~2.1 and 3.4]{newey1994large}, and \cite[Chapter 3]{hall2004generalized}. In this paper, we focus on the local deterministic part of the classical GMM theory.

\paragraph{Optimization and IRLS.}
Our convergence analysis for the fixed-center updates is conceptually motivated by iteratively reweighted least squares (IRLS). In this direction,  \cite{lermanmaunu2018fast} analyzes the fast median subspace algorithm for robust subspace recovery and  \cite{lerman2025global} proves global convergence of a dynamically smoothed IRLS variant under deterministic inlier-outlier conditions. Our proof strategy shares a similar spirit in that we also proves contraction of an interpretable geometric error
under deterministic inlier-outlier conditions
once the weights are controlled by a spectral certificate. The difference is that the geometry of interest in our setting is the covariance of moment gradients in parameter space, rather than distance to a subspace on a Grassmannian.   In our convergence analysis, we use standard optimization terminology from \cite{luenbergerlinear2008} and \cite{dennisschnabel1996numerical}. 

\paragraph{Algorithmic robust statistics.}
The core algorithmic principle in our SGR primitive is motivated by the spectral reweighting and spectral filtering algorithms for robust mean estimation: \cite{dalalyan2022allinone} gives an iteratively reweighted Gaussian-mean estimator with breakdown and asymptotic-efficiency guarantees; \cite{diakonikolas2019sever} filters per-observation gradients generated by a base learner; \cite{hopkins2020robust} and \cite{dong2019quantum} use spectral filtering and quantum entropy scoring to obtain fast robust mean estimation; \cite{rohatgi2022robust} adapts \cite{diakonikolas2019sever} to the GMM base learner. Our proposed \cref{alg:robust-sgr-gmm} differs from these approaches in two strucutral respects: first, \cref{alg:robust-sgr-gmm} uses soft-reweighting with capped-simplex weights instead of hard-filtering; second,  in \cref{alg:robust-sgr-gmm}, the SGR weights are iteratively recomputed for the per-observation gradients of the GMM objective, and as a result, the robustification is integrated into the moment-matching optimization rather than appended as a black-box filtering layer. 

Our choice of applying this spectral principle to the per-observation moment gradients is consistent with previous theoretical results for robust mean estimation, including \cite{prasad2020robustgradient} which analyzes projected gradient descent for convex risk minimization, \cite{cheng2020high} which analyzes the nonconvex robust-mean objective by gradient descent, and \cite{zhu2022quasigradients} which explains why nonconvex robust-estimation landscapes can remain algorithmically tractable through generalized quasi-gradients. 

\subsection{Organization}
\cref{sec:preliminaries} introduces notation and reviews preliminaries on GMM and entropy-regularized spectral games. \cref{sec:robust-sgr-gmm} gives the SGR-GMM algorithm. We prove regret bound and the convergence and termination of \cref{alg:spectral-reweighting}. We then give a local finite-sample GMM analysis for \cref{alg:robust-sgr-gmm}, proving a parameter estimation error bound. For clarity, proofs for \cref{sec:robust-sgr-gmm} are deferred to \cref{app-sec:robust-sgr-gmm}. \cref{sec:robust-dgmm} develops the robust DGMM specialization for estimating heteroscedastic low-rank Gaussian mixtures with additive Gaussian noise and strong $\epsilon$-contamination. \cref{sec:numerical-experiments} reports the numerical results.

\section{Preliminaries}
\label{sec:preliminaries}

\subsection{Notation}
For $N\in\N$, write $[N]=\{1,\ldots,N\}$.  For a symmetric matrix $A$, $\normop{A}$ is its largest absolute eigenvalue if $A$ is indefinite and its largest eigenvalue if $A\succeq0$.  The trace inner product is $\inner{A}{B}=\Tr(A^\top B)$.

\subsection{Method of moments (MM) and generalized method of moments (GMM)}
\label{sec:mm-gmm}
Let $\bY \sim \cD_{\bY} \in \R^d$ be a vector of random variables with the distribution $\cD_{\bY}$ parameterized by $\btheta^\star\in\Theta\subset\R^p$. For the purpose of this paper, moment function is given by
\begin{align}\label{eq:moment-function}
	g(\btheta; \bY) \coloneq &\left(\underbrace{\vectorize(\cM^{(1)}\left({\btheta}\right) - {\bY})^\top}_{g_1(\btheta; \bY)}; \cdots; \underbrace{\vectorize(\cM^{(L)}\left({\btheta}\right) - {\bY}^{\otimes L})^\top}_{g_L(\btheta; \bY)} \right)^\top \in \R^{q}, \mkern9mu g_k(\btheta;\bY)\in\mathbb R^{q_k},
\end{align}
where $q = d + d^2 + \cdots + d^L$, $L$ is the highest moment order, and $\cM^{(k)}({\btheta})$ denotes the $k$-th population moment. The corresponding population moment condition is
\begin{equation}
	m(\btheta) \coloneq \EX[g(\btheta; \bY)], \mkern9mu m_k(\btheta) \coloneq \EX[g_k(\btheta; \bY)],
\end{equation}
and the corresponding Jacobian matrix is 
\begin{equation}
	G(\btheta) \coloneq \nabla_{\btheta} m(\btheta) \in \R^{q\times p}, \mkern9mu G_k(\btheta) \coloneq \nabla_{\btheta} m_k(\btheta) \in \R^{q_k\times p}.
\end{equation}

Given a set of $\epsilon$-contaminated observations $\{\check{\by}_n\}_{n=1}^N$ from a target parametric distribution $\cD_{\btheta^\star}$, the GMM estimator of the parameter $\btheta$, denoted by $\widehat{{\btheta}}^{(\gmm)}$, is obtained by replacing the population moments by their empirical averages and minimizing a weighted quadratic discrepancy:
\begin{equation}
	\begin{aligned}\label{eq:gmm-optimization}
		\widehat{{\btheta}}^{(\gmm)} =	\argmin_{{\btheta} \in \Theta} & \mkern9mu \overline{g}_N\left({\btheta}\right)^T W \overline{g}_N\left({\btheta}\right) \eqcolon Q_N(\btheta),
	\end{aligned} 
\end{equation}
where $\overline{g}_N\left({\btheta}\right)$ is the vector of sample moment conditions
\begin{align}\label{eq:mom-conditions}
	\overline{g}_N\left({\btheta}\right) \coloneq \frac{1}{N} \sum_{n=1}^N g(\btheta; \check{\by}_n)  \in \R^{q},
\end{align}
and  $W\in \R^{q\times q}$ is a symmetric positive semi-definite weighting matrix . When $W= I$, $\widehat{{\btheta}}^{(\gmm)}$ is equivalent to the MM estimator, denoted by  $\widehat{{\btheta}}^{(\mm)}$.  

We define the following quantities:
\begin{itemize}
	\item the inlier moment gradient of the $k$-order moment-matching objective:
	\begin{align}
		\label{eq:sample-gradient-inlier}
		\bg_n^{(k)}(\btheta) \coloneq G_k (\btheta)^\top W_k g(\btheta; \by_n),
	\end{align}
	\item the population mean of the inlier moment gradients:
	\begin{align}
		\label{eq:population-gradient-mean}
		\bmu_{\bg}^{(k)}(\btheta) \coloneq \EX\left[	\bg_n^{(k)}(\btheta) \right],
	\end{align}
	\item the population covariance of the inlier moment gradients:
	\begin{align}
		\label{eq:population-gradient-cov}
		\Sigma_{\bg}^{(k)} (\btheta) \coloneq \Cov\left[		\bg^{(k)}(\btheta) \right] = \EX\left[\left(	\bg_n^{(k)}(\btheta) - \bmu_{\bg}^{(k)}(\btheta) \right)\left(	\bg_n^{(k)}(\btheta) - \bmu_{\bg}^{(k)}(\btheta) \right)^{\top}\right],
	\end{align}
	\item the $\epsilon$-contaminated moment gradients of the $k$-order moment-matching objective:
	\begin{align}
		\label{eq:sample-gradient-contaminated}
		\check{\bg}_n^{(k)}(\btheta) \coloneq G_k (\btheta)^\top W_k g(\btheta; \check{\by}_n),
	\end{align}
	\item the sample gradient covariance restricted to the index set $\cI$ as
	\begin{align}
		\label{eq:adv-gradient-cov}
		\restr{\check{S}_{\bg}^{(k)}(\btheta)}{\cI} \coloneq \frac{1}{|\cI|} \sum_{n\in \cI} \left(\check{\bg}_n^{(k)} (\btheta) - \restr{\overline{\check{\bg}^{(k)}(\btheta)}}{\cI} \right)
		\left(\check{\bg}_n^{(k)}(\btheta) - \restr{\overline{\check{\bg}^{(k)}(\btheta)}}{\cI} \right)^{\top},
	\end{align}
	where the sample mean restricted to the index set $\cI$ is
	\begin{align}
		\restr{\overline{\check{\bg}^{(k)}(\btheta)}}{\cI}  \coloneq \frac{1}{|\cI|} \sum_{n\in \cI}  \check{\bg}_n^{(k)} (\btheta).
	\end{align}
	\begin{remark}
		We use the restriction notation to highlight the index set $\cI$ over which the sample mean (resp. the sample covariance) is taken. Whenever we omit the restriction notation, we implicitly imply that the sample mean (resp. the sample covariance) is taken over all $N$ observations. To avoid notational clutter, we suppress the dependency on $\btheta$ when it is clear from the context. We only highlight the dependency on $\btheta$ when the quantities of interest vary with respect to $\btheta$.  
	\end{remark}
\end{itemize}

\subsection{von Neumann entropy and matrix multiplicative weights (MMW) algorithm}

We recall the entropy-regularized matrix multiplicative-weights (MMW) facts used by the dual spectral player.  The exposition follows  \cite{ruskai2002quantum, nielsen2010quantum, tian2025mmw}.
\begin{definition}[Density matrix]
	\label{def:density-matrix}
	The set of density matrices of dimension $p$ is defined as
	\begin{align}
		\frakD_p \coloneq \{\rho \in \R^{p\times p}: \rho \succeq 0, \Tr[\rho] = 1\}.
	\end{align}
\end{definition}

\begin{definition}[von Neumann entropy]
	\label{def:von-Neumann-entropy}
	The von Neumann entropy of a density matrix $\rho$ is defined as 
	\begin{align}
		S(\rho) \coloneq - \Tr[\rho \log \rho].
	\end{align}
\end{definition}

\begin{lemma}[Gibbs state maximizes entropy-regularized linear functional]
	\label{lemma:Gibbs-state}
	Let $H\in R^{p\times p }$ be symmetric and let $\eta > 0$. Then, the density matrix given by the Gibbs state of $H$
	\begin{align}
		\rho= \frac{\exp\left\{\eta H\right\} }{\Tr\left[\exp\left\{\eta H\right\} \right]},
	\end{align}
	is the unique maximizer of the von Neumann entropy regularized convex program: 
	\begin{align}
		\maximize_{\rho \in \cD} \quad \eta \Tr[H\rho] + S(\rho).
	\end{align}
	Equivalently, $\rho$ is the unique minimizer of
	\begin{align}
		\minimize_{\rho \in \cD} \quad -\eta \Tr[H\rho] -S(\rho).
	\end{align}
\end{lemma}
\begin{proof}
	This is a standard consequence of the nonnegativity of quantum relative entropy  (Klein's inequality) (cf.~\cite[Theorem 3]{ruskai2002quantum}, \cite[Theorem 11.7]{nielsen2010quantum}). 
\end{proof}
\begin{theorem}[MMW regret bound]
	\label{thm:mmw-prelim}
	Let \(A_1,\ldots,A_T\in\R^{p\times p}\) be symmetric matrices satisfying \(\normop{A_t}\le \nu\).  Set \(\rho^{[1]}=I_p/p\) and
	\begin{align}
		\rho^{[t]}=
		\frac{\exp\{\eta_\rho\sum_{r=1}^{t}A_r\}}
		{\Tr\exp\{\eta_\rho\sum_{r=1}^{t}A_r\}}.
	\end{align}
	If \(0<\eta_\rho\nu\le 1\), then for every \(\rho\in\frakD_p\),
	\begin{align}
		\frac1T\sum_{t=1}^T\inner{A_t}{\rho}
		-\frac1T\sum_{t=1}^T\inner{A_t}{\rho^{[t]}}
		\le
		\frac{\log p}{\eta_\rho T}+\eta_\rho\nu^2 .
	\end{align}
	In particular, the choice \(\eta_\rho\asymp \nu^{-1}\sqrt{\log(p)/T}\) gives average regret \(O(\nu\sqrt{\log(p)/T})\).
\end{theorem}
\begin{proof}
	This is the standard mirror-descent regret bound with the von Neumann entropy regularizer. The von Neumann entropy is one-strongly convex in trace norm and that its Fenchel dual is \(\log\Tr\exp(\cdot)\).  Applying mirror descent over the spectraplex gives the claimed inequality \cite[Corollary~1 and Theorem~1]{tian2025mmw}.
\end{proof}

\section{Robust GMM estimation via spectral gradient reweighting}
\label{sec:robust-sgr-gmm}
In this section, we introduce and analyze the robust GMM estimation algorithm. For clarity, we defer all proofs to \cref{app-sec:robust-sgr-gmm}. 

Suppose we have a weight distribution supported on the capped simplex (similar to the definitions in \cite{ zhu2022quasigradients}):
\begin{align}
	\bw \in \Delta_{N,\epsilon}\coloneq \left\{ \bw = (w_1, \dots, w_{N})^{\top} \in \R^{N} \colon \normone{\bw} = 1, \mkern9mu 0\leq w_n \leq \frac{1}{(1-\epsilon)N} \forall n \right\}.
\end{align}
More generally, for a weight distribution supported on the capped simplex given by a set of indices $\cI$, we use the following notation:
\begin{align}
	\bw \in \Delta_{\cI,\epsilon}\coloneq \left\{ \bw = (w_1, \dots, w_{|\cI|})^{\top} \in \R^{|\cI|} \colon \normone{\bw} = 1, \mkern9mu 0\leq w_n \leq \frac{1}{(1-\epsilon)|\cI|} \forall n \right\}.
\end{align}

For any $\bw\in\Delta_{\cI,\epsilon}$, we first define the weight mass restricted to the index set $\cI$:
\begin{align}
	\tau_{\cI}(\bw) \coloneq \sum_{n\in\cI} w_n.
\end{align}
Then we can define the weighted sample gradient covariance (of the per-observation gradients) as
\begin{align}
	\label{eq:adv-gradient-cov-weighted}
	\restr{\check{S}_{\bg,\bw}^{(k)}}{\cI} \coloneq  \sum_{n\in\cI} w_n  \left(\check{\bg}_n^{(k)} - \restr{\overline{\check{\bg}_{\bw}^{(k)}}}{\cI} \right) \left(\check{\bg}_n^{(k)} - \restr{\overline{\check{\bg}_{\bw}^{(k)}}}{\cI}\right)^{\top},
\end{align}
where the weighted sample mean is
\begin{align}
	\label{eq:adv-gradient-mean-weighted}
	\restr{\overline{\check{\bg}_{\bw}^{(k)}}}{\cI} \coloneq  \frac{1}{	\tau_{\cI}(\bw) } \sum_{n\in\cI} 
	w_n \check{\bg}_n^{(k)}.
\end{align}

A central algorithmic principle underlying previous robust mean estimation works is that if a weight distribution on the capped simplex $\bw\in \Delta_{\cI,\epsilon}$ yields small covariance spectral norm, then the weighted sample mean is close to the population mean. We apply this principle on the sample moment gradient covariance 
$\restr{\check{S}_{\bg,\bw}^{(k)}}{\cI}$ defined in \cref{eq:adv-gradient-cov-weighted} and estimate a good weight distribution $\bw \in \Delta_{\cI,\epsilon}$ so that $\normop{\restr{\check{S}_{\bg,\bw}^{(k)}}{\cI} }$ is small. Formally stated, the robust gradient estimation can be reformulated as a feasibility problem: 
\begin{align}
	\text{to find }  \bw \in \Delta_{\cI,\epsilon}, \text{ such that } \normop{\restr{\check{S}_{\bg,\bw}^{(k)}}{\cI}} = \normop{\sum_{n\in\cI} w_n  \left(\check{\bg}_n^{(k)} - \overline{\check{\bg}_{\bw}^{(k)}} \right) \left(\check{\bg}_n^{(k)} - \overline{\check{\bg}_{\bw}^{(k)}} \right)^{\top}} \leq C_{\text{stop}, k}.
\end{align}

What makes the above feasibility problem difficult is that the map $\bw \mapsto \normop{\restr{\check{S}_{\bg,\bw}^{(k)}}{\cI}}$ is non-convex. However, for a fixed center $\widehat{\bmu}$, the feasibility set is convex in $\bw$:
\begin{align}
	\left\{\bw \in \Delta_{\cI,\epsilon}:   \normop{\sum_{n\in\cI}  w_n  \left(\check{\bg}_n^{(k)} - \widehat{\bmu} \right) \left(\check{\bg}_n^{(k)} - \widehat{\bmu} \right)^{\top}} \leq C_{\text{stop}, k}\right\}.
\end{align}

This motivates the key structure behind the design and analysis of  \cref{alg:robust-sgr-gmm}, which consist of three layers: 
\begin{enumerate}
	\item The overarching idea is to apply the primitive \cref{alg:spectral-reweighting} on the set of per-observation gradients of the GMM moment-matching optimization $\{\check{\bg}_n^{(k)}\}_{n=1}^N$ to find a good weight distribution $\bw \in \Delta_{N,\epsilon}$ every few L-BFGS iterations at each GMM estimation step, so that the adversarial outliers have limited influence on the GMM estimation.
	\item Within the primitive \cref{alg:spectral-reweighting}, we first use an initial guess for the fixed center $\widehat{\bmu}$ to make the feasibility problem convex. This allows us to formulate an entropy-regularized spectral game between a sample-weight player and a density-matrix player, which can be solved using the multiplicative weights-matrix multiplicative weights (MW-MMW) update method (see, e.g., \cite{arora2012multiplicative, bubeck_convex_2015, hazan_introduction_2023}). Intuitively, the dual player $\rho$ aims to pick a direction or a mixture of directions with large projected  under the current weight vector. The primal player $\bw\in\Delta_{N,\epsilon}$ downweights the per-observation gradients that are expensive in that direction. Regret bounds guarantee that the returned average weights ${\overline{\bw}^{[s]}}$ is nearly minimax-optimal for the fixed-center game and therefore approximately minimizes the spectral norm for that given fixed center.
	\item Once we certify that for this fixed center $\widehat{\bmu}$, the MW-MMW rounds produce a good weight vector that is an approximate minimizer ${\overline{\bw}^{[s]}} \approx \OPT(\widehat{\bmu}^{[s]})$, we update the guess for the fixed center and iteratively repeat this process. In our analysis, we provide a fixed-point convergence guarantee and a finite termination guarantee for the fixed-center updates.
\end{enumerate}

\begin{algorithm}[H]
	\renewcommand{\thealgocf}{SGR-GMM}
	\caption{Robust GMM estimation via spectral gradient reweighting.}	
	\label{alg:robust-sgr-gmm}
	\SetKwInput{KwInput}{Input} 
	\SetKwInput{KwOutput}{Output}
	\KwInput{\begin{itemize}
			\item $\epsilon$-contaminated observations $\{\check{\by}_n\}_{n=1}^N$,
			\item hyperparameters: the maximum moment order $L$, the maximum DGMM steps $T_{\gmm}$, the maximum L-BFGS iterations $I_{\lbfgs}$, contamination fraction $\epsilon \in (0, 1/3)$, the MW-MMW step sizes $0<\eta_{\rho}, \eta_{w} \leq 1/2$, the inner iterations $T$, threshold constant $C>0$, target accuracy $\delta > 0$, reweighting interval $I_{\interval}$.
	\end{itemize}}
	\KwOutput{estimated parameters  $\widehat{\btheta} \in \Theta \subset \R^{p}.$}
	\DontPrintSemicolon
	
	Initialize $\btheta^{[0]}$. \;
	
	\For{$t=1,\dots,T_{\gmm}$ or until GMM steps converge}{
		Run moment-matching optimization via L-BFGS:
		\For{$i = 1, \dots, I_{\lbfgs}$ or until L-BFGS iterations converge}{
			For each moment order $k = 1, \dots, L$, evaluate the $\epsilon$-contaminated moment gradients of the GMM moment-matching optimization $\check{\bg}_n^{(k)}$.\;
			
			\If{$i - i_{\prev} \geq I_{\interval}$ or $i - i_{\prev} \geq I_{\min}$ and L-BFGS is locally stabilized \label{alg:robust-sgr-gmm-line-lbfgs}}{
				Update the weight vector on the per-observation gradients $\widehat{\bw}^{(k)}$ $\in \Delta_{N,\epsilon}$ for each moment order $k$ via  \cref{alg:spectral-reweighting}.  \;
				
				Reset L-BFGS memory and continue.\;
			}
			Freezing $\widehat{\bw}^{(k)}$ and continue the L-BFGS iterations using the robust objective function and robust gradient for the moment-matching optimization.
		}
		Use $\widehat{\btheta}^{[t]}$ to initialize the next $(t+1)$-th GMM estimation step.\;
	}
\end{algorithm}

\begin{algorithm}[H]
	\renewcommand{\thealgocf}{SGR}
	\caption{Spectral gradient reweighting.}\label{alg:spectral-reweighting}
	\SetKwInput{KwInput}{Input} 
	\SetKwInput{KwOutput}{Output}
	\KwInput{\begin{itemize}
			\item per-observation gradients of moment order $k$, $\{\check{\bg}^{(k)}_n\}_{n=1}^N \subset \R^p$,
			\item hyperparameters: contamination fraction $\epsilon \in (0, 1/3)$, the MW-MMW step sizes $0<\eta_{\rho}, \eta_{w} \leq 1/2$, the inner iterations $T$, threshold constant $C>0$, target accuracy $\delta > 0$.
	\end{itemize}}
	\KwOutput{estimated MW weights $\widehat{\bw}^{(k)}$.}
	\DontPrintSemicolon
	
	Initialize the MW weights $\widehat{\bw}^{[1]} \gets (1/N, \dots, 1/N)$.\label{alg:spectral-reweighting-line-init} \; 
	
	Initialize the fixed center to be the geometric median: $\widehat{\bmu}^{[1]} \gets \widehat{\bmu}^{\text{GeomMed}}\in \argmin_{\bmu\in\R^p} \sum_{n=1}^N \normtwo{\check{\bg}^{(k)}_n - \bmu}.$
	
	\For{$s = 1, 2, \dots, s_{\max}$}{		
		${\bz_n^{[s]}} \gets \check{\bg}^{(k)}_n- \widehat{\bmu}^{[s]} \in \R^{p}$.\;
		
		Restart: either uniform restart $\doublewidehat{\bw}^{[s, 1]} \gets (1/N, \dots, 1/N)$ or warm-start $\doublewidehat{\bw}^{[s, 1]} \gets \widehat{\bw}^{[s]}$.\label{alg:spectral-reweighting-line-restart}\;
		
		\For{$t =  1,  \dots, T$}{
			
			\textbf{Dual MMW update:}
			\begin{substeps}
				Compute the MMW gain matrix: 
				$S^{[s, t]} \gets  \sum_{n=1}^N \doublewidehat{w}_n^{[s, t]}	{\bz_n^{[s]}} {\bz_n^{[s]}}^{\top} $. \;
				\label{alg-step:dual-update-density-matrix} Update dual density matrix: $\rho^{[s, t]} \gets  \frac{\exp\left\{\eta_{\rho} \sum_{t^\prime=1}^{t} S^{[s, t^\prime]} \right\} }{\Tr\left[\exp\left\{\eta_{\rho} \sum_{t^\prime=1}^{t} S^{[s, t^\prime]} \right\} \right]}$.\;
			\end{substeps}
			
			\textbf{Primal MW update:} 	
			\begin{substeps}
				Compute the MW loss: $m_n^{[s, t]} \gets  {\bz_n^{[s]}}^{\top} \rho^{[s, t]}{\bz_n^{[s]}}$.\;
				
				Update primal weights: $\widetilde{w}_n^{[s, t]} \gets \doublewidehat{w}_n^{[s, t]} (1 -\eta_{w} m_n^{[s, t]}), \mkern9mu \doublewidehat{\bw}^{[s, t+1]} \gets \Pi_{\Delta_{N, \epsilon}}^{\RE} \widetilde{\bw}^{[s, t]} $. \;
			\end{substeps}
			
		}
		$\overline{\bw}^{[s]} \gets \frac{1}{T}\sum_{t=1}^T\doublewidehat{\bw}^{[s, t]}, \mkern9mu {\overline{S}^{[s]}} \gets  \frac{1}{T}\sum_{t=1}^T S^{[s, t]}$.\;
		
		\If{$\normop{\overline{S}^{[s]}}  \leq C_{\text{stop}, k}$}{Output $\widehat{\bw}^{(k)} \gets \overline{\bw}^{[s]} $ and terminate.}
		\Else{
			Update weights: $\widehat{\bw}^{[s+1]}  \gets \overline{\bw}^{[s]} $.\;
			Update the fixed center: $\widehat{\bmu}^{[s+1]} \gets   \overline{\check{\bg}_{\widehat{\bw}^{[s+1]}}^{(k)}} =  \sum_{n=1}^N \widehat{w}_n^{[s+1]} \check{\bg}^{(k)}_n$\;
		}	
	}
\end{algorithm}

\begin{remark}
	During each interval between reweighting steps, $\widehat{\bw}^{(k)}$ and $\widehat{o}_k^{[t]}$ are frozen. Whenever they are updated, the L-BFGS memory is reset.  In \cref{alg:robust-sgr-gmm-line-lbfgs} of \cref{alg:robust-sgr-gmm}, besides using $I_{\interval}$ the fixed hyperparameter reweighting interval to decide when to update the sample weight vector and the robust order-specific DGMM weights, we can optionally use the Dennis-Schnabel's scaled-gradient test \cite[Appendix A]{dennisschnabel1996numerical} (gradient times variable scale, normalized by function scale) to get a condition to check that the L-BFGS is locally stabilized, namely:
	\begin{align*}
		\zeta_{\grad} &\leq 10 (\tol)^{1/3}, \text{ and } \zeta_{\param} \leq 10 (\tol)^{1/3}, \mkern9mu \tol \approx 10^{-6}, \\
		\zeta_{\grad} \coloneq &\frac{1}{\max\{1, |Q_N(\btheta^{[i]})|\}}\\
		&  \max\Big\{\norminfty{\nabla_{\bpi}Q_N(\btheta^{[i]})}, \max_{1\leq j\leq K} \left( \max\left\{ 1, \normtwo{\bmu_j}\right\} \normtwo{\nabla_{\bmu_j^{[i]}} Q_N(\btheta^{[i]})}\right), \\
		& \quad \max_{1\leq j\leq K} \left( \max\left\{ 1, \norm{V_j^{[i]}}_{\mathrm{F}}\right\} \norm{\nabla_{V_j} Q_N^{[i]}(\btheta)}_{\mathrm{F}}  \right) \Big\}, \\
		\zeta_{\param} &\coloneq \max\left\{\normone{\bpi^{[i]} - \bpi^{[i-1]} }, \mkern9mu \max_{1\leq j\leq K} \frac{\normtwo{\bmu_j^{[i]} - \bmu_j^{[i-1]} }}{\max\left\{ 1, \normtwo{\bmu_j^{[i]}}\right\}}, \mkern9mu  \max_{1\leq j\leq K} \frac{\norm{ \Sigma_j^{[i]} - \Sigma_j^{[i-1]}}_{\mathrm{F}}}{\max\left\{ 1, \norm{ \Sigma_j^{[i]}}_{\mathrm{F}} \right\}}  \right\}
	\end{align*}
\end{remark}
\begin{remark}
	For generality, we use uniform initialization in \cref{alg:spectral-reweighting-line-init} and uniform restart in \cref{alg:spectral-reweighting-line-restart} in our analysis of \cref{alg:spectral-reweighting}. However, we note that warm start is a useful practical heuristic that can often improve numerical optimization performance. 
\end{remark}

\subsection{Fixed-center regret bound}
\label{sec:mw-mmw-rounds}

\textEnd[both, category=mw-mmw-rounds]{In what follows, we will prove that given a fixed center $\widehat{\bmu}^{[s]}$, the averaged weight vector returned from the MW-MMW rounds, $\widehat{\bw}^{[s]} = \overline{\bw}^{[s]} = \frac{1}{T} \sum_{t = 1}^{T} \doublewidehat{\bw}^{[s, t]}$, approximately minimizes the spectral norm objective for this fixed center.} For a fixed center $\widehat{\bmu} \in \R^p$, define:
\begin{definition}[MW loss as the contamination score]
	\label{def:contamination-score}
	\begin{align}
		m\left(\rho; \check{\bg}_n^{(k)}, \widehat{\bmu}\right)  \coloneq \left(\check{\bg}_n^{(k)} - \widehat{\bmu}\right)^{\top}\rho \left(\check{\bg}_n^{(k)} - \widehat{\bmu}\right).
	\end{align}
\end{definition}

\begin{definition}[MMW gain matrix as the fixed-center covariance]
	\label{def:adv-gradient-cov-weighted-fixed}
	\begin{align}
		S\left(\bw; \{\check{\bg}_n^{(k)}\}_{n\in [N]}, \widehat{\bmu}\right) \coloneq   \sum_{n=1}^N
		w_n	\left(\check{\bg}_n^{(k)} - \widehat{\bmu}\right) \left(\check{\bg}_n^{(k)} - \widehat{\bmu}\right)^{\top}.
	\end{align}
\end{definition}

\begin{definition}[Spectral norm potential]
	\label{def:spectral-norm-potential}
	\begin{align}
		\gamma\left(\bw; \{\check{\bg}_n^{(k)}\}_{n\in [N]}, \widehat{\bmu}\right) \coloneq \normop{S\left(\bw; \{\check{\bg}_n^{(k)}\}_{n\in [N]}, \widehat{\bmu}\right)}.
	\end{align}
\end{definition}

\begin{definition}[Minimax optimizer]
	\begin{align}
		\OPT(\widehat{\bmu}) = \min_{\bw \in \Delta_{N,\epsilon}} \max_{\rho \in \frakD_p}
		\inner{S(\bw; \{\check{\bg}_{n}^{(k)}\}_{n\in [N]}, \widehat{\bmu}) }{\rho}. 
	\end{align}
\end{definition}

\begin{theoremEnd}[end, category=mw-mmw-rounds]{lemma}[Spectral norm objective as a convex-concave game]
	For all fixed center $\widehat{\bmu}$ and for all $\bw \in \Delta_{N, \epsilon}$, 
	\begin{align}
		\label{lemma:spectral-norm-convex-concave-game-connection}
		\gamma\left(\bw; \{\check{\bg}_n^{(k)}\}_{n\in [N]}, \widehat{\bmu}\right) & = \normop{S\left(\bw; \{\check{\bg}_n^{(k)}\}_{n\in [N]}, \widehat{\bmu}\right)} = \max_{\rho \in \frakD_p}\inner{S\left(\bw; \{\check{\bg}_n^{(k)}\}_{n\in [N]}, \widehat{\bmu}\right)}{\rho}
	\end{align}
	and 
	\begin{align}
		\label{lemma:primal-dual-connection}
		\inner{S\left(\bw; \{\check{\bg}_n^{(k)}\}_{n\in [N]}, \widehat{\bmu}\right)}{\rho}	&= \sum_{n=1}^N w_n m\left(\rho; \check{\bg}_n^{(k)}, \widehat{\bmu}\right).
	\end{align}
	
	As a result, 
	\begin{align}
		\min_{\bw \in \Delta_{N,\epsilon}}	\gamma\left(\bw; \{\check{\bg}_n^{(k)}\}_{n\in [N]}, \widehat{\bmu}\right) & = \min_{\bw \in \Delta_{N,\epsilon}} \normop{S\left(\bw; \{\check{\bg}_n^{(k)}\}_{n\in [N]}, \widehat{\bmu}\right)} \\
		&=\min_{\bw \in \Delta_{N,\epsilon}} \max_{\rho \in \frakD_p}\inner{S\left(\bw; \{\check{\bg}_n^{(k)}\}_{n\in [N]}, \widehat{\bmu}\right)}{\rho}\\
		&= \min_{\bw \in \Delta_{N,\epsilon}} \max_{\rho \in \frakD_p}\sum_{n=1}^N w_n m\left(\rho; \check{\bg}_n^{(k)}, \widehat{\bmu}\right).
	\end{align}
\end{theoremEnd}

\begin{theoremEnd}[end, category=mw-mmw-rounds]{lemma}[Normalizing scale]
	\label{lemma:normalizing-scale}
	Define the normalizing scale 
	
	\begin{align}
		\nu \coloneq \diam\left( \{\check{\bg}^{(k)}_n\}_{n\in[N]} \right)^2 = \max_{i,j \in [N]} \normtwo{\check{\bg}^{(k)}_i - \check{\bg}^{(k)}_j}^2.
	\end{align}
	
	Then the following holds:
	\begin{enumerate}
		\item For all $\rho\in \frakD_p$, $m\left(\rho; \{\check{\bg}_n^{(k)}\}_{n\in [N]}, \widehat{\bmu}^{[s]}\right)  \in [0,\nu]$. In particular, each MW loss $m_n^{[s, t]}$ satisfies
		\begin{align}
			m_n^{[s, t]} = m\left(\rho^{[s,t]}; \{\check{\bg}_n^{(k)}\}_{n\in [N]}, \widehat{\bmu}^{[s]}\right)   \in [0,\nu].
		\end{align}
		
		\item For all $\bw\in\Delta_{N,\epsilon}$, $	0\preceq S\left(\bw; \{\check{\bg}_n^{(k)}\}_{n\in [N]}, \widehat{\bmu}^{[s]}\right) \preceq \nu I_p$. In particular,  each MMW gain matrix $M^{[s,t]}$ satisfies
		\begin{align}
			0\preceq S^{[s,t]} = S(\doublewidehat{\bw}^{[s,t]}; \{\check{\bg}_n^{(k)}\}_{n\in [N]}, \widehat{\bmu}^{[s]}) \preceq \nu I_p.
		\end{align}
	\end{enumerate}
\end{theoremEnd}

\begin{proofEnd}
	
	Since  the initialization uses geometric median, which is in the convex hull of $\left( \{\check{\bg}_n\} \right)$ and the weighted mean fixed-center updates preserve the convex hull, we have $\widehat{\bmu}^{[s]} \in \conv\left( \{\check{\bg}_n\} \right)$ for all $s$, where $\conv$ denotes the convex hull. Thus, by choosing $\nu \coloneq \diam\left( \{\check{\bg}_n\} \right)^2 = \max_{i,j} \normtwo{\check{\bg}_i - \check{\bg}_j}^2$, we get $\normtwo{ \left(\check{\bg}_n- \widehat{\bmu}^{[s]}\right)}^2 \leq \normtwo{\check{\bg}_n- \widehat{\bmu}^{[s]}}^2 \leq \nu$.
	
	\begin{enumerate}
		\item For all density matrix $\rho\in \frakD_p$, it is PSD with $\Tr[\rho^{[s, t]}] = 1$ and all its eigenvalues are $\leq 1$. Therefore, 
		\begin{align*}
			0 \leq m\left(\rho; \{\check{\bg}_n^{(k)}\}_{n\in [N]}, \widehat{\bmu}^{[s]}\right) =  \left(\check{\bg}_n- \widehat{\bmu}^{[s]} \right)^{\top} \rho  \left(\check{\bg}_n- \widehat{\bmu}^{[s]} \right) \leq  \normtwo{ \left(\check{\bg}_n- \widehat{\bmu}^{[s]}\right)}^2 \leq \nu. 
		\end{align*}
		
		\item For each rank-one matrix $\left(\check{\bg}_n- \widehat{\bmu}^{[s]} \right)
		\left(\check{\bg}_n- \widehat{\bmu}^{[s]} \right)^{\top}$, we have 
		\begin{align}
			\normtwo{\left(\check{\bg}_n- \widehat{\bmu}^{[s]}\right)
				\left(\check{\bg}_n- \widehat{\bmu}^{[s]} \right)^{\top}} = \normtwo{\check{\bg}_n- \widehat{\bmu}^{[s]}}^2 \leq \nu, 
		\end{align}
		which then implies that for all $\bw\in\Delta_{N,\epsilon}$, 
		\begin{align}
			& 0 \preceq S(\bw; \{\check{\bg}_n^{(k)}\}_{n\in [N]}, \widehat{\bmu}^{[s]}) = \sum_{n\in [N]}
			w_n	\left(\check{\bg}_n- \widehat{\bmu}^{[s]}\right)
			\left(\check{\bg}_n- \widehat{\bmu}^{[s]} \right)^{\top} \preceq \nu I_p.
		\end{align}
	\end{enumerate}
\end{proofEnd}

\begin{theoremEnd}[end, category=mw-mmw-rounds]{theorem}[Primal MW regret bound]
	\label{thm:regret-bound-primal}
	Suppose that $0<\eta_{w} \nu \leq 1/2$. Then we get the following regret bound: after $T$ MW rounds, for all weight vector $\bw \in \Delta_{N,\epsilon}$,
	\begin{align}
		\sum_{t=1}^{T} \inner{\bm{m}^{[s, t]}}{\doublewidehat{\bw}^{[s, t]}} \leq (1+ \eta_{w} \nu )\sum_{t=1}^{T} \inner{\bm{m}^{[s, t]}}{\bw} + \frac{\RE(\bw || \widetilde{\bw}^{[s, 1]})}{\eta_w}. 
	\end{align}
\end{theoremEnd}
\begin{proofEnd}
	By \cref{lemma:normalizing-scale}, every coordinate of the loss vector satisfies
	\(0\le m_n^{[s,t]}\le \nu\).  Set
	\(\ell_n^{[s,t]}=m_n^{[s,t]}/\nu\in[0,1]\).  After
	rescaling by \(\nu\), we can apply the the classical MMW regret bound \cite[Theorem 5.1]{arora2012multiplicative} in the 	minimization (loss) form to get
	\[
	\sum_{t=1}^{T}\langle \bm m^{[s,t]},\doublewidehat\bw^{[s,t]}\rangle
	\le (1+\eta_w\nu)\sum_{t=1}^{T}\langle \bm m^{[s,t]},\bw\rangle
	+\frac{\RE(\bw\Vert \widetilde\bw^{[s,1]})}{\eta_w}.
	\]
\end{proofEnd}

\begin{remark}
	Note that in \cref{alg:spectral-reweighting}, the projection onto the capped simplex at each round \(\Delta_{N,\epsilon}\) \[
	\doublewidehat{\bw}^{[t+1]} \gets \argmin_{\bw \in \Delta_{N,\epsilon}} \RE (\bw ||\widetilde{\bw}^{[s, t]})\]
	is the standard mirror descent with Bregman projection onto a convex feasible set. Therefore, the classical regret bound still holds for all \(w\in\Delta_{N,\epsilon}\) even with the projection step, see e.g., \cite[Theorem 4.2]{bubeck_convex_2015}, \cite[Theorem 2.1]{arora2012multiplicative}.
\end{remark}

\begin{theoremEnd}[end, category=mw-mmw-rounds]{theorem}[Dual MMW regret bound]
	\label{thm:regret-bound-dual}
	Suppose that $0<\eta_{\rho} \nu \leq 1/2$. Let $S^{[s, t]} \coloneq S(\doublewidehat{\bw}^{[s, t]}; \{\check{\bg}_n^{(k)}\}_{n\in [N]}, \widehat{\bmu}^{[s]})$ be the MMW gain matrix at the $t$-round and $\rho^{[s, t]}$ be as defined in \cref{alg:spectral-reweighting}. Then, we get the following regret bound: after $T$ MMW rounds, for all density matrix $\rho \in \frakD_p$,
	\begin{align}
		\label{eq:regret-bound-dual-simplified}
		\sum_{t=1}^{T} \inner{S^{[s, t]}}{\rho} \leq 
		(1+  \eta_{\rho} \nu)\sum_{t=1}^{T} \inner{S^{[s, t]} } {\rho^{[s, t]}}  + \frac{\log p}{\eta_{\rho}}.
	\end{align}
\end{theoremEnd}

\begin{proofEnd}
	
	By \cref{lemma:normalizing-scale}, each gain matrix satisfies
	\(0\preceq S^{[s,t]}\preceq \nu I_p\).  The density-matrix update in
	\cref{alg:spectral-reweighting} is the Gibbs state associated with the cumulative
	gain matrix, as recorded in \cref{lemma:Gibbs-state}.  Hence the matrix
	multiplicative-weights regret inequality applies on the spectraplex
	\(\mathfrak D_p\) \cite[Theorem~5.1]{arora2012multiplicative}. By applying the classical MMW regret bound \cite[Theorem 5.1]{arora2012multiplicative} in the maximization (gain) form, we get the following regret bound: after $T$ rounds of MMW rounds, for all density matrix $\rho \in \frakD_p$,
	\begin{align}
		\label{eq:regret-bound-dual}
		\sum_{t=1}^{T} \inner{S^{[s, t]}} {\rho^{[s, t]}}  \geq \sum_{t=1}^{T} \inner{S^{[s, t]} }{\rho} - \eta_{\rho} \sum_{t=1}^{T} \inner{\left(S^{[s, t]}\right)^2}{\rho^{[s, t]}}  - \frac{\log p}{\eta_{\rho}}.
	\end{align}
	
	Since $0 \preceq S^{[s, t]}  \preceq \nu I_p$ and so all eigenvalues of $S^{[s, t]} \in [0,\nu]$, we have $\left(S^{[s, t]}\right)^2 \preceq \nu S^{[s, t]}$. Combined with the fact that $\rho^{[s, t]}\succeq 0$, we get
	\begin{align}
		\inner{\left(S^{[s, t]}\right)^2}{\rho^{[s, t]}}  \leq \nu \inner{S^{[s, t]}}{\rho^{[s, t]}}.
	\end{align}
	Substituting this into \cref{eq:regret-bound-dual}, we get 
	\begin{align*}
		\sum_{t=1}^{T} \inner{S^{[s, t]}} {\rho^{[s, t]}}  \geq \sum_{t=1}^{T} \inner{S^{[s, t]} }{\rho} - \eta_{\rho} \nu \sum_{t=1}^{T} \inner{S^{[s, t]}}{\rho^{[s, t]}}  - \frac{\log p}{\eta_{\rho}}.
	\end{align*}
	Rearranging the terms, we get
	\begin{align*}
		\sum_{t=1}^{T} \inner{S^{[s, t]}}{\rho} \leq 
		(1+  \eta_{\rho} \nu)\sum_{t=1}^{T} \inner{S^{[s, t]} } {\rho^{[s, t]}}  + \frac{\log p}{\eta_{\rho}}.
	\end{align*}
\end{proofEnd}
\begin{remark}
	The above proposition \cref{thm:regret-bound-dual} justifies the density matrix update \cref{alg:spectral-reweighting} (\cref{alg-step:dual-update-density-matrix}): it produces a sequence $\rho^{[s, t]}$ whose cumulative gain is competitive with the best fixed $\rho$, up to $\log R$ regret. 
\end{remark}

\begin{theoremEnd}[end, category=mw-mmw-rounds]{theorem}[Overall regret bound]
	\label{thm:overall-regret-bound}
	Fix $0<\eta_{\rho}\nu, \eta_{w}\nu \leq 1/2$. After $T$ MW-MMW rounds, for all $\bw\in\Delta_{N,\epsilon}$ and any density matrix $\rho \in \frakD_p$,
	\begin{align}
		\sum_{t=1}^{T} \inner{S^{[s, t]}}{\rho} \leq 
		(1+  \eta_{\rho}\nu) (1+ \eta_{w}\nu) \sum_{t=1}^{T} \inner{\bm{m}^{[s, t]}}{\bw} +  \frac{(1+  \eta_{\rho}\nu)\RE(\bw || \widetilde{\bw}^{[s, 1]})}{\eta_w}  +  \frac{\log p}{\eta_{\rho}}.
	\end{align}
	In particular, for uniform initialization $\widetilde{\bw}^{[s, 1]} = (1/N, \dots, 1/N)$, we have 
	\begin{align}
		\sum_{t=1}^{T} \inner{S^{[s, t]}}{\rho} \leq (1+  \eta_{\rho}\nu)  (1+ \eta_{w}\nu)\sum_{t=1}^{T} \inner{\bm{m}^{[s, t]}}{\bw} +  \frac{(1+  \eta_{\rho}\nu)\log\left(\frac{1}{1-\epsilon}\right)}{\eta_w} + \frac{\log p}{\eta_{\rho}}.
	\end{align}
\end{theoremEnd}
\begin{proofEnd}
	By \cref{lemma:primal-dual-connection}, the primal MW expected loss is equal to the dual MMW payoff for each $t = 1, \dots, T$: 
	\begin{align}
		\inner{\bm{m}^{[s, t]}}{\doublewidehat{\bw}^{[s, t]}}  =  \inner{S^{[s, t]} } {\rho^{[s, t]}} = \sum_{n\in [N]}  \doublewidehat{w}_n^{[s, t]} \bz_n^{\top} \rho \bz_n.
	\end{align}
	Substituting this into the simplified dual MMW regret bound \cref{eq:regret-bound-dual-simplified} and applying the primal regret bound \cref{thm:regret-bound-primal}, we get for all $\bw\in\Delta_{N,\epsilon}$ and any density matrix $\rho \in \frakD_p$, 
	\begin{align*}
		\sum_{t=1}^{T} \inner{S^{[s, t]}}{\rho} &\leq 
		(1+  \eta_{\rho}\nu)\sum_{t=1}^{T} \inner{S^{[s, t]} } {\rho^{[s, t]}}  + \frac{\log p}{\eta_{\rho}}\\
		&= (1+  \eta_{\rho}\nu)\sum_{t=1}^{T} 	\inner{\bm{m}^{[s, t]}}{\doublewidehat{\bw}^{[s, t]}} + \frac{\log p}{\eta_{\rho}}\\
		&\leq  (1+  \eta_{\rho}\nu)  (1+ \eta_{w}\nu)\sum_{t=1}^{T} \inner{\bm{m}^{[s, t]}}{\bw} + (1+  \eta_{\rho}\nu) \frac{\RE(\bw || \widetilde{\bw}^{[s, 1]})}{\eta_w} + \frac{\log p}{\eta_{\rho}}.\label{step:overall-regret-bound-with-proj}\numberthis
	\end{align*}
	
	For the initialization $\widetilde{\bw}^{[s, 1]} = (1/N, \dots, 1/N)$ as per \cref{alg:spectral-reweighting} and any $\bw\in \Delta_{N, \epsilon}$, we have
	\begin{align*}
		\RE(\bw|| \widetilde{w}^{[s, 1]}) &= \sum_{n=1}^N w_n \log( |\cI^{[s]}| w_n)\\
		& \leq \sum_{n=1}^N w_n \log\left(\frac{1}{1-\epsilon}\right)= \log\left(\frac{1}{1-\epsilon}\right).\numberthis
	\end{align*}
	
	Substituting this back into \cref{step:overall-regret-bound-with-proj}, we get 
	\begin{align*}
		\sum_{t=1}^{T} \inner{S^{[s, t]}}{\rho} \leq (1+  \eta_{\rho}\nu)  (1+ \eta_{w}\nu)\sum_{t=1}^{T} \inner{\bm{m}^{[s, t]}}{\bw} +  \frac{(1+  \eta_{\rho}\nu)\log\left(\frac{1}{1-\epsilon}\right)}{\eta_w} + \frac{\log p}{\eta_{\rho}}.
	\end{align*}
\end{proofEnd}

\begin{theoremEnd}[end, category=mw-mmw-rounds]{theorem}[Spectral norm bound]
	\label{thm:spectral-norm-after-mw-mmw}
	Let the output returned by the $T$ MW-MMW rounds in \cref{alg:spectral-reweighting} and and the corresponding average cost matrix be
	\begin{align}
		{\overline{\bw}^{[s]}} \coloneq \frac{1}{T} \sum_{t=1}^{T} \doublewidehat{\bw}^{[s, t]}, \mkern9mu 	{\overline{S}^{[s]}} \coloneq \frac{1}{T} \sum_{t=1}^{T} S^{[s, t]}.
	\end{align}
	Then 
	\begin{align}
		{\overline{S}^{[s]}} = S\left({\overline{\bw}^{[s]}}; \{\check{\bg}_n^{(k)}\}_{n\in [N]}, \widehat{\bmu}^{[s]}\right),
	\end{align} 
	and
	\begin{align}
		\normop{{\overline{S}^{[s]}}} \leq 	(1+  \eta_{\rho}\nu) (1+ \eta_{w}\nu) \OPT(\widehat{\bmu}^{[s]}) + \frac{(1+  \eta_{\rho}\nu)\log\left(\frac{1}{1-\epsilon}\right)}{T \eta_w}   +  \frac{\log p}{T\eta_{\rho}},
	\end{align}
	where $	\OPT(\widehat{\bmu}^{[s]})$ is the minimax optimum for the fixed center $\widehat{\bmu}^{[s]}$:
	\begin{align}
		\OPT(\widehat{\bmu}^{[s]}) = \min_{\bw \in \Delta_{N,\epsilon}} \max_{\rho\in \frakD_p}
		\inner{S(\bw; \{\check{\bg}_n^{(k)}\}_{n\in [N]}, \widehat{\bmu}^{[s]}) }{\rho}.
	\end{align}
	
	Consequently, choosing the step sizes and the number of MW-MMW rounds to be 
	\begin{align}
		\eta_{w} =  \frac{1}{\nu} \sqrt{\frac{\log\left(\frac{1}{1-\epsilon}\right)}{ T}}, \mkern9mu \eta_{\rho} = \frac{1}{\nu} \sqrt{\frac{\log p}{T}}, \mkern9mu T \geq 4 \max \left\{ \log\left(\frac{1}{1-\epsilon}\right), \log p\right\},
	\end{align} 
	the output ${\widehat{\bw}^{[s]}} =  {\overline{\bw}^{[s]}}$ is an $ \cO\left( \nu\sqrt{\frac{\log\left(\frac{1}{1-\epsilon}\right)}{T}} + \nu\sqrt{\frac{\log p}{T}}\right)$-approximation to $	\OPT(\widehat{\bmu}^{[s]})$, specifically,
	\begin{align}
		\gamma\left(\widehat{\bw}^{[s]}; \{\check{\bg}_n^{(k)}\}_{n\in [N]}, \widehat{\bmu}^{[s]}\right)  - \OPT(\widehat{\bmu}^{[s]}) \leq 4 \nu \left( \sqrt{\frac{\log\left(\frac{1}{1-\epsilon}\right)}{T}} + \sqrt{\frac{\log p}{T}}\right) \eqcolon \delta_{T}(\widehat{\bmu}^{[s]}).
	\end{align}
\end{theoremEnd}

\begin{proofEnd}
	For ${\overline{\bw}^{[s]}} \coloneq \frac{1}{T} \sum_{t=1}^{T} \doublewidehat{\bw}^{[s, t]}$, we have ${\overline{\bw}^{[s]}}
	\in \Delta_{\cI^{[s]},\epsilon}$ since convexity preserves feasibility and 
	\begin{align}
		{\overline{S}^{[s]}} &\coloneq \frac{1}{T} \sum_{t=1}^{T} S^{[s, t]} = \frac{1}{T} \sum_{t=1}^{T} S\left(\doublewidehat{\bw}^{[s, t]}; \{\check{\bg}_n^{(k)}\}_{n\in [N]}, \widehat{\bmu}^{[s]}\right) \\
		&= \frac{1}{T} \sum_{t=1}^{T} \left( \sum_{n=1}^N \doublewidehat{w}^{[s, t]}_n {\bz_n^{[s]}} {\bz_n^{[s]}}^{\top} \right)\\
		&= \sum_{n=1}^N \left(\frac{1}{T} \sum_{t=1}^{T} \doublewidehat{w}^{[s, t]}_n \right) {\bz_n^{[s]}} {\bz_n^{[s]}}^{\top}\\
		&=  \sum_{n=1}^N \overline{w}_n {\bz_n^{[s]}} {\bz_n^{[s]}}^{\top} = S\left({\overline{\bw}^{[s]}}; \{\check{\bg}_n^{(k)}\}_{n\in [N]}, \widehat{\bmu}^{[s]}\right)
	\end{align}
	
	By applying \cref{thm:overall-regret-bound}, we get
	\begin{align}
		\normop{{\overline{S}^{[s]}}} &= \max_{\rho \in \frakD_p} \inner{{\overline{S}^{[s]}}}{\rho}=  \max_{\rho \in \frakD_p} \inner{\frac{1}{T} \sum_{t=1}^{T} S^{[s, t]} }{\rho}= \max_{\rho \in \frakD_p} \frac{1}{T}  \sum_{t=1}^{T} \inner{S^{[s, t]}}{\rho}\\
		&\leq 	\frac{(1+  \eta_{\rho}\nu) (1+ \eta_{w}\nu)}{T}  \sum_{t=1}^{T} \inner{\bm{m}^{[s, t]}}{\bw} + \frac{(1+  \eta_{\rho}\nu)\log\left(\frac{1}{1-\epsilon}\right)}{T \eta_w}  +  \frac{\log p}{T\eta_{\rho}}.
	\end{align}
	
	Choose $\bw = \bw^{*}$ where 
	\begin{align}
		\bw^{*} \in \argmin_{\bw \in \Delta_{N,\epsilon}} \normop{S\left({\bw}; \{\check{\bg}_n^{(k)}\}_{n\in [N]}, \widehat{\bmu}^{[s]}\right)}.
	\end{align}
	Then for each $t = 1, \dots, T$, we have 
	\begin{align}
		\frac{1}{T}\sum_{t=1}^{T}	\inner{\bm{m}^{[s, t]}}{\bw^\star} &= \frac{1}{T}\sum_{t=1}^{T}\sum_{n=1}^N w_n^\star  {\bz_n^{[s]}}^{\top} \rho^{[s, t]} {\bz_n^{[s]}} \\
		&=\frac{1}{T}\sum_{t=1}^{T} \inner{ \sum_{n=1}^N w_n^\star {\bz_n^{[s]}} {\bz_n^{[s]}}^{\top}}{\rho^{[s, t]}}\\
		&= \frac{1}{T}\sum_{t=1}^{T}\inner{S\left({\bw^\star}; \{\check{\bg}_n^{(k)}\}_{n\in [N]}, \widehat{\bmu}^{[s]}\right)}{ \rho^{[s, t]}}\\
		& \leq \frac{1}{T}\sum_{t=1}^{T} \normop{S\left({\bw^\star}; \{\check{\bg}_n^{(k)}\}_{n\in [N]}, \widehat{\bmu}^{[s]}\right)} = \OPT(\widehat{\bmu}^{[s]}).
	\end{align}
	
	Since $\gamma\left(\widehat{\bw}^{[s]}; \{\check{\bg}_n^{(k)}\}_{n\in [N]}, \widehat{\bmu}^{[s]}\right) = \normop{S\left(\widehat{\bw}^{[s]};  \{\check{\bg}_n^{(k)}\}_{n\in [N]}, \widehat{\bmu}^{[s]}\right)}  = \normop{S\left(\overline{\bw}^{[s]};  \{\check{\bg}_n^{(k)}\}_{n\in [N]}, \widehat{\bmu}^{[s]}\right)} =	\normop{{\overline{S}^{[s]}}}$, we get
	\begin{align}
		\gamma\left(\widehat{\bw}^{[s]}; \{\check{\bg}_n^{(k)}\}_{n\in\cI}, \widehat{\bmu}^{[s]}\right) =
		\normop{{\overline{S}^{[s]}}} \leq 	(1+  \eta_{\rho}\nu) (1+ \eta_{w}\nu) \OPT(\widehat{\bmu}^{[s]}) + \frac{(1+  \eta_{\rho}\nu)\log\left(\frac{1}{1-\epsilon}\right)}{T \eta_w}  +  \frac{\log p}{T\eta_{\rho}}.
	\end{align}
	To minimize the RHS of the last step, we choose the step sizes to be 
	\begin{align}
		\eta_{w} =  \frac{1}{\nu} \sqrt{\frac{\log\left(\frac{1}{1-\epsilon}\right)}{ T}}, \mkern9mu \eta_{\rho} = \frac{1}{\nu} \sqrt{\frac{\log p}{T}}.
	\end{align}
	Note that as long as $T \geq 4 \max \left\{ \log\left(\frac{1}{1-\epsilon}\right), \log p\right\}$,  we can guarantee that $0 < \eta_{w}\nu, \eta_{\rho}\nu\leq \frac{1}{2}$.
	Additionally, since $ \normop{S^{[s, t]}}\leq \nu$ for all $t$ by \cref{lemma:normalizing-scale}, $\OPT(\widehat{\bmu}^{[s]}) \leq \nu$. After substituting the step sizes and rearranging, we get 
	\begin{align}
		\gamma\left(\widehat{\bw}^{[s]}; \{\check{\bg}_n^{(k)}\}_{n\in\cI}, \widehat{\bmu}^{[s]}\right) - \OPT(\widehat{\bmu}^{[s]}) &\leq  \nu\left(\eta_w \nu + \eta_\rho \nu + \eta_w \eta_\rho \nu^2\right)
		+ \frac{(1+  \eta_{\rho}\nu)\log\left(\frac{1}{1-\epsilon}\right)}{T \eta_w} 
		+\frac{\log p}{T\eta_\rho} \\
		& \leq 4 \nu \left( \sqrt{\frac{\log\left(\frac{1}{1-\epsilon}\right)}{T}} + \sqrt{\frac{\log p}{T}}\right) \eqcolon \delta_{T}(\widehat{\bmu}^{[s]}).
	\end{align}
\end{proofEnd}

\begin{remark}
	\cref{thm:spectral-norm-after-mw-mmw} shows that the MW-MMW rounds approximately solve the convex-concave min-max game for a fixed center	$\widehat{\bmu}^{[s]}$ and thereby approximately computes a reweighting with minimal projected second-moment operator norm about the fixed center $\widehat{\bmu}^{[s]}$. It is worth noting that this bound matches the mirror descent rate in \cite[Theorem 4.2]{bubeck_convex_2015}:
	\begin{theorem}[{\cite[Theorem 4.2]{bubeck_convex_2015}}]
		Given a mirror map $\Phi$ that is $\rho$-strongly convex on $\cX \cap \cD$ w.r.t.~$\norm{\cdot}$, radius $R^2= \sum_{x\in \cX\cap\cD} \Phi(x) - \Phi(x_1)$  $\eta =\frac{R}{L} \sqrt{\frac{2\rho}{t}}$, and $f$ that is convex and $L$-Lipschitz w.r.t.~$\norm{\cdot}$, the mirror descent with step size $\eta =\frac{R}{L} \sqrt{\frac{2\rho}{T}}$ satisfies
		\begin{align}
			f\left(\frac{1}{T} \sum_{t=1}^T x_t\right) - f(x^\star) \leq RL\sqrt{\frac{2}{\rho T}}.
		\end{align}
	\end{theorem}
	To interpret this result in our context,
	\begin{itemize}
		\item $L = 1$ since the MW loss $m_n^{[s, t]} \in [0,1]$ are already normalized (by \cref{lemma:normalizing-scale}),
		\item primal $\bw$-player: here the mirror maps is the relative entropy; the radius $R$ is therefore $\RE(\bw || \widetilde{\bw}^{[s, 1]})$, which is bounded by $\log\left(\frac{1}{1-\epsilon}\right)$,
		\item dual $\rho$-player: here the mirror map is the von Neumann entropy on density matrices; the radius is therefore $\log p$. 
	\end{itemize}
\end{remark}

\subsection{Convergence of the fixed-center updates}
\label{sec:fixed-center-updates}
\textEnd[both, category=fixed-center-updates]{In this section, we will show that the fixed-center updates in \cref{alg:spectral-reweighting}
	converge and that terminate within $\cO\left(\frac{\log\left( \frac{(e^{[1]} - R_{\infty})_{+}}{(R - R_{\infty})_{+}}\right)}{\log\left(\frac{1}{\alpha_{\epsilon}}\right)}\right)$.
	
	Fix the moment order $k$. The fixed center at iteration $s$ is $\widehat{\bmu}^{[s]} = \overline{\check{\bg}_{\widehat{\bw}^{[s]}}^{(k)}} = \sum_{n=1}^N \widehat{w}_n^{[s]} \check{\bg}_n^{(k)}$.  The population mean, as defined in \cref{eq:population-gradient-mean}, is $\bmu_{\bg}^{(k)} \coloneq \EX_Y\left[\bg^{(k)}(\bY)\right]$. The fixed-center update is the following:
	\begin{align}
		\mu^{[s+1]} \gets \overline{\check{\bg}_{\widehat{\bw}^{[s+1]}}^{(k)}} = \sum_{n=1}^N \widehat{w}_n^{[s+1]} \check{\bg}_n^{(k)}, \mkern9mu \widehat{\bw}^{[s+1]} \approx \OPT(\widehat{\bmu}^{[s]}).
\end{align} }

The convergence of the fixed-center updates requires the following stability conditions on the inlier per-observation gradients.
\begin{assumption}[Stability conditions on the inliers]
	\label{ass:inlier-stability} 
	Fix the moment order $k \in [L]$. Suppose that the set of per-observation gradients $\{\check{\bg}_n^{(k)} \}_{n=1}^N \in \R^p$ is distributed as $\cD_g$ with population mean $\bmu_{\bg}^{(k)}$ (\cref	{eq:population-gradient-mean}) and population covariance $\Sigma_{\bg}^{(k)}$ (\cref	{eq:population-gradient-cov}). Assume the following stability condition: there exists $\delta_{\mu,k}, \delta_{\Sigma,k} \geq 0$ such that for every inlier weight vector $\bw \in \Delta_{\cI_{\inlier}, \frac{\epsilon}{1-\epsilon}}$:
	\begin{align}
		&\normtwo{\sum_{n\in \cI_{\inlier}} w_n 	\left(\check{\bg}_n^{(k)} - \bmu_{\bg}^{(k)}\right)} \leq \delta_{\mu,k},	\label{eq:stability-condition-mean}\\
		&\sum_{n\in\cI_{\inlier}}w_n  \left(\check{\bg}_n^{(k)} - \bmu_{\bg}^{(k)}\right) \left(\check{\bg}_n^{(k)} - \bmu_{\bg}^{(k)}\right)^{\top}  \preceq  \Sigma_{\bg}^{(k)}  + \delta_{\Sigma,k} I.	\label{eq:stability-condition-cov}
	\end{align}
\end{assumption}

\begin{theoremEnd}[end, category=fixed-center-updates]{lemma}[Oracle inlier weight is feasible]
	\label{lemma:oracle-feasible}
	Define the oracle inlier weight vector 
	\begin{align}
		w_n^\sharp \coloneq \frac{\mathbf 1_{\{n\in \cI_{\inlier}\}}}{\left|\cI_{\inlier}\right|},\mkern9mu n=1,\dots,N.
	\end{align}
	Then $\bw^\sharp\in\Delta_{N,\epsilon}$.
\end{theoremEnd}

\begin{proofEnd}
	We have $\sum_n w_n^\sharp =1$ and $w_n^\sharp\ge 0$. Since $|\cI_{\inlier}|\ge (1-\epsilon)N$,
	\begin{align}
		0\leq 	w_n^\sharp \leq \frac{1}{(1-\epsilon)N},
		\mkern9mu \forall n=1,\dots,N,
	\end{align}
	which implies that $\bw^\sharp\in\Delta_{N,\epsilon}$.
\end{proofEnd}

\begin{theoremEnd}[end, category=fixed-center-updates]{lemma}[Inlier-outlier decomposition]
	\label{lemma:inlier-outlier-decomp}
	For any $\bw\in\Delta_{N,\epsilon}$, define the outlier mass and the inlier mass:
	\begin{align}
		\tau_{\outlier}(\bw) \coloneq \sum_{n\in\cI_{\outlier}} w_n, \mkern9mu \tau_{\inlier}(\bw) \coloneq \sum_{n\in\cI_{\inlier}} w_n  = 1- \tau_{\outlier}.
	\end{align}
	Then, the following hold for all $\bw \in \Delta_{N,\epsilon}$:
	\begin{subequations}
		\begin{alignat}{2}
			\text{\textbullet}\mkern9mu & \tau_{\outlier}(\bw) \leq \frac{\epsilon}{1-\epsilon}, \mkern9mu \tau_{\inlier}(\bw)  \geq \frac{1-2\epsilon}{1-\epsilon}, \label{lemma:inlier-outlier-decomp-weight}\\
			\text{\textbullet}\mkern9mu  &\overline{\check{\bg}_{\bw}^{(k)}} =  \tau_{\inlier}(\bw)  \restr{\overline{\check{\bg}_{\bw}^{(k)}}}{\cI_{\inlier}} + \tau_{\outlier}(\bw)  \restr{\overline{\check{\bg}_{\bw}^{(k)}}}{\cI_{\outlier}}\label{lemma:inlier-outlier-decomp-mean}\\ 
			\text{\textbullet}\mkern9mu  &\check{S}_{\bg,\bw}^{(k)} = \tau_{\inlier}(\bw)	\restr{\check{S}_{\bg,\bw}^{(k)}}{\cI_{\inlier}}  + \tau_{\outlier}(\bw)	\restr{\check{S}_{\bg,\bw}^{(k)}}{\cI_{\outlier}}+ \tau_{\inlier}(\bw) \tau_{\outlier}(\bw) \left( \restr{\overline{\check{\bg}_{\bw}^{(k)}}}{\cI_{\inlier}} - \restr{\overline{\check{\bg}_{\bw}^{(k)}}}{\cI_{\outlier}} \right) \left( \restr{\overline{\check{\bg}_{\bw}^{(k)}}}{\cI_{\inlier}} - \restr{\overline{\check{\bg}_{\bw}^{(k)}}}{\cI_{\outlier}} \right)^\top,\label{lemma:inlier-outlier-decomp-cov}
		\end{alignat}
	\end{subequations}
	where $	\restr{\overline{\check{\bg}_{\bw}^{(k)}}}{\cI}$ is the weighted sample mean as defined in \cref{eq:adv-gradient-mean-weighted} and $\restr{\check{S}_{\bg,\bw}^{(k)}}{\cI}$ is the weighted covariance as defined in \cref{eq:adv-gradient-cov-weighted}.
\end{theoremEnd}
\begin{proofEnd}
	Since $|\cI_{\outlier}|\leq \epsilon N$ and each per-observation gradient weight is bounded by $1/((1-\epsilon)N)$, we have
	\begin{align}
		\tau_{\outlier}(\bw) &=\sum_{n\in\cI_{\outlier}} w_n \leq \frac{|\cI_{\outlier}|}{(1-\epsilon)N} \leq \frac{\epsilon}{1-\epsilon},\\
		\tau_{\inlier}(\bw)  &= 1 - \tau_{\outlier}(\bw) \geq \frac{1-2\epsilon}{1-\epsilon}.
	\end{align}
	The mean decomposition is due to the partition of the weighted sum over $\cI_{\inlier}\sqcup\cI_{\outlier}$:
	\begin{align}
		\overline{\check{\bg}_{\bw}^{(k)}} &\coloneq \sum_{n=1}^N w_n \check{\bg}_n^{(k)} \\
		&= \sum_{n\in\cI_{\inlier}} w_n  \check{\bg}_n^{(k)} + \sum_{n\in\cI_{\outlier}} w_n \check{\bg}_n^{(k)}\\
		&= \tau_{\inlier}(\bw)\left(\frac{1}{\tau_{\inlier}(\bw)}\sum_{n\in\cI_{\inlier}} w_n \check{\bg}_n^{(k)}\right)
		+\tau_{\outlier}(\bw)\left(\frac{1}{\tau_{\outlier}(\bw)}\sum_{n\in\cI_{\outlier}} w_n \check{\bg}_n^{(k)}\right)\\
		&= \tau_{\inlier}(\bw)  \restr{\overline{\check{\bg}_{\bw}^{(k)}}}{\cI_{\inlier}} + \tau_{\outlier}(\bw)  \restr{\overline{\check{\bg}_{\bw}^{(k)}}}{\cI_{\outlier}}.
	\end{align}
	For the covariance decomposition:
	\begin{align}
		\check{\bg}_n^{(k)} - \overline{\check{\bg}_{\bw}^{(k)}}
		=
		\begin{cases}
			\left(\check{\bg}_n^{(k)} - 	\restr{\overline{\check{\bg}_{\bw}^{(k)}}}{\cI_{\inlier}} \right) 
			+ \tau_{\outlier}(\bw)\left(\restr{\overline{\check{\bg}_{\bw}^{(k)}}}{\cI_{\inlier}} - 	\restr{\overline{\check{\bg}_{\bw}^{(k)}}}{\cI_{\outlier}}  \right), & n\in\cI_{\inlier},\\
			\left(\check{\bg}_n^{(k)} -	\restr{\overline{\check{\bg}_{\bw}^{(k)}}}{\cI_{\outlier}}\right) - \tau_{\inlier}(\bw)\left(	\restr{\overline{\check{\bg}_{\bw}^{(k)}}}{\cI_{\inlier}}- 	\restr{\overline{\check{\bg}_{\bw}^{(k)}}}{\cI_{\outlier}}\right), & n\in\cI_{\outlier}.
		\end{cases}
	\end{align}
	For brevity, we denote $a_{\inlier} \coloneq \restr{\overline{\check{\bg}_{\bw}^{(k)}}}{\cI_{\inlier}}$, $a_{\outlier} \coloneq \restr{\overline{\check{\bg}_{\bw}^{(k)}}}{\cI_{\outlier}}$, $\Delta \coloneq a_{\inlier}  - a_{\outlier} $. Expanding the definition of $\check{S}_{\bg,\bw}^{(k)}$, we get
	\begin{align*}
		\check{S}_{\bg,\bw}^{(k)} &\coloneq \sum_{n=1}^N w_n  \left(\check{\bg}_n^{(k)} - \overline{\check{\bg}_{\bw}^{(k)}} \right) \left(\check{\bg}_n^{(k)} - \overline{\check{\bg}_{\bw}^{(k)}} \right)^{\top} \numberthis\\
		&= \sum_{n\in\cI_{\inlier}} w_n  \left(	\left(\check{\bg}_n^{(k)} - a_{\inlier}\right) 
		+ \tau_{\outlier}(\bw)\Delta \right) \left(\left(\check{\bg}_n^{(k)} - a_{\inlier}\right) 
		+ \tau_{\outlier}(\bw)\Delta \right)^\top \\
		&\quad +  \sum_{n\in\cI_{\outlier}}  w_n \left( 	\left(\check{\bg}_n^{(k)} -	a_{\outlier}\right) - \tau_{\inlier}(\bw)\Delta \right)  \left( 	\left(\check{\bg}_n^{(k)} -	a_{\outlier}\right) - \tau_{\inlier}(\bw)\Delta \right) ^\top \numberthis\\
		&=  \sum_{n\in\cI_{\inlier}} w_n \left(\check{\bg}_n^{(k)} - a_{\inlier}\right) \left(\check{\bg}_n^{(k)} - a_{\inlier}\right) ^\top  + \tau_{\outlier}(\bw)^2 \sum_{n\in\cI_{\inlier}} w_n \Delta  \Delta^\top \\
		& \quad + \underbrace{\tau_{\outlier}(\bw) \sum_{n\in\cI_{\inlier}} w_n  \left(\check{\bg}_n^{(k)} - a_{\inlier}\right) \Delta^\top}_{= 0} +  \underbrace{\tau_{\outlier} (\bw) \sum_{n\in\cI_{\inlier}} w_n \Delta \left(\check{\bg}_n^{(k)} - a_{\inlier}\right)^\top}_{= 0}\\
		& \quad +  \sum_{n\in\cI_{\outlier}} w_n \left(\check{\bg}_n^{(k)} - a_{\outlier}\right) \left(\check{\bg}_n^{(k)} - a_{\outlier}\right) ^\top + \tau_{\inlier}(\bw)^2 \sum_{n\in\cI_{\outlier}} w_n \Delta  \Delta^\top\\
		&\quad - \underbrace{\tau_{\inlier}(\bw) \sum_{n\in\cI_{\outlier}} w_n  \left(\check{\bg}_n^{(k)} - a_{\outlier}\right) \Delta^\top}_{= 0} - \underbrace{\tau_{\inlier}(\bw) \sum_{n\in\cI_{\outlier}} w_n \Delta \left(\check{\bg}_n^{(k)} - a_{\outlier}\right)^\top}_{= 0}\numberthis\\
		&= \tau_{\inlier}(\bw)	\restr{\check{S}_{\bg,\bw}^{(k)}}{\cI_{\inlier}}+ \tau_{\inlier}(\bw) \tau_{\outlier}(\bw)^2\Delta  \Delta^\top + \tau_{\outlier}(\bw)	\restr{\check{S}_{\bg,\bw}^{(k)}}{\cI_{\outlier}} + \tau_{\outlier}(\bw) \tau_{\inlier}(\bw)^2\Delta  \Delta^\top \numberthis\\
		&= \tau_{\inlier}(\bw)	\restr{\check{S}_{\bg,\bw}^{(k)}}{\cI_{\inlier}} +  \tau_{\outlier}(\bw)	\restr{\check{S}_{\bg,\bw}^{(k)}}{\cI_{\outlier}}  + \tau_{\inlier} (\bw) \tau_{\outlier} (\bw) \left((1- \tau_{\outlier}(\bw) )\Delta \Delta^\top + \tau_{\outlier} (\bw)\Delta \Delta^\top \right) \numberthis\\
		&= \tau_{\inlier}(\bw)	\restr{\check{S}_{\bg,\bw}^{(k)}}{\cI_{\inlier}} +  \tau_{\outlier}(\bw)	\restr{\check{S}_{\bg,\bw}^{(k)}}{\cI_{\outlier}}  + \tau_{\inlier} (\bw) \tau_{\outlier}(\bw) \Delta \Delta^\top,\numberthis
	\end{align*}
	which is the claimed identity in \cref{lemma:inlier-outlier-decomp-cov}: 
	\begin{align*}
		\check{S}_{\bg,\bw}^{(k)} = \tau_{\inlier}(\bw)	\restr{\check{S}_{\bg,\bw}^{(k)}}{\cI_{\inlier}}  + \tau_{\outlier}(\bw)	\restr{\check{S}_{\bg,\bw}^{(k)}}{\cI_{\outlier}}+ \tau_{\inlier}(\bw) \tau_{\outlier}(\bw) \left( \restr{\overline{\check{\bg}_{\bw}^{(k)}}}{\cI_{\inlier}} - \restr{\overline{\check{\bg}_{\bw}^{(k)}}}{\cI_{\outlier}} \right) \left( \restr{\overline{\check{\bg}_{\bw}^{(k)}}}{\cI_{\inlier}} - \restr{\overline{\check{\bg}_{\bw}^{(k)}}}{\cI_{\outlier}} \right)^\top.
	\end{align*}
\end{proofEnd}

\begin{theoremEnd}[end, category=fixed-center-updates]{lemma}[Centering identity]
	\label{lemma:centering-identity}
	Suppose	$\cI \subseteq [N]$ is an arbitrary index set. The following holds for all $\bw \in \Delta_{\cI,\epsilon}$ and for all $\widehat{\bmu} \in \R^{p}$:
	\begin{align}
		S \left(\bw; \{\check{\bg}_n^{(k)}\}_{n\in\cI} , \widehat{\bmu}\right) =  \restr{\check{S}_{\bg,\bw}^{(k)}}{\cI} + \left( \restr{\overline{\check{\bg}_{\bw}^{(k)}}}{\cI} - \widehat{\bmu}\right) \left( \restr{\overline{\check{\bg}_{\bw}^{(k)}}}{\cI} - \widehat{\bmu}\right)^{\top}, 
	\end{align}
	
	and therefore,
	\begin{align}
		\normop{\restr{\check{S}_{\bg,\bw}^{(k)}}{\cI}} \leq \normop{S \left(\bw; \{\check{\bg}_n^{(k)}\}_{n\in\cI} , \widehat{\bmu}\right)},
	\end{align}
	where $	\restr{\overline{\check{\bg}_{\bw}^{(k)}}}{\cI}$ is the weighted sample mean as defined in \cref{eq:adv-gradient-mean-weighted}, $\restr{\check{S}_{\bg,\bw}^{(k)}}{\cI}$ is the weighted covariance as defined in \cref{eq:adv-gradient-cov-weighted}, and $S \left(\bw; \{\check{\bg}_n^{(k)}\}_{n\in\cI} , \widehat{\bmu}\right)$ is defined in 	\cref{def:adv-gradient-cov-weighted-fixed}.
\end{theoremEnd}
\begin{proofEnd}
	Expanding $\check{\bg}_n^{(k)} - \widehat{\mu} = \left(\check{\bg}_n^{(k)} - \restr{\overline{\check{\bg}_{\bw}^{(k)}}}{\cI} \right) + \left(\restr{\overline{\check{\bg}_{\bw}^{(k)}}}{\cI}- \widehat{\bmu}\right)$, we get
	\begin{align*}
		&\left(\check{\bg}_n^{(k)} - \widehat{\mu} \right) \left(\check{\bg}_n^{(k)} - \widehat{\bmu} \right)^{\top} \\
		=& \left(\check{\bg}_n^{(k)} - \restr{\overline{\check{\bg}_{\bw}^{(k)}}}{\cI} \right) \left(\check{\bg}_n^{(k)} - \restr{\overline{\check{\bg}_{\bw}^{(k)}}}{\cI} \right)^{\top} 
		+ \left( \restr{\overline{\check{\bg}_{\bw}^{(k)}}}{\cI} - \widehat{\bmu}\right) \left( \restr{\overline{\check{\bg}_{\bw}^{(k)}}}{\cI} - \widehat{\bmu}\right)^{\top} 
		+ 2 \left( \check{\bg}_n^{(k)} - \restr{\overline{\check{\bg}_{\bw}^{(k)}}}{\cI}\right) \left( \restr{\overline{\check{\bg}_{\bw}^{(k)}}}{\cI} - \widehat{\bmu}\right)^{\top}.\numberthis
	\end{align*}
	
	Multiplying both sides by $w_n$ and summing over $n$, the mixed terms vanish since $\sum_{n\in\cI} w_n  \left( \check{\bg}_n^{(k)} - \restr{\overline{\check{\bg}_{\bw}^{(k)}}}{\cI}\right) = 0$, and we get
	\begin{align}
		S \left(\bw; \{\check{\bg}_n^{(k)}\}, \widehat{\bmu}\right) &\coloneq  \sum_{n\in\cI}   w_n \left(\check{\bg}_n^{(k)} - \widehat{\bmu} \right) \left(\check{\bg}_n^{(k)} - \widehat{\bmu} \right)^{\top}\\
		&= \sum_{n\in\cI}  w_n  \left(\check{\bg}_n^{(k)} - \restr{\overline{\check{\bg}_{\bw}^{(k)}}}{\cI} \right) \left(\check{\bg}_n^{(k)} - \restr{\overline{\check{\bg}_{\bw}^{(k)}}}{\cI} \right)^{\top} 
		+ \left( \restr{\overline{\check{\bg}_{\bw}^{(k)}}}{\cI} - \widehat{\bmu}\right) \left( \restr{\overline{\check{\bg}_{\bw}^{(k)}}}{\cI} - \widehat{\bmu}\right)^{\top}\\
		&= \restr{\check{S}_{\bg,\bw}^{(k)}}{\cI} + \left(\restr{\overline{\check{\bg}_{\bw}^{(k)}}}{\cI} - \widehat{\bmu}\right) \left( \restr{\overline{\check{\bg}_{\bw}^{(k)}}}{\cI} - \widehat{\bmu}\right)^{\top},
	\end{align}
	which then implies that 
	\begin{align*}
		\normop{\restr{\check{S}_{\bg,\bw}^{(k)}}{\cI} } \leq \normop{S \left(\bw; \{\check{\bg}_n^{(k)}\}_{n\in\cI}, \widehat{\bmu}\right)}.
	\end{align*}
\end{proofEnd}

\begin{theoremEnd}[end, category=fixed-center-updates]{lemma}[Contraction factor]
	\label{lemma:contraction-factor}
	Under \cref{ass:inlier-stability}, for any arbitrary fixed center $\widehat{\bmu} \in \R^{p}$, the following holds for all $\bw \in \Delta_{N,\epsilon}$:
	\begin{align}
		\normtwo{\overline{\check{\bg}_{\bw}^{(k)}}- \bmu_{\bg}^{(k)} } \leq \delta_{\mu,k} + \alpha_{\epsilon} \sqrt{\gamma\left({\bw}; \{\check{\bg}_n^{(k)}\}_{n\in [N]}, \widehat{\bmu} \right)}, \mkern9mu \alpha_{\epsilon} \coloneq \sqrt{\frac{\epsilon}{1-2\epsilon}}. 
	\end{align}
	In particular, $\alpha_{\epsilon} <1$ whenever $\epsilon <\frac{1}{3}$.
\end{theoremEnd}
\begin{proofEnd}
	By \cref{lemma:inlier-outlier-decomp-mean}, we get 
	\begin{align}
		\normtwo{\overline{\check{\bg}_{\bw}^{(k)}}- \bmu_{\bg}^{(k)} } &=	\normtwo{ \tau_{\inlier}(\bw) 
			\left(  \restr{\overline{\check{\bg}_{\bw}^{(k)}}}{\cI_{\inlier}}- \bmu_{\bg}^{(k)} \right) 
			+ \tau_{\outlier}(\bw) \left( \restr{\overline{\check{\bg}_{\bw}^{(k)}}}{\cI_{\outlier}} - \bmu_{\bg}^{(k)}  \right) }\\
		&= \normtwo{
			\left(  \restr{\overline{\check{\bg}_{\bw}^{(k)}}}{\cI_{\inlier}}- \bmu_{\bg}^{(k)} \right) 
			+ \tau_{\outlier}(\bw) \left( \restr{\overline{\check{\bg}_{\bw}^{(k)}}}{\cI_{\outlier}} -\restr{\overline{\check{\bg}_{\bw}^{(k)}}}{\cI_{\inlier}}  \right) }\\
		&\leq \normtwo{ \restr{\overline{\check{\bg}_{\bw}^{(k)}}}{\cI_{\inlier}}- \bmu_{\bg}^{(k)}} 
		+ \tau_{\outlier}(\bw) \normtwo{ \restr{\overline{\check{\bg}_{\bw}^{(k)}}}{\cI_{\outlier}} -\restr{\overline{\check{\bg}_{\bw}^{(k)}}}{\cI_{\inlier}}  }\label{step:inlier-outlier-decomp-mean}
	\end{align}
	
	By \cref{lemma:inlier-outlier-decomp-cov}, we get
	\begin{align}
		\tau_{\inlier}(\bw) \tau_{\outlier}(\bw) \left( \restr{\overline{\check{\bg}_{\bw}^{(k)}}}{\cI_{\inlier}} - \restr{\overline{\check{\bg}_{\bw}^{(k)}}}{\cI_{\outlier}} \right) \left( \restr{\overline{\check{\bg}_{\bw}^{(k)}}}{\cI_{\inlier}} - \restr{\overline{\check{\bg}_{\bw}^{(k)}}}{\cI_{\outlier}} \right)^\top &\preceq 	\check{S}_{\bg,\bw}^{(k)} \\
		\tau_{\inlier}(\bw) \tau_{\outlier}(\bw) \normtwo{\restr{\overline{\check{\bg}_{\bw}^{(k)}}}{\cI_{\inlier}} - \restr{\overline{\check{\bg}_{\bw}^{(k)}}}{\cI_{\outlier}}}^2 &\leq \normop{\check{S}_{\bg,\bw}^{(k)}}\\
		\normtwo{\restr{\overline{\check{\bg}_{\bw}^{(k)}}}{\cI_{\outlier}} - \restr{\overline{\check{\bg}_{\bw}^{(k)}}}{\cI_{\inlier}}} &\leq \sqrt{ \frac{\normop{\check{S}_{\bg,\bw}^{(k)}}}{ \tau_{\inlier}(\bw) \tau_{\outlier}(\bw)}}.
		\label{step:inlier-outlier-decomp-cov}
	\end{align}
	
	By \cref{lemma:centering-identity}, we know that 
	\begin{align}
		\normop{\check{S}_{\bg,\bw}^{(k)}} \leq  \gamma\left({\bw}; \{\check{\bg}_n^{(k)}\}_{n\in [N]}, \widehat{\bmu} \right).	\label{step:centering-identity}
	\end{align}
	Combining \cref{step:inlier-outlier-decomp-mean}, \cref{step:inlier-outlier-decomp-cov}, \cref{step:centering-identity}, and rearranging, we get 
	\begin{align}
		\normtwo{\overline{\check{\bg}_{\bw}^{(k)}}- \bmu_{\bg}^{(k)} } & \normtwo{ \restr{\overline{\check{\bg}_{\bw}^{(k)}}}{\cI_{\inlier}}- \bmu_{\bg}^{(k)}} 
		+ \tau_{\outlier}(\bw)  \sqrt{\frac{\gamma\left({\bw}; \{\check{\bg}_n^{(k)}\}_{n\in [N]}, \widehat{\bmu}\right)}{	\tau_{\inlier}(\bw) \tau_{\outlier}(\bw)}}.\\
		&\leq \delta_{\mu,k} + \sqrt{\frac{\tau_{\outlier}(\bw)}{\tau_{\inlier}(\bw)}} \sqrt{\gamma\left({\bw}; \{\check{\bg}_n^{(k)}\}_{n\in [N]}, \widehat{\bmu}\right)}\\
		&\leq  \delta_{\mu,k} + \underbrace{\sqrt{\frac{\epsilon}{1-2\epsilon}}}_{\eqcolon \alpha_{\epsilon}} \sqrt{\gamma\left({\bw}; \{\check{\bg}_n^{(k)}\}_{n\in [N]}, \widehat{\bmu}\right)},
	\end{align}
	where the last step is due to \cref{lemma:inlier-outlier-decomp-weight}. 
\end{proofEnd}

\begin{theoremEnd}[end, category=fixed-center-updates]{lemma}[Bound on $\OPT(\widehat{\bmu})$ in terms of the error]
	\label{lemma:bound-opt}
	Let $\widehat{\bmu}\in \R^{p}$ be an arbitrary vector and let 
	$\nu \coloneq \diam\left( \{\check{\bg}_n\} \right)^2 = \max_{i,j} \normtwo{\check{\bg}_i - \check{\bg}_j}^2$ be the normalizing scale as defined in \cref{lemma:normalizing-scale}
	
	Then, under \cref{ass:inlier-stability}, we have 
	\begin{align}
		\OPT(\widehat{\bmu}) \leq  \normop{\Sigma_{\bg}^{(k)}}  +  \delta_{\Sigma,k} +\left( \delta_{\mu,k}  + \normtwo{\widehat{\bmu} -\bmu_{\bg}^{(k)} } \right)^2.
	\end{align} 
\end{theoremEnd}
\begin{proofEnd}
	By \cref{lemma:oracle-feasible},  the oracle inlier weight is feasible, that is, $\bw^\sharp \in \Delta_{N,\epsilon}$. Apply the stability assumption, we get 
	\begin{align}
		\normtwo{\restr{\overline{\check{\bg}_{\bw^\sharp}^{(k)}}}{\cI_{\inlier}}  - \bmu_{\bg}^{(k)}} \leq \delta_{\mu,k}, \mkern9mu \normop{\restr{\check{S}_{\bg, \bw^\sharp}^{(k)}}{\cI_{\inlier}}} \leq \normop{\Sigma_{\bg}^{(k)}}  +  \delta_{\Sigma,k},
	\end{align}
	which implies that $ \normop{\restr{\check{S}_{\bg, \bw^\sharp}^{(k)}}{\cI_{\inlier}}} \leq \normop{\Sigma_{\bg}^{(k)}}  +  \delta_{\Sigma,k}$.
	Applying \cref{lemma:centering-identity} to $\bw^\sharp$ and an arbitrary center $\widehat{\bmu}$: 
	\begin{align}
		\normop{S \left(\bw^\sharp; \{\check{\bg}_n^{(k)}\}, \widehat{\bmu} \right)} &= \normop{ \restr{\check{S}_{\bg, \bw^\sharp}^{(k)}}{\cI_{\inlier}} + \left(\restr{\overline{\check{\bg}_{\bw^\sharp}^{(k)}}}{\cI_{\inlier}}  - \widehat{\bmu}\right) \left(\restr{\overline{\check{\bg}_{\bw^\sharp}^{(k)}}}{\cI_{\inlier}}  - \widehat{\bmu}\right)^{\top}} \\
		&\leq \normop{\restr{\check{S}_{\bg, \bw^\sharp}^{(k)}}{\cI_{\inlier}} }  + \normtwo{\restr{\overline{\check{\bg}_{\bw^\sharp}^{(k)}}}{\cI_{\inlier}}  - \widehat{\bmu}}^2 \\
		&\leq \normop{\Sigma_{\bg}^{(k)}} +  \delta_{\Sigma,k} + \normtwo{\restr{\overline{\check{\bg}_{\bw^\sharp}^{(k)}}}{\cI_{\inlier}}  - \widehat{\bmu}}^2\\
		&\leq \normop{\Sigma_{\bg}^{(k)}}  +  \delta_{\Sigma,k} + \left( \normtwo{\restr{\overline{\check{\bg}_{\bw^\sharp}^{(k)}}}{\cI_{\inlier}}  -  \bmu_{\bg}^{(k)}} + \normtwo{ \bmu_{\bg}^{(k)} - \widehat{\bmu}} \right)^2\\ 
		&\leq \normop{\Sigma_{\bg}^{(k)}}  +  \delta_{\Sigma,k} +\left( \delta_{\mu,k} + \normtwo{\widehat{\bmu} -\bmu_{\bg}^{(k)} } \right)^2.
	\end{align}
	Therefore,
	\begin{align}
		\OPT(\widehat{\bmu}) = \min_{\bw \in \Delta_{N,\epsilon}} \max_{\rho\in \frakD_p}
		\inner{S(\bw; \{\check{\bg}_n^{(k)}\},\widehat{\bmu}) }{\rho}&= \min_{\bw \in \Delta_{N,\epsilon}} \normop{S\left({\bw}; \{\check{\bg}_n^{(k)}\}, \widehat{\bmu}\right)}\\
		&\leq	\normop{S \left(\bw^\sharp; \{\check{\bg}_n^{(k)}\}, \widehat{\bmu} \right)} \\
		&\leq  \normop{\Sigma_{\bg}^{(k)}}  +  \delta_{\Sigma,k} +\left(\delta_{\mu,k} + \normtwo{\widehat{\bmu} -\bmu_{\bg}^{(k)} } \right)^2.
	\end{align}
	
\end{proofEnd}

\begin{theoremEnd}[end, category=fixed-center-updates]{theorem}[Outer-loop convergence]
	\label{thm:outer-loop-convergence}
	Let $\alpha_{\epsilon} = \sqrt{\frac{\epsilon}{1-2\epsilon}}$. Denote the error at the $s$-th outer-loop update as $e^{[s]} \coloneq 	\normtwo{\widehat{\bmu}^{[s]} - \bmu_{\bg}^{(k)}}$. Then, under \cref{ass:inlier-stability}, the recurrence inequality holds for all $s\geq 1$:
	\begin{align}
		e^{[s+1]} &\leq \alpha_{\epsilon}e^{[s]} + R_{\epsilon,T},
	\end{align} 
	where 
	\begin{align}
		\label{eq:R-eps-T} 
		R_{\epsilon,T} = (1+ \alpha_{\epsilon})\delta_{\mu,k} + \alpha_{\epsilon} \sqrt{ \normop{\Sigma_{\bg}^{(k)}}  +  \delta_{\Sigma,k} + 4 \nu \left( \sqrt{\frac{\log\left(\frac{1}{1-\epsilon}\right)}{T}} + \sqrt{\frac{\log p}{T}}\right)}.
	\end{align}
	
	As a result, for all $s\geq 1$,
	\begin{align}
		e^{[s]} &\leq  \alpha_{\epsilon}^{s-1} e^{[1]} + \frac{1- \alpha_{\epsilon}^{s-1}}{1-\alpha_{\epsilon}} R_{\epsilon,T},
	\end{align}
	
	In particular, the outer-loop fixed-center updates convergence geometrically to the convergence radius $R_{\infty}$, that is, 
	\begin{align}
		\limsup_{s\to\infty} e^{[s]} \leq R_{\infty} = \frac{R_{\epsilon,T}}{1-\alpha_{\epsilon}}.
	\end{align}
\end{theoremEnd}

\begin{remark}[Interpretation of the rate]
	\label{rem:rate-interpretation}
	The spectral term in \cref{eq:R-eps-T} gives a $\sqrt\epsilon\,\sqrt{\normop\Sigma}$ contamination contribution, which is the natural bounded-covariance robust-mean rate.  Sharper $O(\epsilon)$ or $O(\epsilon\sqrt{\log(1/\epsilon)})$ rates require stronger distributional assumptions in the spirit of the Gaussian and sub-Gaussian robust-mean analyses of \cite{dalalyan2022allinone}. In our work, we do not impose such assumptions since our priority is to assert minimal distributional assumptions in order to make the primitive \cref{alg:spectral-reweighting} as general as possible. 
\end{remark}

\begin{proofEnd}
	Using the fact that $ \widehat{\bw}^{[s+1]}$ is an $ \cO\left( \nu\sqrt{\frac{\log\left(\frac{1}{1-\epsilon}\right)}{T}} + \nu\sqrt{\frac{\log p}{T}}\right)$-approximate minimizer of $\bw \mapsto \gamma\left({\bw}; \{\check{\bg}_n^{(k)}\}_{n\in [N]}, \widehat{\bmu}^{[s]}\right)$, we get for all $s\geq 1$,
	\begin{align}
		e^{[s+1]} &\coloneq	\normtwo{\widehat{\bmu}^{[s+1]} - \bmu_{\bg}^{(k)}} = \normtwo{\overline{\check{\bg}_{\widehat{\bw}^{[s+1]}}^{(k)}}- \bmu_{\bg}^{(k)} } \\
		&\leq \delta_{\mu,k} + \alpha_{\epsilon} \sqrt{\gamma\left(\widehat{\bw}^{[s+1]}; \{\check{\bg}_n^{(k)}\}_{n\in [N]}, \widehat{\bmu}^{[s]}\right)} \\
		&\leq\delta_{\mu,k} + \alpha_{\epsilon} \sqrt{\OPT(\bmu^{[s]}) + \delta_{T}(\widehat{\bmu}^{[s]}) } \\
		& \leq  \delta_{\mu,k} + \alpha_{\epsilon} \sqrt{ \normop{\Sigma_{\bg}^{(k)}}  +  \delta_{\Sigma,k} +\left( \delta_{\mu,k} + \normtwo{\widehat{\bmu}^{[s]} -\bmu_{\bg}^{(k)} } \right)^2 + \delta_{T}(\widehat{\bmu}^{[s]}) } \\
		&\leq  \delta_{\mu,k} + \alpha_{\epsilon}\left( \delta_{\mu,k} + \normtwo{\widehat{\bmu}^{[s]} -\bmu_{\bg}^{(k)} } \right) + \alpha_{\epsilon} \sqrt{ \normop{\Sigma_{\bg}^{(k)}}  +  \delta_{\Sigma,k} + \delta_{T}(\widehat{\bmu}^{[s]})}\\
		&= \alpha_{\epsilon} \underbrace{\normtwo{\widehat{\bmu}^{[s]} -\bmu_{\bg}^{(k)} }}_{\eqcolon e^{[s]}} +
		\underbrace{(1+ \alpha_{\epsilon}) \delta_{\mu,k} + \alpha_{\epsilon} \sqrt{ \normop{\Sigma_{\bg}^{(k)}}  +  \delta_{\Sigma,k} + \delta_{T}(\widehat{\bmu}^{[s]})}}_{\eqcolon R_{\epsilon,T}},
	\end{align}
	which gives the desired recurrent inequality 
	\begin{align}
		e^{[s+1]} &\leq \alpha_{\epsilon}e^{[s]} + R_{\epsilon,T},
	\end{align} 
	where
	\begin{align}
		R_{\epsilon,T} &= (1+ \alpha_{\epsilon})\delta_{\mu,k} + \alpha_{\epsilon} \sqrt{ \normop{\Sigma_{\bg}^{(k)}}  +  \delta_{\Sigma,k} + \delta_{T}(\widehat{\bmu}^{[s]})}\\
		&=  (1+ \alpha_{\epsilon})\delta_{\mu,k} + \alpha_{\epsilon} \sqrt{ \normop{\Sigma_{\bg}^{(k)}}  +  \delta_{\Sigma,k} + 4 \nu \left( \sqrt{\frac{\log\left(\frac{1}{1-\epsilon}\right)}{T}} + \sqrt{\frac{\log p}{T}}\right)}.
	\end{align}
	
	Since $R_{\epsilon,T}$ is a constant, by using the standard result for solution of the affine linear recurrence, we get
	\begin{align}
		e^{[s]} &\leq  \alpha_{\epsilon}^{s-1} e^{[1]} + \frac{1- \alpha_{\epsilon}^{s-1}}{1-\alpha_{\epsilon}} R_{\epsilon,T},
	\end{align}
	which implies that the outer-loop fixed-center updates convergence geometrically to the explicit radius, that is, 
	\begin{align}
		\label{eq:convergence-radius}
		\limsup_{s\to\infty} e^{[s]} \leq	R_{\infty} = \frac{R_{\epsilon,T}}{1-\alpha_{\epsilon}} = \frac{(1+ \alpha_{\epsilon}) \delta_{\mu,k} + \alpha_{\epsilon} \sqrt{ \normop{\Sigma_{\bg}^{(k)}}  +  \delta_{\Sigma,k} + \delta_{T}(\widehat{\bmu}^{[s]})}}{1-\alpha_{\epsilon}}.
	\end{align}
\end{proofEnd}

\begin{theoremEnd}[end, category=fixed-center-updates]{theorem}[Finite outer-loop termination]
	\label{thm:sgr-termination-error}
	As per \cref{alg:spectral-reweighting}, the termination criteria be $\normop{\overline{S}^{[s]}} =\gamma\left(\widehat{\bw}^{[s]}; \{\check{\bg}_n^{(k)}\}_{n\in [N]}, \widehat{\bmu}^{[s]}\right)  \leq C_{\text{stop}, k}$. 
	Fix a target radius $R_k > R_{\infty}$, where $R_{\infty}$ is the radius of convergence defined in \cref{eq:convergence-radius}.  Under \cref{ass:inlier-stability}, choosing the termination condition of \cref{alg:spectral-reweighting}, namely $\gamma\left(\widehat{\bw}^{[s]}; \{\check{\bg}_n^{(k)}\}_{n\in [N]}, \widehat{\bmu}^{[s]}\right) \leq C_{\text{stop}, k}$, such that
	\begin{align}
		\label{eq:Cstop}
		C_{\text{stop}, k} \geq \normop{\Sigma_{\bg}^{(k)}} + \delta_{\Sigma,k} + ( \delta_{\mu,k} + R_k)^2 + \delta_{T},
	\end{align}
	then we have the following: 
	\begin{enumerate}
		\item If the error is within the radius of convergence $e^{[s]}\leq R$, then the termination condition of \cref{alg:spectral-reweighting} is satisfied at iteration $s$, no later than
		\begin{align}
			s_{\max} \coloneq 1+\ceil{\frac{\log\left( \frac{(e^{[1]} - R_{\infty})}{(R - R_{\infty})_{+}}\right)}{\log\left(\frac{1}{\alpha_{\epsilon}}\right)}}.
		\end{align}
		\item If the algorithm stops at iteration $s$ and returns the weight vector $\bw^{[s]}$, then the corresponding weighted mean satisfies 
		\begin{align}
			\normtwo{\overline{\check{\bg}_{\widehat{\bw}^{[s]}}^{(k)}} - \bmu_{\bg}^{(k)}} \leq \delta_{\mu,k} + \alpha_{\epsilon} \sqrt{C_{\text{stop}, k}}.
		\end{align}
	\end{enumerate}
\end{theoremEnd}

\begin{proofEnd}
	Fix a target radius $R_k > R_{\infty}$. By \cref{thm:outer-loop-convergence}, 
	\begin{align}
		\alpha_{\epsilon}^{s-1} (e^{[1]} - R_{\infty}) \leq R_k - R_{\infty}.
	\end{align}
	
	Solving for $s$, we get that for all $s\geq s_{\max} \coloneq 1+\ceil{\frac{\log\left( \frac{(e^{[1]} - R_{\infty})_{+}}{(R_k - R_{\infty})_{+}}\right)}{\log\left(\frac{1}{\alpha_{\epsilon}}\right)}}$, we have
	\begin{align}
		e^{[s]} \leq R_k		\implies \gamma\left(\widehat{\bw}^{[s]}; \{\check{\bg}_n^{(k)}\}_{n\in [N]}, \widehat{\bmu}^{[s]}\right) &\leq \normop{\Sigma_{\bg}^{(k)}} + \delta_{\Sigma,k} + ( \delta_{\mu,k} + R_k)^2 + \delta_{T} \leq  C_{\text{stop}, k},
	\end{align}
	where the last inequality is due to  \cref{lemma:bound-opt} and \cref{thm:spectral-norm-after-mw-mmw}. Therefore, with the stopping rule $ \gamma\left(\widehat{\bw}^{[s]}; \{\check{\bg}_n^{(k)}\}_{n\in [N]}, \widehat{\bmu}^{[s]}\right) \leq C_{\text{stop}, k}$, \cref{alg:spectral-reweighting} terminates after at most $	s_{\max} \coloneq 1+\ceil{\frac{\log\left( \frac{(e^{[1]} - R_{\infty})_{+}}{(R_k - R_{\infty})_{+}}\right)}{\log\left(\frac{1}{\alpha_{\epsilon}}\right)}}$ outer iterations. 
\end{proofEnd}

\begin{remark}
	\cref{thm:outer-loop-convergence}  gives a contraction guarantee of the form
	\begin{align}
		e^{[s+1]} \leq \alpha_{\epsilon} e^{[s]} + \text{statistical floor} + \text{optimization floor}.
	\end{align}
	
	Combining \cref{thm:spectral-norm-after-mw-mmw} with \cref{thm:outer-loop-convergence}, we obtain the explicit dependence
	\begin{align}
		R_\infty = \frac{(1+\alpha_{\epsilon})\delta}{1-\alpha_{\epsilon}} +
		\frac{\alpha_{\epsilon}}{1-\alpha_{\epsilon}}
		\sqrt{\normop{\Sigma}+\delta_{\Sigma,k} +4\nu\left(\sqrt{\frac{\log\left(\frac{1}{1-\epsilon}\right)}{T}}+\sqrt{\frac{\log p}{T}}\right)}.
	\end{align}
	Thus, to make the optimization floor no larger than a prescribed amount $R_{\mathrm{opt}}>0$, it is sufficient to impose
	\begin{align}
		\delta_{T}\leq  \left(\frac{(1-\alpha_\epsilon) R_{\mathrm{opt}}}{\alpha_\epsilon}\right)^2,
	\end{align}
	which, by \cref{thm:spectral-norm-after-mw-mmw}, is ensured by the explicit inner-loop budget
	\begin{equation}
		\label{eq:T-choice}
		T \gtrsim \nu^2\left(\log p+\log\left(\frac{1}{1-\epsilon}\right)\right)
		\left(\frac{\alpha_\epsilon}{(1-\alpha_\epsilon)R_{\mathrm{opt}}}\right)^4.
	\end{equation}
\end{remark}

\subsection{Local finite-sample GMM analysis}
\label{sec:finite-sample-gmm}
\textEnd[both, category=finite-sample-gmm]{The results in the previous sections control the error for robust moment gradient estimation.  In
	this section, we convert the gradient estimation error for \cref{alg:spectral-reweighting} into a local finite-sample parameter estimation error for \cref{alg:robust-sgr-gmm}. 
	Let $\widehat\bw^{(k)}(\btheta)\in\Delta_{N,\epsilon}$ be the output of \cref{alg:spectral-reweighting} run on the score cloud $\{\check\bg_n^{(k)}(\btheta)\}_{n=1}^N$. Let \(\Pi_k:\mathbb R^q\to\mathbb R^{q_k}\) denote the canonical projection onto
	the \(k\)-th moment block. Let $A_k(\btheta) \coloneq G(\btheta)^\top W \Pi_k^\top \in \R^{p\times q_k}$. Then, the population moment gradients and the SGR-weighted per-observation moment gradient can be rewritten in terms of the moment blocks: 
	\begin{align}
		\Psi(\btheta)
		&\coloneq G(\btheta)^\top W m(\btheta) = \sum_{k=1}^L A_k(\btheta) m_k(\btheta) = \sum_{k=1}^L  a_k \bmu_{\bg}^{(k)}(\btheta),
		\label{eq:population-gmm-score}\\
		\widehat\Psi^{(\sgr)}(\btheta)
		&\coloneq \sum_{k=1}^L a_k \sum_{n=1}^N \widehat w_n^{(k)}(\btheta)\,
		\check\bg_n^{(k)}(\btheta).
		\label{eq:sgr-score}
\end{align}}

We state the finite-sample analogues of the standard rank and differentiability assumptions used in classical GMM theory. 

\begin{assumption}[Standard GMM local identification conditions, see, e.g., \cite{hansen1982large, newey1994large, hall2004generalized}]
	\label{ass:local-identification}
	We assume that there exists \(r_0>0\) such that the closed ball
	\begin{align}
		\cB_0:=\{\btheta\in\Theta:\normtwo{\btheta-\btheta^\star}\leq r_0\}
	\end{align}
	is contained in \(\Theta\) and the following standard GMM local smooth identification conditions hold: 
	\begin{enumerate}
		\item \textbf{Correct specification:} $m (\btheta^\star) = \bm{0}$ (and consequently, $\Psi(\btheta^\star) =\bm{0}$).
		
		\item \textbf{Differentiability:} $m(\btheta)$ is continuously differentiable on
		$\cB_0$, with derivative $G(\btheta) = \nabla_{\btheta} m(\btheta)$, and there is a Lipschitz constant $L_G<\infty$ such that
		\begin{equation}\label{eq:G-lipschitz}
			\normop{G(\btheta)-G(\btheta')}
			\leq L_G\normtwo{\btheta-\btheta'}
			\qquad \text{for all }\btheta,\btheta'\in\cB_0,
		\end{equation}
		
		\item \textbf{Full-rank local identification:} with $G^\star=G(\btheta^\star)$,
		\begin{equation}\label{eq:Hstar-lambda}
			H^\star=(G^\star)^\top W G^\star,
			\qquad
			\lambda^\star=\lambda_{\min}(H^\star)>0,
		\end{equation}
		
		\item \textbf{Local radius condition:}
		\begin{equation}\label{eq:radius-condition}
			\normop{W}L_G r_0\left(\frac32\normop{G^\star}+\frac12 L_G r_0\right)
			\leq \frac{\lambda^\star}{2}.
		\end{equation}
	\end{enumerate}
\end{assumption}
\begin{remark}[Relation to classical GMM rank assumptions]
	The Lipschitz derivative
	condition is a quantitative finite-sample version of the differentiability and continuity hypotheses in \cite[Assumption~3.5]{hall2004generalized}. The condition $\lambda_{\min}((G^\star)^\top W G^\star)>0$ is the local full-rank identification condition for the weighted moment map.  It is the finite-dimensional version of \cite[Assumption~3.6]{hall2004generalized}
	and \cite[Theorem~3.4]{newey1994large}. The radius condition \cref{eq:radius-condition} makes explicit how small the local basin must be for the nonlinear score to remain strongly monotone.
\end{remark}

\begin{theoremEnd}[end, category=finite-sample-gmm]{lemma}[\cref{ass:local-identification} implies local monotonicity]\label{lem:local-monotonicity}
	Under \cref{ass:local-identification}, for every $\btheta\in\cB_0$,
	\begin{equation}\label{eq:local-monotonicity-report}
		\inner{\Psi(\btheta)}{\btheta-\btheta^\star}
		\ge \frac{\lambda^\star}{2}\normtwo{\btheta-\btheta^\star}^2.
	\end{equation}
\end{theoremEnd}

\begin{proofEnd}
	Let 
	\begin{align}
		\bDelta\coloneq\btheta-\btheta^\star, \mkern9mu  \overline G(\btheta) \coloneq \int_0^1 G(\btheta^\star+t\bDelta)\mathrm{d} t. 
	\end{align}
	Since $m(\btheta^\star)=0$ and $m$ is continuously differentiable on the line segment between $\btheta^\star$ and $\btheta$,
	\begin{equation}\label{eq:m-mean-value}
		m(\btheta)
		=\overline G(\btheta) \bDelta.
	\end{equation}
	Therefore, by the definition \eqref{eq:population-gmm-score},
	\begin{align}
		\Psi(\btheta)&= G(\btheta)^\top W m(\btheta) = G(\btheta)^\top	W \overline G(\btheta)\bDelta.\\
		\implies \inner{\Psi(\btheta)}{\bDelta}
		&=\bDelta^\top G(\btheta)^\top W\overline G(\btheta)\bDelta.		\label{eq:score-inner-proof}
	\end{align}
	By the Lipschitz condition \eqref{eq:G-lipschitz} in \cref{ass:local-identification},
	\begin{equation}\label{eq:Gbar-proof-bounds}
		\normop{G(\btheta)-G^\star}\le L_G\normtwo{\bDelta}.
	\end{equation}
	Since 
	\begin{align}
		\overline G(\btheta)  - G^\star = \int_0^1 G(\btheta^\star+t\bDelta) - G^\star \mathrm{d} t,
	\end{align}
	we get 
	\begin{align}
		\normop{\overline G(\btheta)-G^\star}\le \int_0^1 L_G t \normtwo{\bDelta} \mathrm{d} t =  \frac12 L_G\normtwo{\bDelta}.
	\end{align}
	By expanding the terms, we get 
	\begin{align*}
		\normop{G(\btheta)^\top W\overline G(\btheta)-(G^\star)^\top WG^\star} &\le
		\normop{G(\btheta)-G^\star}\normop{W}\normop{\overline G(\btheta)}
		+\normop{G^\star}\normop{W}\normop{\overline G(\btheta)-G^\star} \\
		& \quad  + \normop{G(\btheta)-G^\star}\normop{W}\normop{G^\star}\numberthis\\
		&\le
		\normop{W}L_G\normtwo{\bDelta}
		\left(\frac32\normop{G^\star}+\frac12L_G\normtwo{\bDelta}\right)
		\le \frac{\lambda^\star}{2},\numberthis
	\end{align*}
	where the final inequality is due to $\normtwo{\bDelta}\le r_0$ and the local radius condition \eqref{eq:radius-condition} in \cref{ass:local-identification}.
	
	Since $H^\star = (G^\star)^\top WG^\star, \lambda^\star = \lambda_{\min}(H^\star) > 0$ by \eqref{eq:score-inner-proof} in \cref{ass:local-identification}, we get
	\begin{align*}
		\inner{\Psi(\btheta)}{\bDelta} &= \bDelta^\top G(\btheta)^\top W\overline G(\btheta)\bDelta \numberthis\\
		&\ge
		\bDelta^\top (G^\star)^\top WG^\star\bDelta
		-\frac{\lambda^\star}{2}\normtwo{\bDelta}^2 \numberthis\\
		&\ge
		\lambda^\star\normtwo{\bDelta}^2
		-\frac{\lambda^\star}{2}\normtwo{\bDelta}^2
		=\frac{\lambda^\star}{2}\normtwo{\bDelta}^2.\numberthis
	\end{align*}
\end{proofEnd}

We state the high-probability version of \cref{ass:inlier-stability} and \cref{thm:sgr-termination-error} in the form used for the local finite-sample GMM theorem in \cref{thm:finite-sample-parameter}.
\begin{assumption}[Inlier stability conditions, high-probability version]
	\label{ass:inlier-stability-high-prob} 	
	Fix $0\leq \epsilon<1/3$ and $\zeta\in(0,1)$. With probability at least $1-\zeta$ over the inliers, the following event holds for every $k\in[L], \btheta\in\cB_0, \bw \in \Delta_{\cI_{\inlier}, \frac{\epsilon}{1-\epsilon}}$, 
	\begin{align}
		&\normtwo{\sum_{n\in \cI_{\inlier}} w_n 	\left(\check{\bg}_n^{(k)}(\btheta) - \bmu_{\bg}^{(k)}(\btheta) \right)} \leq \delta_{\mu,k}(\zeta),\\
		&\sum_{n\in\cI_{\inlier}}w_n  \left(\check{\bg}_n^{(k)}(\btheta)  - \bmu_{\bg}^{(k)}(\btheta) \right) \left(\check{\bg}_n^{(k)}(\btheta)  - \bmu_{\bg}^{(k)}(\btheta) \right)^{\top}  \preceq  \Sigma_{\bg}^{(k)}(\btheta)   + \delta_{\Sigma,k}(\zeta) I,
	\end{align}
\end{assumption}

\begin{corollary}\label{thm:sgr-termination-error-high-prob}
	Fix $\btheta\in\cB_0$ and $k\in[L]$. Under \cref{ass:inlier-stability-high-prob} and on the same event, the output $\widehat\bw^{(k)}(\btheta)$ of  \cref{alg:spectral-reweighting} satisfies
	\begin{equation}
		\normtwo{\sum_{n=1}^N\widehat w_n^{(k)}(\btheta)\check\bg_n^{(k)}(\btheta)
			-\bmu_{\bg}^{(k)}(\btheta)}
		\leq \delta_{\mu,k} (\zeta) +\alpha_\epsilon\sqrt{C_{k}}, \mkern9mu \alpha_{\epsilon} = \sqrt{\frac{\epsilon}{1-2\epsilon}},
	\end{equation}
	where $C_k$ is either of the following two cases:
	\begin{enumerate}
		\item If the stopping certificates are specified directly, then $C_k = C_{\text{stop}, k}$.
		\item If the stopping certificates are obtained from the MW-MMW rounds and the outer-loop convergence radius is \(R_k\), then it is sufficient to take
		\begin{equation}\label{eq:r-fixedbudget}
			C_k = \sup_{\btheta \in \cB_0}\normop{\Sigma_{\bg}^{(k)}(\btheta)}  +\delta_{\Sigma,k}(\zeta) + \left(\delta_{\mu,k}(\zeta) + R_k \right)^2 +\delta_{T, k},
		\end{equation}
		where, for squared diameter \(\nu_k\),
		\begin{equation}\label{eq:delta-T}
			\delta_{T, k}
			\leq 4\nu_k\left\{\sqrt{\frac{\log p}{T}}+
			\sqrt{\frac{\log(1/(1-\epsilon))}{T}}\right\},
		\end{equation}
	\end{enumerate}
\end{corollary}

\begin{assumption}[Numerical optimizer conditions]\label{ass:numerical-optimizer-conditions}
	The numerical optimizer used in \cref{alg:robust-sgr-gmm} returns a \(\widehat\btheta\) that satisfies
	\begin{equation}
		\label{eq:theta-local}
		\widehat\btheta \in\cB_0,
	\end{equation}
	and
	\begin{equation}
		\label{eq:opt-residual}
		\normtwo{\widehat\Psi^{(\sgr)}(\widehat\btheta)}\le\delta_{\opt}.
	\end{equation}
\end{assumption}

\begin{theoremEnd}[end, category=finite-sample-gmm]{theorem}[Local finite-sample parameter estimation error for \cref{alg:robust-sgr-gmm}]\label{thm:finite-sample-parameter}		
	Fix $0\leq \epsilon<1/3$ and $\zeta\in(0,1)$. Suppose that \cref{ass:local-identification}, \cref{ass:inlier-stability-high-prob}, and \cref{ass:numerical-optimizer-conditions} hold. Then the following holds with probability at least $1-\zeta$:
	\begin{align}\label{eq:main-error-bound}
		\normtwo{\widehat\btheta^{(\sgrgmm)}-\btheta^\star}
		\leq
		\frac{2}{\lambda^\star}
		\left( \underbrace{\sum_{k=1}^L a_k \left(\delta_{\mu,k}(\zeta) +\alpha_\epsilon\sqrt{C_{k}} \right)}_{\text{SGR error}} + \underbrace{\delta_{\opt} }_{\text{optimizer error}}  \right), \mkern9mu \alpha_{\epsilon} = \sqrt{\frac{\epsilon}{1-2\epsilon}},
	\end{align}
	where 
	\begin{align}
		C_{k} = \begin{cases}
			C_{\text{stop}, k}, & (\text{if $C_{\text{stop}, k}$ are specified directly})\\
			\sup_{\btheta \in \cB_0}\normop{\Sigma_{\bg}^{(k)}(\btheta)} +\delta_{\Sigma,k}(\zeta) + \left(\delta_{\mu,k}(\zeta) + R_k \right)^2   +\delta_{T, k}, & (\text{otherwise}).
		\end{cases}
	\end{align}
\end{theoremEnd}

\begin{proofEnd}
	Set $\widehat\btheta=\widehat\btheta^{(\sgrgmm)}$ and
	$\widehat{\bDelta}=\widehat\btheta-\btheta^\star$.  If $\widehat{\bDelta}=0$, the theorem is trivial, so assume
	$\widehat{\bDelta}\neq0$.  
	
	Under \cref{ass:local-identification}, by applying \cref{lem:local-monotonicity}, we get 
	\begin{equation}\label{eq:proof-monotone-start}
		\frac{\lambda^\star}{2}\normtwo{\widehat{\bDelta}}^2
		\leq \inner{\Psi(\widehat\btheta)}{\widehat{\bDelta}}
		\leq \normtwo{\Psi(\widehat\btheta)}\normtwo{\widehat{\bDelta}},
	\end{equation}
	where the second inequality is Cauchy's inequality.  It remains to bound
	$\normtwo{\Psi(\widehat\btheta)}$. 
	\begin{align}\label{eq:score-triangle}
		\normtwo{\Psi(\widehat\btheta)}
		&\leq
		\normtwo{\Psi(\widehat\btheta)-\widehat\Psi^{(\sgr)}(\widehat\btheta)}
		+\normtwo{\widehat\Psi^{(\sgr)}(\widehat\btheta)}.
	\end{align}
	Under \cref{ass:inlier-stability-high-prob}, by applying  \cref{thm:sgr-termination-error-high-prob} and using the
	triangle inequality, we get a bound on the first term in \cref{eq:score-triangle}:
	\begin{align}\label{eq:sgr-bound}
		\normtwo{\Psi(\widehat\btheta)-\widehat\Psi^{(\sgr)}(\widehat\btheta)}
		\leq \sum_{k=1}^L a_k \left(\delta_{\mu,k}(\zeta) +\alpha_\epsilon\sqrt{C_{k}} \right),
	\end{align}
	where 
	\begin{align}
		C_{k} = \begin{cases}
			C_{\text{stop}, k}, & (\text{if $C_{\text{stop}, k}$ are specified directly})\\
			\sup_{\btheta \in \cB_0}\normop{\Sigma_{\bg}^{(k)}(\btheta)} +\delta_{\Sigma,k}(\zeta) + \left(\delta_{\mu,k}(\zeta) + R_k \right)^2   +\delta_{T, k}, & (\text{otherwise}).
		\end{cases}
	\end{align}
	
	The second term in \cref{eq:score-triangle} is the optimizer residual \cref{eq:opt-residual}, which has the following bound due to \cref{ass:numerical-optimizer-conditions}
	\begin{align}
		\label{eq:optimizer-solution-bound}
		\normtwo{\widehat\Psi^{(\sgr)}(\widehat\btheta)} \leq \delta_{\opt}. 
	\end{align}
	
	Combining \cref{eq:score-triangle}, \cref{eq:sgr-bound}, and \cref{eq:optimizer-solution-bound}, we get 
	\begin{align}\label{eq:score-triangle-final}
		\normtwo{\Psi(\widehat\btheta)} \leq  	\normtwo{\Psi(\widehat\btheta)-\widehat\Psi^{(\sgr)}(\widehat\btheta)}+\normtwo{\widehat\Psi^{(\sgr)}(\widehat\btheta)}
		\leq \sum_{k=1}^L a_k \left(\delta_{\mu,k}(\zeta) +\alpha_\epsilon\sqrt{C_{k}} \right) + \delta_{\opt}. 
	\end{align}
	Combining
	\cref{eq:proof-monotone-start} and \cref{eq:score-triangle-final}, and dividing by
	$\normtwo{\widehat{\bDelta}}$, gives \cref{eq:main-error-bound}. 
\end{proofEnd}

\begin{remark}
	The local finite-sample parameter estimation error of \cref{alg:robust-sgr-gmm} has the following distinct components: 
	\begin{enumerate}
		\item The factor \(\lambda^\star\) is the local GMM identification strength.  Weak identification inflates every error term.
		\item \(\delta_{\mu,k}(\zeta)\) depends on the clean inlier stability. 
		\item \(\alpha_\epsilon\sqrt{C_k}\) is the robust reweighting contribution.  Since \(\alpha_\epsilon\asymp\sqrt\epsilon\) for small \(\epsilon\), the bound recovers the bounded-covariance robust-mean scaling; sharper \(\epsilon\sqrt{\log(1/\epsilon)}\) behavior requires stronger score-tail assumptions and a sharper robust mean primitive.
		\item \(\delta_{T,k}\) is the inner MW-MMW optimization floor.  Since \(\delta_{T,k}=O(\nu_k\sqrt{(\log p+\log(1/(1-\epsilon)))/T})\), the parameter bound contains this through \(\sqrt{C_k}\).  The outer optimizer contributes separately through \(\delta_{\opt}\).
	\end{enumerate}
\end{remark}

\begin{remark}
	\label{rem:global-local}
	\cref{thm:finite-sample-parameter} proves a local deterministic perturbation theorem.  To assert that an arbitrary run of \(\lbfgs\) reaches the correct basin of a nonconvex objective would require additional global identification conditions, e.g., \cite[discussion following Eq. (1.4)]{newey1994large} or additional landscape conditions analogous in spirit to the IRLS analysis of \cite{lermanmaunu2018fast} or the landscape analysis of \cite{cheng2020high}.  
\end{remark}

\section{Robust DGMM specialization for Gaussian mixture modeling}
\label{sec:robust-dgmm}

\subsection{Heteroscedastic low-rank GMs under additive noise and adversarial contamination}

\begin{model}[Heteroscedastic low-rank GMs] \label{model:heteroscedastic}
	Fix the number of mixture components $K \geq 2$. Let $h \in [K]$ be a discrete random variable such that $0 < \probP(h = j) = \pi_j < 1$ for $j = 1, \dots, K$ and $\sum_{j=1}^K \pi_j = 1$, where $\pi_j$ is the mixing probability of the $j$-th mixture component. Conditional on $h = j$, each mixture component $\bX_j \in \R^d$ is a \emph{\textbf{low-rank Gaussian}}, that is, 
	\begin{align}
		\bX \mid (h = j) \sim \cN(\bmu_j, \Sigma_j), \mkern9mu  R_j \coloneqq \rank \Sigma_j  \leq R_{\max} \coloneqq \max\{R_1,\dots,R_K\} \leq d.
	\end{align}
	
	Then the random vector $\bY$ drawn as $\bX_h$ is said to be a \emph{\textbf{heteroscedastic low-rank Gaussian Mixtures (GMs)}} and has the following data generating process:
	\begin{equation}\label{eq:data-gen-heteroscedastic}
		{\by_n} = {\bmu}_{h} + \bm{\Xi}_{n}, \mkern9mu \bm{\Xi}_{n} \iid \cN(\bm{0},\Sigma_{h}), \mkern9mu n = 1, \dots, N.
	\end{equation}
	In addition, if the components $\bX_j$ are weakly separated, i.e., $\norm{\Sigma_j}_F \gg \norm{{\bmu}_j}_2$, then the random vector $\bY$ drawn as $\bX_h$ is said to be a \emph{\textbf{weakly separated heteroscedastic low-rank GMs}}.
\end{model}
\par First, we introduce additive noise to \cref{model:heteroscedastic}: 
\begin{model}[Additive noise model]
	\label{model:heteroscedastic-additive}
	We say that a mixture variable $ \bY \sim$~\cref{model:heteroscedastic} is observed in the presence of independent additive noise $\bxi \indep \bX_h$, if 
	\begin{align}
		\widetilde{\bY} = \bY+ \bxi.
	\end{align}
	$\widetilde{\bY}$ has the following data generating process:
	\begin{equation}\label{eq:data-gen-heteroscedastic-additive-noise}
		\widetilde{\by}_n = \by_n + \bxi_n  = {\bmu}_{h} + \bm{\Xi}_{n} + \bxi_n,  \mkern9mu \bm{\Xi}_{n} \iid \cN(\bm{0},\Sigma_{h}), \bxi_n \iid \cD_{\bxi}.
	\end{equation}
\end{model}
In the scope of this paper, we will assume that the additive noise distribution is known a priori and that the additive noise is Gaussian-distributed: $\cD_{\bxi} = \cN(\bm{0},\Sigma_{\bxi})$, where the noise covariance $\Sigma_{\bxi})$ is known and symmetric positive semidefinite.

\par  Then, we introduce strong $\epsilon$-contamination to \cref{model:heteroscedastic-additive}:

\begin{model}[Strong contamination model] 
	\label{model:heteroscedastic-strong-contamination}
	Given a parameter $0 \leq \epsilon < \frac{1}{3}$, we say that an additive-noise mixture variable defined in \cref{model:heteroscedastic-additive} is observed in the presence of \emph{\textbf{strong $\epsilon$-contamination}} if there is an adversary that inspects the sample \text{i.i.d.} $\{\widetilde{\by}_n\}_{n=1}^N \iid$ \cref{model:heteroscedastic-additive} and corrupts up to $\epsilon N$ number of points by replacing them by arbitrary points in $\R^d$. The data generating process is the following:
	\begin{align}
		\check{\by}_n = \begin{cases*}
			\widetilde{\by}_n, & (if $n \in \cI_{\inlier}$)\\
			\ba_n \in \R^d, & (if $n \in  \cI_{\outlier}$)
		\end{cases*},
	\end{align}
	with a partition $\{1, \dots, N\} = \cI_{\inlier} \sqcup \cI_{\outlier}, |\cI_{\outlier}| \leq \epsilon N$.
\end{model}

\subsection{Robust DGMM}
We first compute the objective function and gradients required for the robust DGMM estimation. We will refer to $\check{\phi}^{(k)} (\btheta; \bxi)$ as the \textbf{``model term''} and the Bell polynomials of the type in \cref{eq:alpha-recurrence} as the ``model-term Bell polynomials,'' to highlight that their evaluation only involve model parameters. Similarly, we will refer to $\check{\psi}^{(k)} (\btheta; \bxi,  \check{\by}_{n})$ as the \textbf{``per-observation cross term''} and the Bell polynomials of the type in \cref{eq:beta-recurrence} as the ``cross-term Bell polynomials,'' since their evaluation requires both the model parameters and the sample data. Note that $\widehat{o}^{[t]}_k$ and $\frac{1}{N^2} \sum_{n=1}^N \sum_{n^{\prime}=1}^N C_{k,n,n^\prime}$ are both constants that remain unchanged during the L-BFGS optimization. 

\subsubsection{Model term}
Write $\btheta=(\pi_1,\ldots,\pi_K; \mu_1,\ldots, \mu_K; V_1, \ldots, V_K)$, where $\pi_j\in\R$, $\mu_j\in\R^d$, and $V_j\in\R^{d\times R_j}$.
Define the order-$k$ model term:
\begin{align}\label{eq:model-term-def}
	\check{\phi}^{(k)} (\btheta; \bxi) &= \sum_{i=1}^K \sum_{j=1}^K \pi_i \pi_j B_{k}\left(	\left(\check{\kappa}^{(1)}_{ij}\right)  , \dots, 	\left(\check{\kappa}^{(k)}_{ij} \right)  \right),\\
	\check{\kappa}^{(\ell)}_{ij}  &=
	\begin{cases*}\label{eq:cumulants}
		\left\langle{\bmu}_j, {\bmu}_i \right\rangle, & ($l = 1$) \\
		(l-1)! \Tr\left[\left(\left(V_i {V_i}^\top + \Sigma_{\bxi}\right) \left(V_j {V_j}^\top + \Sigma_{\bxi}\right)\right)^{\frac{l}{2}}\right]\\ \quad + \frac{l!}{2}  {{\bmu}_i}^\top \left(V_j {V_j}^\top + \Sigma_{\bxi}\right)  \left(\left(V_i {V_i}^\top + \Sigma_{\bxi}\right) \left(V_j {V_j}^\top + \Sigma_{\bxi}\right)\right)^{\frac{l-2}{2}} {\bmu}_i \\
		\quad + \frac{l!}{2} {{\bmu}_j}^\top \left(\left(V_i {V_i}^\top + \Sigma_{\bxi}\right) \left(V_j {V_j}^\top + \Sigma_{\bxi}\right)\right)^{\frac{l-2}{2}}\left(V_i {V_i}^\top + \Sigma_{\bxi}\right) {\bmu}_j, & ($l$ is even) \\
		l! {{\bmu}_j}^\top \left(\left(V_i {V_i}^\top + \Sigma_{\bxi}\right) \left(V_j {V_j}^\top + \Sigma_{\bxi}\right)\right)^{\frac{l-1}{2}} {\bmu}_i, & ($l$ is odd) \\
	\end{cases*}
\end{align}
where the model-term Bell polynomials respect the following recurrence relation
\begin{align}
	\label{eq:alpha-recurrence}
	\begin{cases*}
		B_{0}\left( \check{\kappa}^{(1)}_{ij}, \dots, \check{\kappa}^{(k)}_{ij} \right) = 1, & (base case)\\
		B_{k}\left( \check{\kappa}^{(1)}_{ij}, \dots, \check{\kappa}^{(k)}_{ij} \right) = \sum_{\ell=0}^{k-1} {k-1 \choose \ell}B_{k-\ell-1}\left(\check{\kappa}^{(1)}_{ij}, \dots, \check{\kappa}^{(k-\ell-1)}_{ij}\right) \check{\kappa}^{(\ell+1)}_{ij}. & (induction step)
	\end{cases*}
\end{align}

\subsubsection{Per-observation cross term}
Fix an observation $\check{\by}_n\in\R^d$ and define the order-$k$ per-observation cross term:
\begin{equation}\label{eq:cross-term-def}
	\check{\psi}^{(k)}(\btheta; \bxi, \check{\by}_n) =  \sum_{j=1}^{K} \pi_j 
	B_{k}\!\Big(
	{\check{\by}_n}^{\top}\bmu_j,\;
	{\check{\by}_n}^{\top}\left(V_jV_j^\top  + \Sigma_{\bxi} \right){\check{\by}_n},\;
	0,\ldots,0
	\Big), 
\end{equation}
where the cross-term Bell polynomials respect the following recurrence relation
\begin{align}
	\label{eq:beta-recurrence}
	\begin{cases*}
		B_{0}\left({{\check{\by}_n}}^\top {\bmu}_j, {{\check{\by}_n}}^\top \left(V_jV_j^\top  + \Sigma_{\bxi} \right) {\check{\by}_n}, 0, \dots, 0\right) = 1, & (base case)\\
		B_{1}\left({{\check{\by}_n}}^\top {\bmu}_j, {{\check{\by}_n}}^\top \left(V_jV_j^\top  + \Sigma_{\bxi} \right) {\check{\by}_n}, 0, \dots, 0\right) = {{\check{\by}_n}}^\top {\bmu}_j, & (base case) \\
		B_{k}\left({{\check{\by}_n}}^\top {\bmu}_j, {{\check{\by}_n}}^\top \left(V_jV_j^\top  + \Sigma_{\bxi} \right) {\check{\by}_n}, 0, \dots, 0\right) =  \\
		\quad \quad \quad \quad \quad \mkern9mu B_{k-1}\left({{\check{\by}_n}}^\top {\bmu}_j, {{\check{\by}_n}}^\top \left(V_jV_j^\top  + \Sigma_{\bxi} \right) {\check{\by}_n}, 0, \dots, 0\right)  {{\check{\by}_n}}^\top {\bmu}_j \\
		\quad \quad \quad \quad \quad \mkern9mu + (k-1) B_{k-2}\left({{\check{\by}_n}}^\top {\bmu}_j, {{\check{\by}_n}}^\top \left(V_jV_j^\top  + \Sigma_{\bxi} \right) {\check{\by}_n}, 0, \dots, 0\right) {{\check{\by}_n}}^\top \left(V_jV_j^\top  + \Sigma_{\bxi} \right) {\check{\by}_n}. & (induction step)
	\end{cases*}
\end{align}

\subsubsection{Robust DGMM objective function}
After computing the model term and the cross term, we then get the robust DGMM objective evaluated at the observed points $\{\check{\by}_n\}_{n=1}^N$ from \cref{model:heteroscedastic-strong-contamination} by substituting the noisy, $\epsilon$-contaminated model terms $\check{\phi}^{(k)}(\btheta; \bxi)$ and per-observation cross terms $\check{\psi}^{(k)}(\btheta; \bxi, \check{\by}_n)$ into the DGMM objective at $t$-th GMM estimation step \cite[Eq.~(31)]{zhang2025diagonally}:
\begin{align}
	\label{eq:dgmm-objective}
	Q^{[t]}_N(\btheta)  &= \sum_{k=1}^L \widehat{o}^{[t]}_k  \left(\check{\phi}^{(k)} (\btheta; \bxi) -2 \sum_{n=1}^N \widehat{w}^{(k)}_n \check{\psi}^{(k)}(\btheta; \bxi, \check{\by}_n)+ \sum_{n=1}^N \sum_{n^{\prime}=1}^N \widehat{w}^{(k)}_n \widehat{w}^{(k)}_{n^{\prime}}  C_{k,n,n^\prime}  \right),
\end{align} 
and the following quantities that are frozen during the optimization: 
\begin{enumerate}
	\item the pre-computed moment sum $ \sum_{n^{\prime}=1}^N C_{k,n,n^\prime}$ and $C_{k,n,n}$, where $C_{k,n,n^\prime} = \inner{{\check{\by}_n}}{\check{\by}_{n^\prime}}^k$, 
	\item the weight vector on the per-observation gradients  $\widehat{\bw}^{(k)}$ $\in \Delta_{N,\epsilon}$ obtained from \cref{alg:spectral-reweighting},
	\item the robust order-specific DGMM weights:
	\begin{align*}
		\widehat{o}^{[t]}_k 
		&= \frac{  \sum_{n} \left(\widehat{w}^{(k)}_n\right)^2 \left(\check{\phi}^{(k)}(\btheta; \bxi) - 2  \check{\psi}^{(k)}(\btheta; \bxi, \check{\by}_n) +  C_{k, n, n}\right)}{\begin{matrix}\sum_{n, n^\prime} \sum_{k^\prime} 
				\widehat{w}^{(k)}_n \widehat{w}^{(k)}_{n^{\prime}}
				\widehat{w}^{(k^\prime)}_{n}   
				\widehat{w}^{(k^\prime)}_{n^{\prime}}  
				\left(\check{\phi}^{(k)}(\btheta; \bxi)  - \check{\psi}^{(k)}(\btheta; \bxi, \check{\by}_n) - \check{\psi}^{(k)}(\btheta; \bxi, \check{\by}_{n^\prime})+   C_{k,n,n^\prime}\right)&\\	\quad \quad \quad\quad\quad \quad \left( \check{\phi}^{(k^\prime)}(\btheta; \bxi)  - \check{\psi}^{(k^\prime)}(\btheta; \bxi, \check{\by}_n)- \check{\psi}^{(k^\prime)}(\btheta; \bxi, \check{\by}_{n^\prime}) + C_{k^\prime, n, n^\prime}\right) \end{matrix}}.\numberthis
	\end{align*}
\end{enumerate}

\subsubsection{Gradients of \texorpdfstring{$\check{\phi}^{(k)}(\btheta; \bxi)$}{} and \texorpdfstring{$\check{\psi}^{(k)}(\btheta; \bxi, \check{\by}_n)$}{}}

First, recall the following fact. For $k\in\mathbb{N}$, let $B_k:\R^k\to\R$ denote the (exponential) Bell polynomial. For brevity, denote $\bkappa=(\kappa^{(1)},\dots,\kappa^{(k)})\in\R^k$. We get
\begin{align}
	\frac{\partial B_k(\bkappa)}{\partial \kappa^{(\ell)}}&=\binom{k}{\ell}\,B_{k-\ell}(\bkappa),\qquad \ell\in[k],\label{eq:bell-first}\\
	\frac{\partial^2 B_k(\bkappa)}{\partial \kappa^{(\ell_1)}\,\partial \kappa^{(\ell_2)}}&=\binom{k}{\ell_1}\binom{k-\ell_1}{\ell_2}\,B_{k-\ell_1-\ell_2}(\bkappa),\qquad \ell_1,\ell_2\in[k],\label{eq:bell-second}
\end{align}
and $\frac{\partial B_k}{\partial \bkappa^{(\ell)}}\equiv 0$ when $\ell>k$.
\begin{align*} 
	\nabla_{\pi_j} \check{\phi}^{(k)}(\btheta; \bxi) &= 2 \sum_{i=1}^K \pi_i B_{k} \left(\check{\kappa}^{(1)}_{ij},\dots, 
	\check{\kappa}^{(k)}_{ij}\right),\numberthis\\
	\nabla_{{\bmu}_j}\check{\phi}^{(k)}(\btheta; \bxi) &= 2 \sum_{i=1}^K   \pi_i \pi_j \sum_{l=1}^k \binom{k}{l} B_{k-l} \left(\check{\kappa}^{(1)}_{ij},\dots, 
	\check{\kappa}^{(k-l)}_{ij}\right) \nabla_{{\bmu}_j} \check{\kappa}^{(l)}_{ij},\numberthis\\
	\nabla_{V_j}  \check{\phi}^{(k)}(\btheta; \bxi) &= 2 \sum_{i=1}^K \pi_i \pi_j \sum_{l=1}^k \binom{k}{l} B_{k-l} \left(\check{\kappa}^{(1)}_{ij},\dots, 
	\check{\kappa}^{(k-l)}_{ij}\right) \nabla_{V_j} \check{\kappa}^{(l)}_{ij}.\numberthis
\end{align*}
From the cumulants in \cref{eq:cumulants}, we get the gradients of the cumulants w.r.t.~${\bmu}_j$ and $V_j$:
\begin{equation}
	\label{eq:cumulants-gradients-mu}
	\nabla_{{\bmu}_j} \check{\kappa}^{(l)}_{ij} = 
	\begin{cases*}
		{\bmu}_i , & ($l = 1$) \\
		l! \left(\left(V_i {V_i}^\top + \Sigma_{\bxi}\right)  \left(V_j {V_j}^\top + \Sigma_{\bxi}\right) \right)^{\frac{l-2}{2}} \left(V_i {V_i}^\top + \Sigma_{\bxi}\right)  {\bmu}_j, & ($l$ is even) \\
		l! \left(\left(V_i {V_i}^\top + \Sigma_{\bxi}\right)  \left(V_j {V_j}^\top + \Sigma_{\bxi}\right) \right)^{\frac{l-1}{2}} {\bmu}_i, & ($l$ is odd) \\
	\end{cases*}
\end{equation}

{\scriptsize
	\begin{align*}
		\label{eq:cumulants-gradients-v}
		\nabla_{V_j} \check{\kappa}^{(l)}_{ij} = &\begin{cases*}
			0 , & ($l = 1$) \\
			l! \left(\left(V_i {V_i}^\top + \Sigma_{\bxi}\right)  \left(V_j {V_j}^\top + \Sigma_{\bxi}\right) \right)^{\frac{l-2}{2}} \left(V_i {V_i}^\top + \Sigma_{\bxi}\right)  V_j \\
			\quad + l! \sum_{p=0}^{\frac{l-2}{2}-1} \left(\left(V_i {V_i}^\top + \Sigma_{\bxi}\right)  \left(V_j {V_j}^\top + \Sigma_{\bxi}\right) \right)^p {\bmu}_i {{\bmu}_i}^\top \left(\left(V_j {V_j}^\top + \Sigma_{\bxi}\right)  \left(V_i {V_i}^\top + \Sigma_{\bxi}\right) \right)^{\frac{l-2}{2} - p} V_j \\
			\quad +  l! \sum_{p=0}^{\frac{l-2}{2}-1} \left(V_i {V_i}^\top + \Sigma_{\bxi}\right)  \left(\left(V_j {V_j}^\top + \Sigma_{\bxi}\right)  \left(V_i {V_i}^\top + \Sigma_{\bxi}\right) \right)^p {\bmu}_j {{\bmu}_j}^\top \left(\left(V_i {V_i}^\top + \Sigma_{\bxi}\right)  \left(V_j {V_j}^\top + \Sigma_{\bxi}\right) \right)^{\frac{l-2}{2}-1 - p} \left(V_i {V_i}^\top + \Sigma_{\bxi}\right)  V_j, & ($l$ is even) \\
			l! \sum_{p=0}^{\frac{l-1}{2}-1} \left(V_i {V_i}^\top + \Sigma_{\bxi}\right)  \left(\left(V_j {V_j}^\top + \Sigma_{\bxi}\right)  \left(V_i {V_i}^\top + \Sigma_{\bxi}\right) \right)^p {\bmu}_j {{\bmu}_i}^\top \left(\left(V_j {V_j}^\top + \Sigma_{\bxi}\right)  \left(V_i {V_i}^\top + \Sigma_{\bxi}\right) \right)^{\frac{l-1}{2}-1-p} V_j \\
			\quad + l!  \sum_{p=0}^{\frac{l-1}{2}-1}  \left(\left(V_i {V_i}^\top + \Sigma_{\bxi}\right)  \left(V_j {V_j}^\top + \Sigma_{\bxi}\right) \right)^p {\bmu}_i {{\bmu}_j}^\top \left(\left(V_i {V_i}^\top + \Sigma_{\bxi}\right)  \left(V_j {V_j}^\top + \Sigma_{\bxi}\right) \right)^{\frac{l-1}{2}-1-p} \left(V_i {V_i}^\top + \Sigma_{\bxi}\right)  V_j. & ($l$ is odd)
		\end{cases*}\numberthis
	\end{align*}
}

For $h\in[K]$:
\begin{align*}
	\nabla_{\pi_h}\kappa^{(\ell)}_j(\check{\by}_n)&=0\qquad \text{(all $\ell$)},\\
	\nabla_{\mu_h}\kappa^{(1)}_j(\check{\by}_n)&= \1_{\{h = j\}} \check{\by}_n,\quad 
	\nabla_{\mu_h}\kappa^{(2)}_j(\check{\by}_n)=0,\quad 
	\nabla_{\mu_h}\kappa^{(\ell)}_j(\check{\by}_n)=0\ \ (\ell\ge 3),\\
	\nabla_{V_h}\kappa^{(1)}_j(\check{\by}_n)&=0,\qquad 
	\nabla_{V_h}\kappa^{(2)}_j(\check{\by}_n)=\1_{\{h = j\}}\,2\,\check{\by}_n\check{\by}_n^\top V_h,\qquad 
	\nabla_{V_h}\kappa^{(\ell)}_j(\check{\by}_n)=0\ \ (\ell\ge 3).
\end{align*}
Note that only $h=j$ contributes, only $(1)$ depends on $\mu_h$ and only $(2)$ depends on $V_h$. Higher cumulants vanish for Gaussian random varaibles. Using \cref{eq:bell-first} and the chain rule on \cref{eq:cross-term-def}:
\begin{align*}
	\nabla_{\pi_h}\check{\psi}^{(k)}(\btheta; \bxi, \check{\by}_n)
	&=\sum_{j=1}^K \nabla_{\pi_h} \left( \pi_j
	B_{k}\!\Big(
	{\check{\by}_n}^{\top}\bmu_j,\;
	{\check{\by}_n}^{\top}\left(V_jV_j^\top + \Sigma_{\bxi}\right){\check{\by}_n},\;
	0,\ldots,0
	\Big) \right)\\
	&=
	B_{k}\!\Big(
	{\check{\by}_n}^{\top}\bmu_h,\;
	{\check{\by}_n}^{\top}\left(V_h V_h^\top + \Sigma_{\bxi}\right){\check{\by}_n},\;
	0,\ldots,0
	\Big)  \text{(only $h=j$ survives; $B_k$ independent of $\pi_h$)}\label{eq:cross-term-first-derivatives-pi}\numberthis
\end{align*}

\begin{align*}
	\nabla_{\mu_h}\check{\psi}^{(k)}(\btheta; \bxi, \check{\by}_n) &=\sum_{j=1}^K \pi_j \sum_{m=1}^k \frac{\partial 
		B_{k}\!\Big(
		{\check{\by}_n}^{\top}\bmu_j,\;
		{\check{\by}_n}^{\top}\left(V_jV_j^\top + \Sigma_{\bxi}\right){\check{\by}_n},\;
		0,\ldots,0
		\Big) }{\partial \kappa^{(m)}_j(\check{\by}_n)}\,\nabla_{\mu_h}\kappa^{(\ell)}_j(\check{\by}_n)\\
	&= \pi_h \binom{k}{1}  B_{k-1}\!\Big(
	{\check{\by}_n}^{\top}\bmu_h,\;
	{\check{\by}_n}^{\top}\left(V_hV_h^\top + \Sigma_{\bxi}\right){\check{\by}_n},\;
	0,\ldots,0
	\Big) \,\check{\by}_n,\label{eq:cross-term-first-derivatives-mu}\numberthis
\end{align*}

\begin{align*}
	\nabla_{\vectorize{V_h}}\check{\psi}^{(k)}(\btheta; \bxi, \check{\by}_n)&=\sum_{j=1}^K \pi_j \sum_{m=1}^k \frac{\partial 
		B_{k}\!\Big(
		{\check{\by}_n}^{\top}\bmu_j,\;
		{\check{\by}_n}^{\top}\left(V_jV_j^\top + \Sigma_{\bxi}\right){\check{\by}_n},\;
		0,\ldots,0
		\Big)  }{\partial \kappa^{(m)}_j(\check{\by}_n)}\,\nabla_{V_h}\kappa^{(\ell)}_j(\check{\by}_n)\\
	&= \pi_h \binom{k}{2}  B_{k-2}\!\Big(
	{\check{\by}_n}^{\top}\bmu_j,\;
	{\check{\by}_n}^{\top}\left(V_jV_j^\top + \Sigma_{\bxi}\right){\check{\by}_n},\;
	0,\ldots,0
	\Big) \,\left(2\, \vectorize\left(\check{\by}_n\check{\by}_n^\top V_h\right)\right)\\
	&= \pi_h k (k-1)  \vectorize\left(\check{\by}_n\check{\by}_n^\top V_h\right) B_{k-2}\!\Big(
	{\check{\by}_n}^{\top}\bmu_j,\;
	{\check{\by}_n}^{\top}\left(V_jV_j^\top + \Sigma_{\bxi}\right){\check{\by}_n},\;
	0,\ldots,0
	\Big) \label{eq:cross-term-first-derivatives-V}\numberthis
\end{align*}
In \cref{eq:cross-term-first-derivatives-mu}, among $\kappa_h^{(m)}, m = 1,\dots, k$, only $\kappa_h^{(1)}$ depends on $\mu_h$; in \cref{eq:cross-term-first-derivatives-V}, among $\kappa_h^{(m)}, m = 1,\dots, k$, only $\kappa_h^{(2)}$ depends on $V_h$. Concatenating these gradients, we get 
\begin{align*}
	\label{eq:adv-gradient}
	\check{\bg}_n^{(k)} \coloneq  \nabla_{\btheta} \check{\psi}^{(k)}(\btheta; \bxi, \check{\by}_n)  = \big[& \nabla_{\pi_1} \check{\psi}^{(k)}(\btheta; \bxi, \check{\by}_n); \dots; \nabla_{\pi_K} \check{\psi}^{(k)}(\btheta; \bxi, \check{\by}_n);   \nabla_{\bmu_1} \check{\psi}^{(k)}(\btheta; \bxi, \check{\by}_n); \dots; \nabla_{\bmu_K} \check{\psi}^{(k)}(\btheta; \bxi, \check{\by}_n);   \\
	&\vectorize\left( \nabla_{V_1} \check{\psi}^{(k)}(\btheta; \bxi, \check{\by}_n)\right) ; \dots;  \vectorize\left( \nabla_{V_K} \check{\psi}^{(k)}(\btheta; \bxi, \check{\by}_n)\right) \big]^{\top} \in \R^p. \numberthis
\end{align*}

\begin{algorithm}[H]
	\renewcommand{\thealgocf}{RobustDGMM}
	\begingroup
	\scriptsize
	\caption{Robust DGMM estimation for Gaussian mixture modeling.}	\label{alg:robust-dgmm}
	\SetKwInput{KwInput}{Input} 
	\SetKwInput{KwOutput}{Output}
	\KwInput{\begin{itemize}
			\item noisy, $\epsilon$-contaminated observations $\{\check{\by}_n\}_{n=1}^N \sim$ \cref{model:heteroscedastic-strong-contamination},
			\item hyperparameters: the maximum moment order $L$, the maximum covariance rank $R_{\max}$, the number of mixture components $K$, Gaussian additive noise covariance $\Sigma_{\bxi}$, the maximum DGMM steps $T_{\dgmm}$, the maximum L-BFGS iterations $I_{\lbfgs}$, contamination fraction $\epsilon \in (0, 1/3)$, the MW-MMW step sizes $0<\eta_{\rho}, \eta_{w} \leq 1/2$, the inner iterations $T$, threshold constant $C>0$, target accuracy $\delta > 0$, reweighting interval $I_{\interval}$.
	\end{itemize}}
	\KwOutput{estimated parameters  $\widehat{\btheta} \coloneq \left[\widehat{\pi}_1; \dots; \widehat{\pi}_K; \widehat{\bmu}_1; \dots; \widehat{\bmu}_K; \vectorize(\widehat{V}_1); \dots; \vectorize(\widehat{V}_K)\right]^{\top} \in \Theta \subset \R^{p}.$}
	\DontPrintSemicolon
	
	Initialize $\btheta^{[0]}$ as per \cite{zhang2025diagonally}: $\pi^{[0]}_j = \frac{1}{K}$, ${\bmu}^{[0]}_j \sim \mathrm{Unif}\left(\{x \in \R^d : \norm{x}_2 = 1\}\right)$, $\Sigma^{[0]}_j = U^{[0]}_j {U^{[0]}_j}^T$, where $U^{[0]}_j$ is a $d\times R_{\max}$ random orthonormal matrix. \;
	
	Pre-compute the moment sum $ \sum_{n^{\prime}=1}^N C_{k,n,n^\prime}$ and $C_{k,n,n}$, where $C_{k,n,n^\prime} = \inner{{\check{\by}_n}}{\check{\by}_{n^\prime}}^k$.\;

	For each moment order $k$, compute model terms and cross terms at initialization: 

	\begin{align*}
		\check{\phi}^{(k)} (\btheta^{[0]}; \bxi) = \sum_{i=1}^K \sum_{j=1}^K \pi_i^{[0]} \pi_j^{[0]} B_{k}\left(	\left(\check{\kappa}^{(1)}_{ij}\right)^{[0]}  , \dots, 	\left(\check{\kappa}^{(k)}_{ij} \right)^{[0]}  \right),\mkern9mu
		\check{\psi}^{(k)}(\btheta^{[0]}; \bxi, \check{\by}_n) =   \sum_{j=1}^{K} \pi_j^{[0]} 
		B_{k}\!\Big(
		{\check{\by}_n}^{\top}\bmu_j^{[0]} ,\;
		{\check{\by}_n}^{\top}\left(V_j^{[0]}  {V_j^{[0]} }^\top + \Sigma_{\bxi}\right){\check{\by}_n},\;
		0,\ldots,0
		\Big).
	\end{align*}\label{alg:robust-dgmm-line-init} \;
	
	\For{$t=1,\dots,T_{\dgmm}$ or until DGMM steps converge}{
		Run the softmax reparameterized unconstrained moment-matching optimization via L-BFGS as per \cite{zhang2025diagonally}: 
		\For{$i = 1, \dots, I_{\lbfgs}$ or until L-BFGS iterations converge}{
			For each moment order $k$, evaluate the noisy, $\epsilon$-contaminated per-observation cross terms $	\check{\psi}^{(k)}(\btheta; \bxi, \check{\by}_n)$ at the noisy, $\epsilon$-contaminated observations $\check{\by}_n \sim$ \cref{model:heteroscedastic-strong-contamination}, using the current parameter estimates $   \btheta \coloneq\left[\pi_1; \dots; \pi_K; {\bmu}_1; \dots; {\bmu}_K; \vectorize(V_1); \dots; \vectorize(V_K)\right]^T \in \Theta \subset \R^{p}$:
			\begin{align*}
				\check{\phi}^{(k)} (\btheta; \bxi) = \sum_{i=1}^K \sum_{j=1}^K \pi_i \pi_j B_{k}\left(	\check{\kappa}^{(1)}_{ij}  , \dots, 	\check{\kappa}^{(k)}_{ij}  \right), \mkern9mu
				\check{\psi}^{(k)}(\btheta; \bxi, \check{\by}_n) =  \sum_{j=1}^{K} \pi_j 
				B_{k}\!\Big(
				{\check{\by}_n}^{\top}\bmu_j ,\;
				{\check{\by}_n}^{\top}\left(V_j  {V_j}^\top + \Sigma_{\bxi}\right){\check{\by}_n},\;
				0,\ldots,0
				\Big).
			\end{align*}	\label{alg:robust-dgmm-line-opt} \;
			
			For each moment order $k$, compute $\check{\bg}_n^{(k)} \coloneq \nabla_{\btheta} \check{\psi}^{(k)}(\btheta; \bxi, \check{\by}_n)$, given in \cref{eq:cross-term-first-derivatives-pi}, \cref{eq:cross-term-first-derivatives-mu}, \cref{eq:cross-term-first-derivatives-V}.\;
			
			\If{$i - i_{\prev} \geq I_{\interval}$ or $i - i_{\prev} \geq I_{\min}$ and L-BFGS is locally stabilized \label{alg:robust-dgmm-line-lbfgs}}{
				Update the weight vector on the per-observation gradients  $\widehat{\bw}^{(k)}$ $\in \Delta_{N,\epsilon}$ for each moment order $k$ via  \cref{alg:spectral-reweighting}. Note that in practice, \cref{alg:spectral-reweighting} should be initialized with the previous weight vector (warm-start). \;
				
				Update the robust order-specific DGMM weights $\{\widehat{o}^{[t]}_k\}_{k\le L}$ (the robust version of $\widehat{w}^{[t]}_k$ in \cite{zhang2025diagonally}) for each moment order $k$:
				\begin{align*}
					\widehat{o}^{[t]}_k 
					&= \frac{  \sum_{n} \left(\widehat{w}^{(k)}_n\right)^2 \left(\check{\phi}^{(k)}(\btheta; \bxi) - 2  \check{\psi}^{(k)}(\btheta; \bxi, \check{\by}_n) +  C_{k, n, n}\right)}{\begin{matrix}\sum_{n, n^\prime} \sum_{k^\prime} \widehat{w}^{(k)}_n 
							\widehat{w}^{(k)}_{n^{\prime}}
							\widehat{w}^{(k^\prime)}_{n}   
							\widehat{w}^{(k^\prime)}_{n^{\prime}}  
							\left(\check{\phi}^{(k)}(\btheta; \bxi)  - \check{\psi}^{(k)}(\btheta; \bxi, \check{\by}_n) - \check{\psi}^{(k)}(\btheta; \bxi, \check{\by}_{n^\prime})+   C_{k,n,n^\prime}\right)&\\	\quad \quad \quad\quad\quad \quad \left( \check{\phi}^{(k^\prime)}(\btheta; \bxi)  - \check{\psi}^{(k^\prime)}(\btheta; \bxi, \check{\by}_n)- \check{\psi}^{(k^\prime)}(\btheta; \bxi, \check{\by}_{n^\prime}) + C_{k^\prime, n, n^\prime}\right) \end{matrix}}.
				\end{align*}\;
				Reset L-BFGS memory and continue.\;
			}
			Freezing $\widehat{\bw}^{(k)}$ and $\widehat{o}^{[t]}_k $, continue the L-BFGS iterations using the robust gradient for the moment-matching optimization:
			\begin{align*}
				\nabla_{\btheta} Q^{[t]}_N(\btheta)  &= \sum_{k=1}^L \widehat{o}^{[t]}_k  \left( \nabla_{\btheta} \check{\phi}^{(k)} (\btheta; \bxi) -2 \sum_{n=1}^N \widehat{w}^{(k)}_n \nabla_{\btheta} \check{\psi}^{(k)}(\btheta; \bxi, \check{\by}_n)  \right).
			\end{align*}
		}
		Use $\widehat{\btheta}^{[t]}$ to initialize the next $(t+1)$-th DGMM estimation step.\;
	}
	\endgroup
\end{algorithm}

\section{Numerical experiments}
\label{sec:numerical-experiments}
The code and data used in this section are available at
\url{https://github.com/liu-lzhang/sgr-gmm}.
\subsection{Numerical experiments for \cref{alg:spectral-reweighting}}
\label{sec:reweight-gradients-numerical-experiments}
\paragraph{Data-generating model.} We first isolate the spectral gradient reweighting primitive \cref{alg:spectral-reweighting} from the nonconvex DGMM optimization problem and test its performance  on synthetic contaminated per-observation gradients $\{\check{\bg}_n\}_{n=1}^N \subset \R^p$, where $p=10$ and $N=600$. We suppress the superscript $(k)$ to keep the notation light. Here we make the following assumptions:
\begin{itemize}
	\item \textbf{Inliers.} The inliers are drawn from a Gaussian model
	\begin{align*}
		\{\check{\bg}_n\}_{n\in \cI_{\inlier}} \sim \mathcal{N}(	\bmu_{\bg} ,	\Sigma_{\bg}  ),
	\end{align*}
	with population mean 
	\begin{align*}
		\bmu_{\bg} = (0.50,-0.50,0.25,0,0.75,0,\dots,0)\in \R^{10},
	\end{align*}
	and diagonal population covariance
	\begin{align*}
		\Sigma_{\bg} = \diag\left( e^{-t_1},\dots,e^{-t_{10}}\right), \mkern9mu t_j = \frac{2(j-1)}9.
	\end{align*}
	Thus $\normop{\Sigma_{\bg}} = 1$ and the eigenvalues decay exponentially. 
	
	\item \textbf{Outliers.} We consider \emph{directional} outliers with default outlier strength is $\sigma=8$: directional contamination along the smallest-eigenvalue direction $v_{\min}$, with outliers drawn from $\cN(\bmu_{\bg}+\sigma v_{\min},0.1I_p)$.
\end{itemize}
All computations use fixed random seeds for reproducibility. Each plotted point averages 50 independent repetitions and reports one empirical standard deviation. 

\paragraph{Metrics.}
Let $\widehat{\bw}$ be the final weight vector returned by \cref{alg:spectral-reweighting} and $\overline{\check{\bg}_{\widehat{\bw}}} \coloneq  \frac{1}{N} \sum_{n=1}^N 
w_n \check{\bg}_n$ be the resulting weighted mean. To measure the performance of \cref{alg:spectral-reweighting}, we compute the following quantities:
\begin{enumerate}
	\item \textbf{Estimation error:} $\err(\overline{\check{\bg}_{\widehat{\bw}}} ) \coloneq \norm{\overline{\check{\bg}_{\widehat{\bw}}} -\bmu_{\bg}}_2$, which measures statistical accuracy.
	\item \textbf{Residual outlier mass:} $\tau_{\mathrm{out}}(\bw) \coloneq	\sum_{n\in \cI_{\mathrm{out}}} w_n,$ which measures whether the reweighting has removed the corrupted mass
	\item \textbf{Fixed-center spectral norm:} $\normop{S(\bw;\{\check{\bg}_n\},\widehat{\bmu}^{[s]})}$ for  current fixed center \(\widehat\mu^{[s]}\) in the outer loop, which is the empirical analogue of the fixed-center spectral certificate in \cref{thm:spectral-norm-after-mw-mmw}.
\end{enumerate}

\subsubsection{Accuracy under increasing contamination.}
\cref{fig:reweight-gradients-contamination} compares the estimation error for the sample mean, coordinatewise median, geometric median, oracle inlier mean and the weighted mean produced by \cref{alg:spectral-reweighting} under increasing contamination fractions $\epsilon \in \{0,0.05,0.10,\dots,0.40\}$. Here, the assumed contamination fraction is set equal to the actual contamination fraction. In \cref{fig:reweight-gradients-contamination}, the weighted mean produced by \cref{alg:spectral-reweighting}  is almost indistinguishable from the oracle inlier mean, even with increasing contamination levels. In fact, the estimation error difference between \cref{alg:spectral-reweighting} and the oracle-inlier error never exceeds $7.24e-05$, and the residual outlier mass never exceeds $0.000217$. This is even stronger than what the current theory has proved. Intuitively, this suggests that the directional outliers are  cleanly separated so that the MW-MMW game can still identify an almost oracle inlier weighting well even beyond the worst-case regime covered by the proof. In contrast, the sample mean behaves as expected: its mean $\ell_2$ error grows almost linearly from $\epsilon=0$ to $\epsilon=0.40$. This is consistent with the heuristic bias law $\|\EX[ \overline{\check{\bg}_n} - \bmu_{\bg}]\|_2 \approx \sigma \epsilon$. The geometric median and the coordinatewise median is much more stable than the sample mean, but the estimation error still grows with increasing contamination levels.

\begin{figure}[H]
	\captionsetup{justification=justified, singlelinecheck=off}
	\centering
	\subfloat{\includegraphics[width=0.75\textwidth]{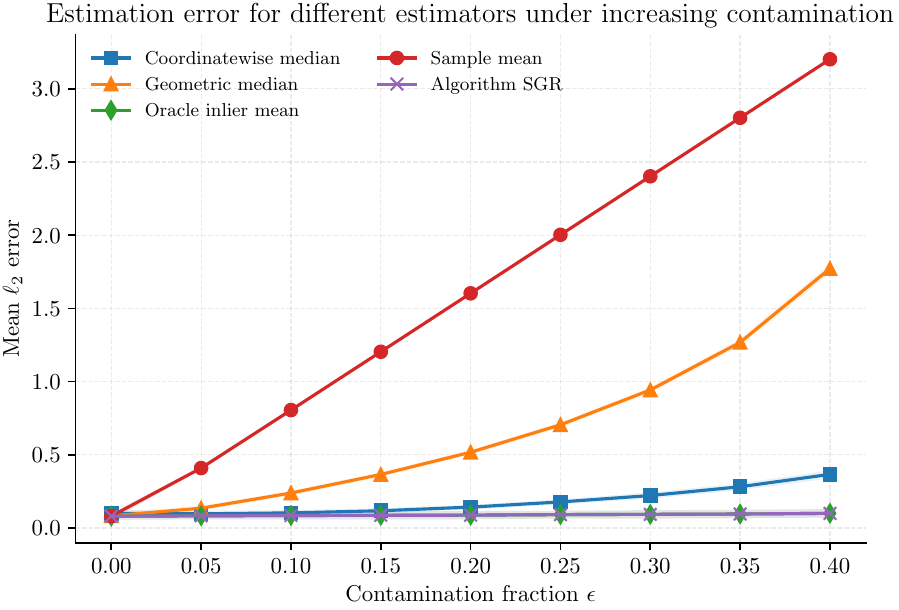}}
	\caption{Mean $\ell_2$ estimation error under increasing contamination in the case of directional outliers.}\label{fig:reweight-gradients-contamination}
\end{figure}

\subsubsection{Progress over outer-loop iterations.}
\cref{fig:reweight-gradients-outer_iterations} shows the full outer-loop history for one representative run with contamination fraction \(\epsilon=0.10\). From \cref{fig:reweight-gradients-outer_iterations}, we have the following observations:
\begin{itemize}
	\item The weighted-mean error drops sharply within four outer iterations, and it then plateaus at $0.091079$. This is aligned with the outer-loop convergence theorem in \cref{thm:outer-loop-convergence}.
	
	\item The weight update $\|w^{[s]}-w^{[s-1]}\|_1$ decreases from $1.49\times 10^{-1}$ to
	$3.23\times10^{-5}$, and the fixed-center update $\|\bmu^{[s+1]}-\bmu^{[s]}\|_2$ decreases from $2.47\times10^{-2}$ to $1.48\times10^{-4}$, both of which suggest that the empirical stabilization rule is detecting a fixed point.
	
	\item The fixed-center spectral norm decreases from \(1.746736\) under uniform weights to \(1.062284\) under the final robust weights, then remains around $1.062$. This suggests that after a few outer iterations the objective has reached the clean inlier scale $\normop{\Sigma_{\bg}}=1$, and subsequent iterations mainly refine the weights and the center.
\end{itemize}
\begin{figure}[H]
	\captionsetup{justification=justified, singlelinecheck=off}
	\centering
	\subfloat{\includegraphics[width=0.4\textwidth]{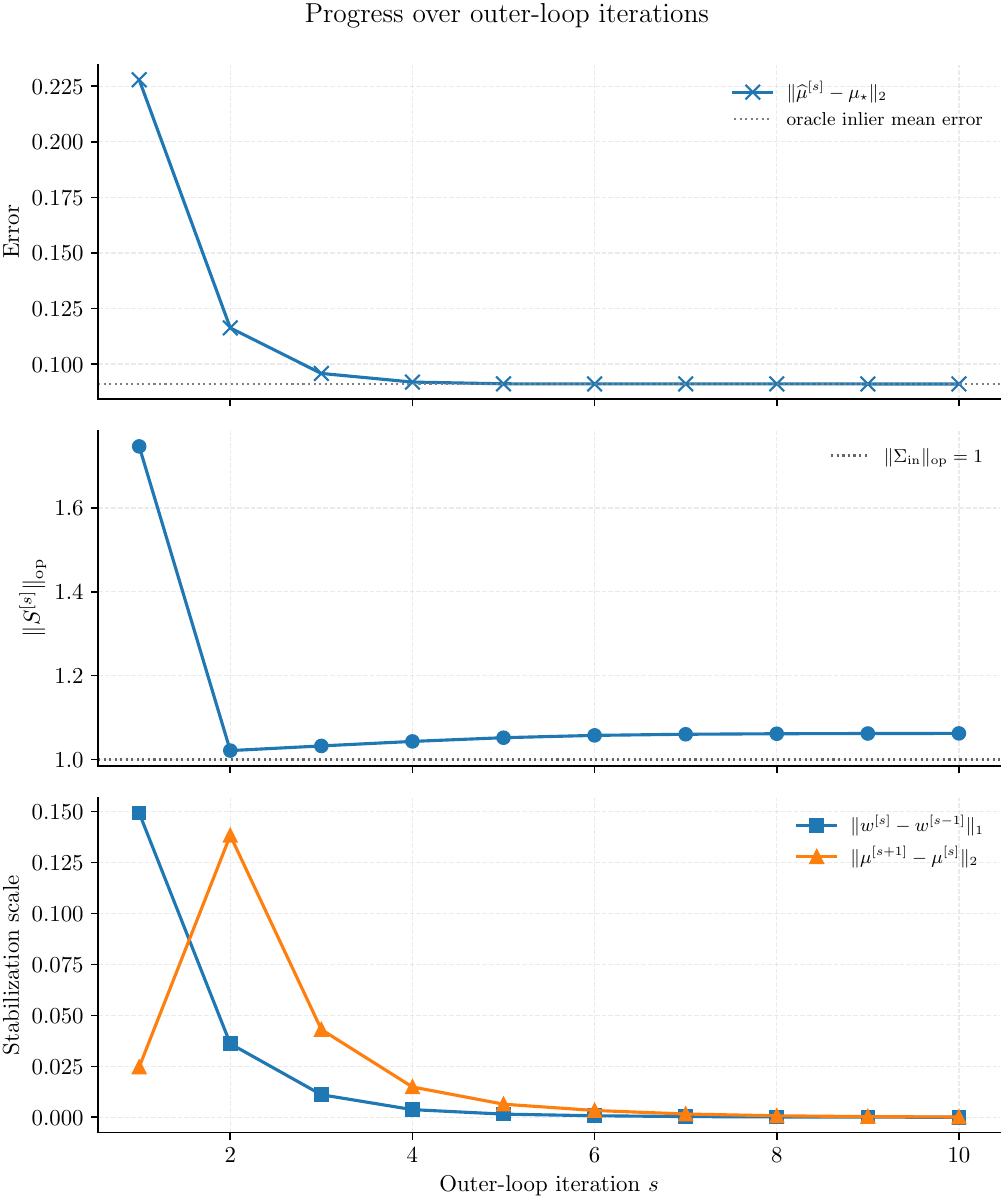}}
	\caption{Progress over outer-loop iterations in the case of directional outliers. The horizontal dotted lines mark the clean covariance scale $\normop{\Sigma_{\mathrm{in}}}=1$ and the oracle inlier-mean error.}\label{fig:reweight-gradients-outer_iterations}
\end{figure}

\subsubsection{Sensitivity to the assumed contamination level.}
\cref{fig:reweight-gradients-epsilon-sensitivity} studies the misspecified contamination level, where the actual contamination fixed at $0.10$ and the assumed contamination varies from
$\epsilon \in \{0.05,0.08,0.10,0.12,0.15,0.20\}$.
\begin{figure}[H]
	\captionsetup{justification=justified, singlelinecheck=off}
	\centering
	\subfloat{\includegraphics[width=\textwidth]{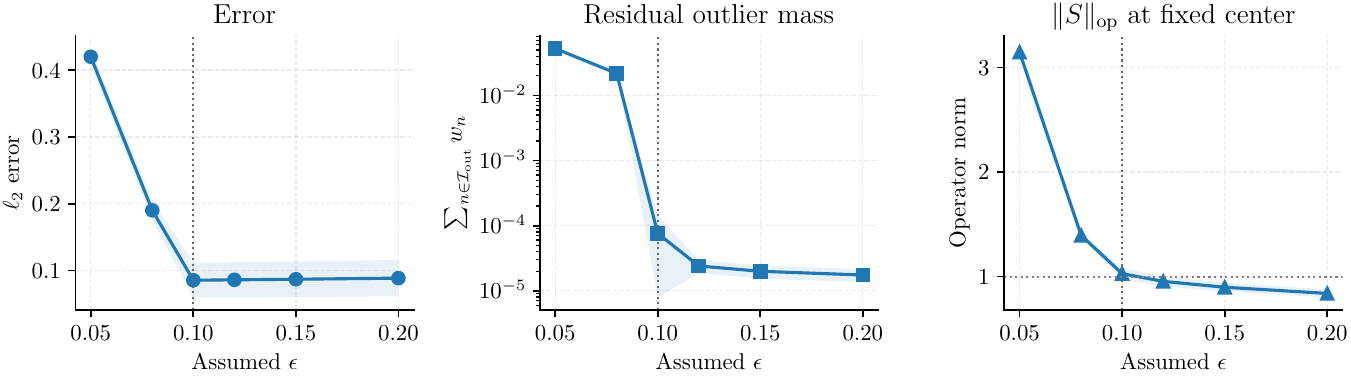}}
	\caption{Sensitivity to the user-supplied contamination level when the actual contamination is $0.10$.  The vertical dotted line marks the correctly specified value.}\label{fig:reweight-gradients-epsilon-sensitivity}
\end{figure}
From \cref{fig:reweight-gradients-epsilon-sensitivity}, we have the following observations:

\begin{itemize}
	\item 	Underestimating \(\epsilon\) is harmful because the capped simplex is not restrictive enough when the assumed contamination is too small:
	at \(\epsilon_{\mathrm{assumed}}=0.05\) the mean error is \(0.419977\) and the residual outlier mass is \(0.052632\).  Once the assumed level reaches the true level, the error returns to the oracle scale: at \(\epsilon_{\mathrm{assumed}}=0.10\), the mean error is \(0.085506\) and the outlier mass is \(7.6\times10^{-5}\).  Mild overestimation is numerically benign in this experiment, although it increases the admissible cap and may reduce efficiency in harder instances.
	\item Mild overestimation is numerically benign in this experiment, although it increases the admissible cap and may reduce efficiency in harder instances. For $\epsilon \in \{0.12,0.15,0.20\}$, the mean error stays in the narrow range $0.086134$-$0.088659$, comparable to the error at \(\epsilon_{\mathrm{assumed}}=0.10\), which is \(0.085506\).
\end{itemize}

\subsection{Numerical experiments for \cref{alg:robust-dgmm}}
\label{sec:robust-dgmm-numerical-experiments}
\paragraph{Data-generating model.}
The numerical experiments in this section compared the following estimation methods:
\begin{enumerate}
	\item \emph{Naive DGMM}: the diagonal DGMM estimator run directly on the observations, ignoring both additive noise and contamination.
	\item \emph{Noise-aware DGMM}: the same moment estimator with the known additive covariance included in the model moments, but without spectral gradient reweighting.
	\item \emph{RobustDGMM}: \cref{alg:robust-dgmm}, which is the DGMM and Gaussian mixture specialization of the main algorithm SGR-GMM \cref{alg:robust-sgr-gmm}, where we use the noise-aware model moments, per-order weights from \cref{alg:spectral-reweighting}, and robust diagonal order weights.
	\item \emph{sklearn EM}: an expectation-maximization likelihood baseline initialized from the same \verb+k-means+++ centers in the outlier-geometry comparison.
\end{enumerate}

Here we use the following data-generating model:
\begin{itemize}
	\item \textbf{Inliers.} The inliers are drawn from \cref{model:heteroscedastic}, following \cite{zhang2025diagonally}, and we set 
	\begin{align*}
		d = 5, K = 2, N = 1000, R_{\min} = 2, R_{\max}^{\mathrm{true}} = 2, R_{\max}^{\mathrm{fit}} =  2, L = 4,
	\end{align*}
	with the centers drawn uniformly at random from the sphere with radius $5$, covariance singular values drawn uniformly at random from the interval $[1, 2]$, and parameter initialization by $k$-means++ to align with sklearn EM for comparison.
	\item \textbf{Additive noise.} Additive noise is isotropic Gaussian with known covariance $
	\Sigma_{\bxi}=0.10\,I_d$.
	\item \textbf{Outliers.} We consider 
	\emph{Gaussian replacement outliers}, where each contaminated observation is replaced by an isotropic Gaussian outlier with standard deviation \verb+outlier_std=4.0+. 
\end{itemize}

\paragraph{Metrics.} The error metrics follow from \cite{zhang2025diagonally}, namely, using the mixture component permutation $\widehat\sigma\in S_K$ that minimizes the average relative covariance error,
\[
\widehat\sigma\in\argmin_{\sigma\in S_K}
\frac1K\sum_{j=1}^K
\frac{\|\widehat\Sigma_j-\Sigma_{\sigma(j)}^{\star}\|_F}{\|\Sigma_{\sigma(j)}^{\star}\|_F},
\]
we compute the average relative errors in the mixing weights, centers, and covariances:
\[
\err_{\pi}=\frac1K\sum_{j=1}^K \frac{|\widehat\pi_j-\pi_{\widehat\sigma(j)}^{\star}|}{|\pi_{\widehat\sigma(j)}^{\star}|},
\quad
\err_{\mu}=\frac1K\sum_{j=1}^K \frac{\|\widehat\mu_j-\mu_{\widehat\sigma(j)}^{\star}\|_2}{\|\mu_{\widehat\sigma(j)}^{\star}\|_2},
\quad
\err_{\Sigma}=\frac1K\sum_{j=1}^K \frac{\|\widehat\Sigma_j-\Sigma_{\widehat\sigma(j)}^{\star}\|_F}{\|\Sigma_{\widehat\sigma(j)}^{\star}\|_F}.
\]

Additionally, we compute the outlier mass  robust diagnostic to measure how aggressively \cref{alg:spectral-reweighting} is suppressing contaminated observations:
\begin{align*}
	\tau_{\outlier} \coloneqq \sum_{n\in\mathcal I_{\outlier}}\widehat{w}_n^{(k)},
	\qquad k=1,\dots,L.
\end{align*}

\subsubsection{Convergence and reweighting diagnostics }
\cref{fig:convergence-diagnostics} records a representative noisy-and-contaminated run and suggests that the numerical implementation behaves as the theory suggests. The parameter-error plots show that the center and covariance errors improve rapidly during the first few reweightings and then stabilize. The mixing weight error likely struggles more due to the softmax reparameterization. The objective, parameter-change, and weight-change plots also show eventual stabilization. This behavior is consistent with the decomposition in the main theorem of the paper in \cref{thm:finite-sample-parameter}: once the robust gradient estimation error is reduced, the remaining error is dominated by local numerical optimization and by the intrinsic finite-sample moment error.

\begin{figure}[H]
	\centering
	\includegraphics[width=0.98\textwidth]{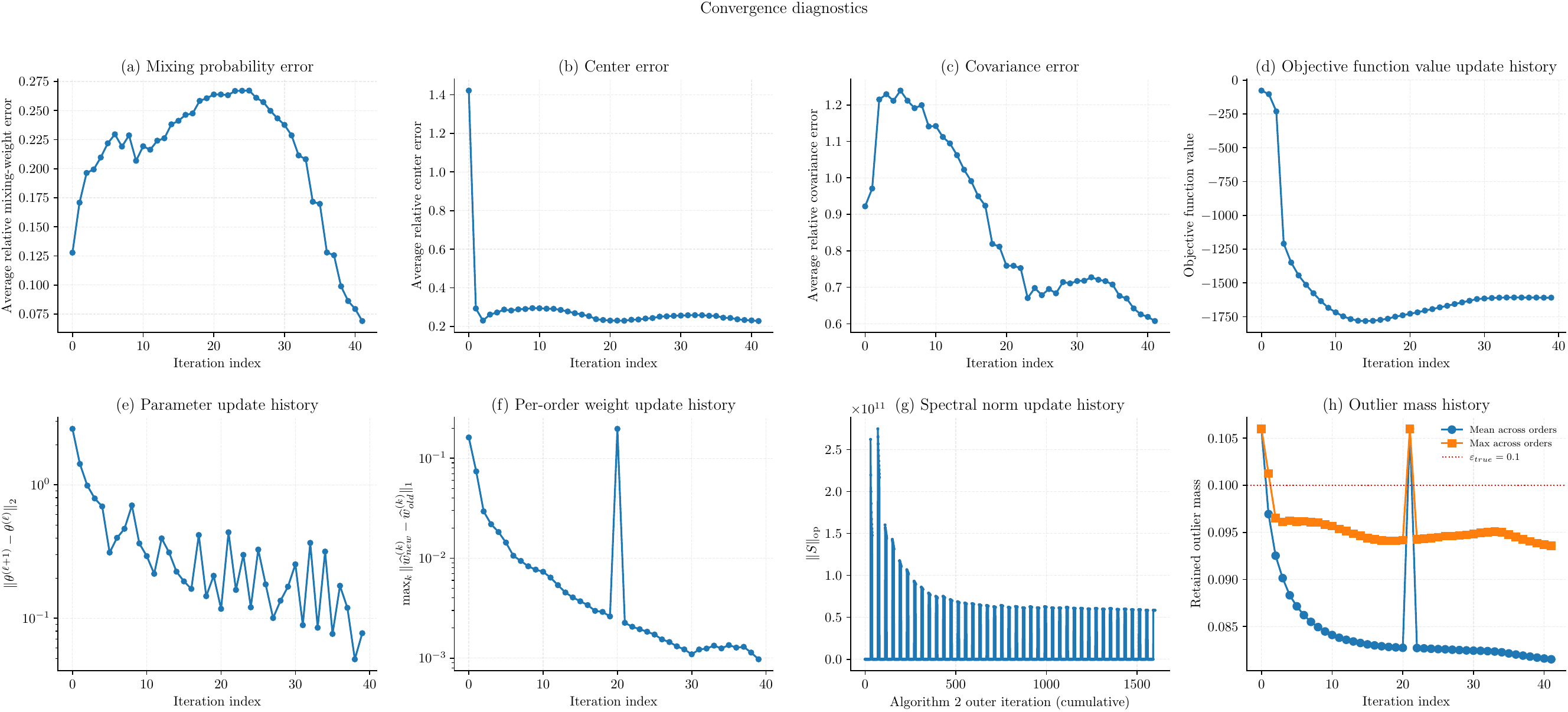}
	\caption{Representative convergence diagnostics for \cref{alg:robust-dgmm} under additive noise and contamination fraction $\epsilon = 0.1$.  The figure records the objective, parameter displacement, robust weight changes, order-weight evolution, per-order retained outlier mass, and final component-level errors.}
	\label{fig:convergence-diagnostics}
\end{figure}

\subsubsection{Repeated-trial statistical validation}
In \cref{tab:repeated-trials-summary} and \cref{fig:repeated-trials}, we report the repeated trials to test \cref{alg:robust-dgmm}. In particular, we observe the following:
\begin{enumerate}
	\item In the \emph{clean} configuration, all three DGMM variants are numerically indistinguishable. This is what one should expect since \cref{alg:robust-dgmm} reduces to the baseline DGMM in the clean configuration.
	\item In the \emph{noise-only} configuration, \cref{alg:robust-dgmm} and the noise-aware estimator coincide. This shows that the robust gradient reweighting does not induce unintended side effects when contamination is absent.
	\item In the \emph{contamination-only} and \emph{noise-plus-contamination} configurations, \cref{alg:robust-dgmm} yields a substantial reduction in errors for  mixing probability estimation, center estimation, and covariance estimation. The trade-off, however, is a longer runtime, due to repeated calls to \cref{alg:spectral-reweighting}.
\end{enumerate}
\begin{table}[!t]
	\centering
	\caption{Repeated-trial summary for the end-to-end DGMM experiment.  The reported standard deviations are across five statistical seeds in the fast-mode notebook.}
	\label{tab:repeated-trials-summary}
	\begingroup
	\small
	\begin{tabular}{llrrrrrrr}
		\toprule
		Configuration & Method & Err\(_\pi\) mean & Err\(_\pi\) std. & Err\(_\mu\) mean & Err\(_\mu\) std. & Err\(_\Sigma\) mean & Err\(_\Sigma\) std. & Runtime mean (s)\\
		\midrule
		Clean & Naive DGMM       & 0.069885 & 0.065303 & 0.047776 & 0.036162 & 0.198338 & 0.134030 & 0.413457\\
		Clean & Noise-aware DGMM & 0.069885 & 0.065303 & 0.047776 & 0.036162 & 0.198338 & 0.134030 & 0.421000\\
		Clean & RobustDGMM       & 0.069885 & 0.065303 & 0.047776 & 0.036162 & 0.198338 & 0.134030 & 0.408343\\
		\midrule
		Contamination only & Naive DGMM       & 1.154906 & 1.828217 & 1.517911 & 1.020011 & 38.457284 & 45.599675 & 0.385756\\
		Contamination only & Noise-aware DGMM & 1.154906 & 1.828217 & 1.517911 & 1.020011 & 38.457284 & 45.599675 & 0.385914\\
		Contamination only & RobustDGMM       & 0.067575 & 0.056400 & 0.523574 & 0.228449 & 2.458773  & 0.884887  & 392.628207\\
		\midrule
		Noise only & Naive DGMM       & 0.352071 & 0.425875 & 0.205168 & 0.115748 & 0.991184 & 0.730812 & 0.427669\\
		Noise only & Noise-aware DGMM & 0.065715 & 0.068061 & 0.077080 & 0.070463 & 0.345049 & 0.320983 & 2.191373\\
		Noise only & RobustDGMM       & 0.065715 & 0.068061 & 0.077080 & 0.070463 & 0.345049 & 0.320983 & 1.527712\\
		\midrule
		Both & Naive DGMM       & 1.687308 & 1.648177 & 1.704079 & 0.730592 & 25.719702 & 23.690259 & 0.401185\\
		Both & Noise-aware DGMM & 1.402508 & 1.716189 & 1.622599 & 0.912302 & 27.490586 & 30.497567 & 1.253636\\
		Both & RobustDGMM       & 0.226619 & 0.295889 & 0.465811 & 0.249497 & 2.177167  & 0.750899  & 388.968583\\
		\bottomrule
	\end{tabular}
	\endgroup
\end{table}

\begin{figure}[H]
	\centering
	\includegraphics[width=0.98\textwidth]{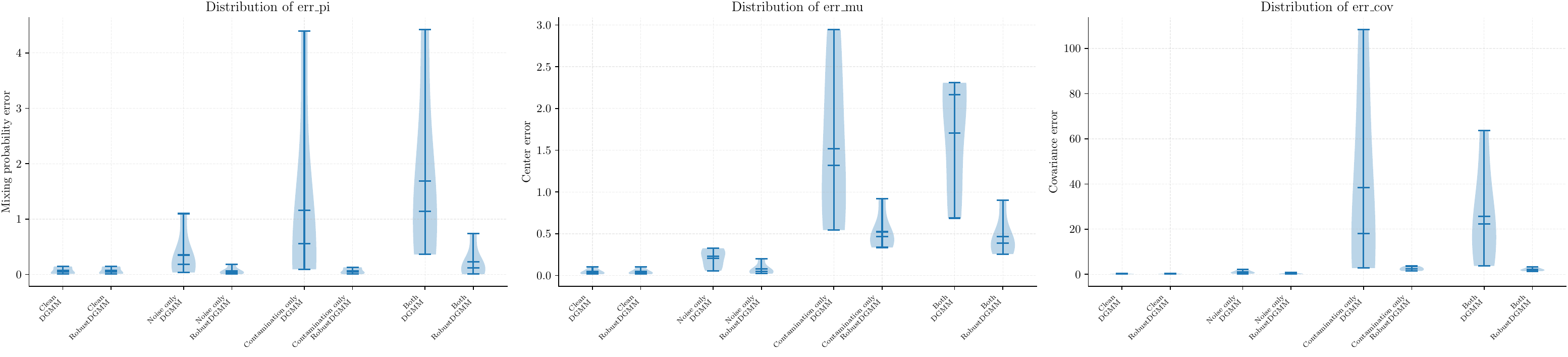}
	\caption{Distribution of repeated-trial DGMM errors in the clean, noise-only, contamination-only, and combined noise-plus-contamination configurations.}
	\label{fig:repeated-trials}
\end{figure}

\subsubsection{Baseline comparisons and the role of outlier geometry}
\cref{fig:comparison-outlier-geometries} compares naive DGMM, noise-aware DGMM, sklearn EM and RobustDGMM using the same  initial parameters (obtained from \verb+k-means+++). For the purpose of comparison, in addition to the \emph{Gaussian replacement outliers}, we additionally consider the more challenging \emph{uniform-box outliers}, where the outliers are drawn from a noisy uniform box $U([4,10]^d)+\cN(0,0.1I_d)$.  The robust moment estimator improves substantially on the two non-robust moment estimators under structured contamination. The comparison with EM depends on the outlier geometry.  Under the more difficult \emph{uniform-box outliers}, RobustDGMM has notably smaller estimation errors in mixing weight, center, and covariance than EM. For the easier \emph{Gaussian replacement outliers}, the gap between the two methods is less decisive. This is likely due to two reasons: first, it is easier for a likelihood fit initialized near the inlier clusters under \emph{Gaussian replacement outliers};  second, the present experiment is low-dimensional with a moderate sample size, which is too small for moment-based methods to show their practical advantages against the EM algorithm. 

\begin{figure}[H]
	\centering
	\includegraphics[width=0.98\textwidth]{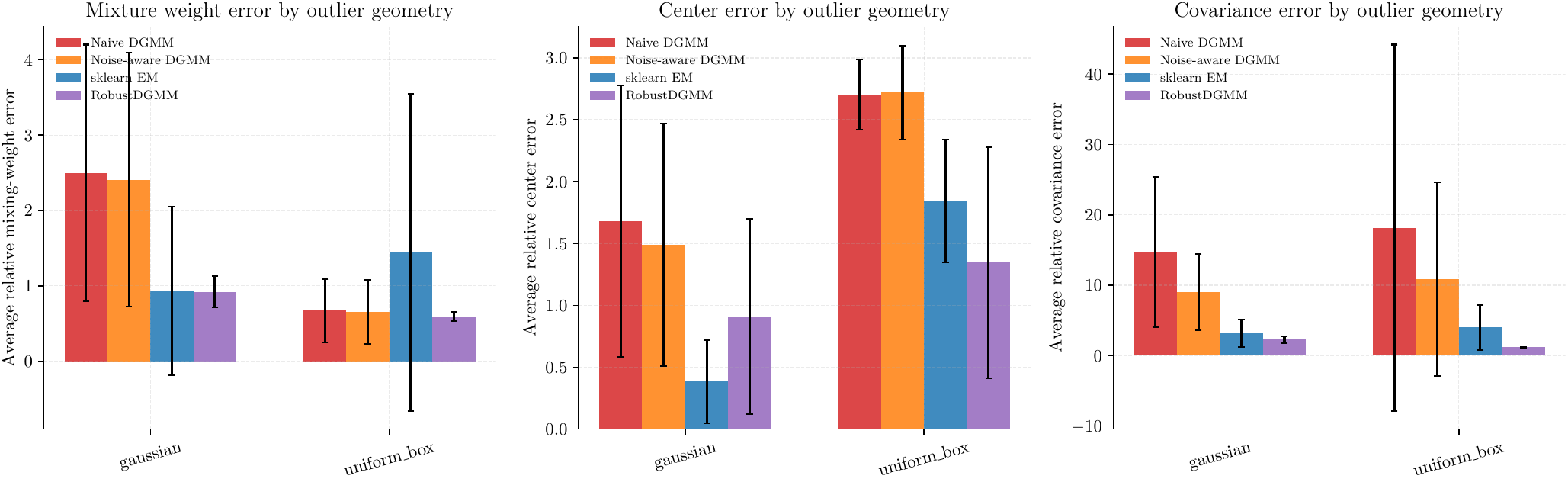}
	\caption{Mixing probability error, center error, and covariance errors across the methods under different outlier geometries.}
	\label{fig:comparison-outlier-geometries}
\end{figure}

\section{Conclusions}
\label{sec:conclusions}
In this work, we develop the SGR-GMM algorithm based on a spectral gradient reweighting primitive in the space of moment-matching gradients. The final local finite-sample GMM parameter estimation error decomposes into the following interpretable quantities: local identification \((\lambda^\star)\), clean inlier stability \((\delta_{\mu,k},\delta_{\Sigma,k})\), contamination level \((\alpha_\epsilon)\), inner spectral-game optimization error \((\delta_{T,k})\), and the numerical optimizer residual \((\delta_{\opt})\).

We further specialize the SGR-GMM algorithm in the framework of DGMM for estimating heteroscedastic low-rank Gaussian mixtures observed under additive Gaussian noise and strong $\epsilon$-contamination. The numerical experiments verify that the spectral reweighting primitive is near-oracle under separated directional contamination and robust DGMM improves over non-robust moment baselines under structured contamination. The comparison with EM is more specific to the outlier geometry, which is expected because likelihood-based methods can be strong in benign outlier geometries but might struggle in more challenging outlier geometries. 

Beyond the focus of this paper, \cref{alg:robust-sgr-gmm} can be adapted to other moment-based estimation procedures in settings where likelihood-based estimation is unavailable, misspecified, or computationally inconvenient. These procedures include the classical method of moments, minimum-distance estimators, linear and nonlinear instrumental-variable estimators, iterated and continuously updated GMM, and more general \(Z\)-estimation procedures (see, e.g., \cite{hall2004generalized} for additional moment-based procedures). More broadly, \cref{alg:robust-sgr-gmm} suggests a general algorithmic design principle: for estimators determinated by empirical first-order information, by making the empirical update directions spectrally stable, robustness to adversarial contamination can be achieved effectively while still preserving the original estimating equations and optimization architecture as much as possible. 

\section*{Funding}
L.Z.~and A.S.~were supported in part by AFOSR FA9550-23-1-0249, DARPA HR0011-25-3-E002, NSF DMS 2510039, and the Simons Foundation Math+X Investigator Award.

\section*{Data availability statement}
The code and data used in \cref{sec:numerical-experiments} are available at \url{https://github.com/liu-lzhang/sgr-gmm}. 

\bibliographystyle{plain}
\bibliography{sgr_gmm_references.bib}

\begin{thebibliography}{10}

\bibitem{arora2012multiplicative}
Sanjeev Arora, Elad Hazan, and Satyen Kale.
\newblock The multiplicative weights update method: a meta-algorithm and
  applications.
\newblock {\em Theory Comput.}, 8:121--164, 2012.

\bibitem{bubeck_convex_2015}
Sébastien Bubeck.
\newblock Convex {Optimization}: {Algorithms} and {Complexity}, November 2015.
\newblock arXiv:1405.4980 [math].

\bibitem{cheng2020high}
Yu~Cheng, Ilias Diakonikolas, Rong Ge, and Mahdi Soltanolkotabi.
\newblock High-dimensional robust mean estimation via gradient descent.
\newblock In {\em International conference on machine learning}, pages
  1768--1778. PMLR, 2020.

\bibitem{dalalyan2022allinone}
Arnak~S. Dalalyan and Arshak Minasyan.
\newblock All-in-one robust estimator of the {Gaussian} mean.
\newblock {\em The Annals of Statistics}, 50(2), April 2022.

\bibitem{dennisschnabel1996numerical}
J.~E. Dennis, Jr. and Robert~B. Schnabel.
\newblock {\em Numerical methods for unconstrained optimization and nonlinear
  equations}, volume~16 of {\em Classics in Applied Mathematics}.
\newblock Society for Industrial and Applied Mathematics (SIAM), Philadelphia,
  PA, 1996.
\newblock Corrected reprint of the 1983 original.

\bibitem{diakonikolas2019robust}
Ilias Diakonikolas, Gautam Kamath, Daniel Kane, Jerry Li, Ankur Moitra, and
  Alistair Stewart.
\newblock Robust estimators in high-dimensions without the computational
  intractability.
\newblock {\em SIAM Journal on Computing}, 48(2):742--864, 2019.

\bibitem{diakonikolas2019sever}
Ilias Diakonikolas, Gautam Kamath, Daniel Kane, Jerry Li, Jacob Steinhardt, and
  Alistair Stewart.
\newblock Sever: A robust meta-algorithm for stochastic optimization.
\newblock In {\em International Conference on Machine Learning}, pages
  1596--1606. PMLR, 2019.

\bibitem{diakonikolas2023algorithmic}
Ilias Diakonikolas and Daniel~M Kane.
\newblock {\em Algorithmic high-dimensional robust statistics}.
\newblock Cambridge university press, 2023.

\bibitem{dong2019quantum}
Yihe Dong, Samuel Hopkins, and Jerry Li.
\newblock Quantum entropy scoring for fast robust mean estimation and improved
  outlier detection.
\newblock {\em Advances in Neural Information Processing Systems}, 32, 2019.

\bibitem{donohohuber1983breakdown}
David Donoho and Peter~J. Huber.
\newblock The notion of breakdown point.
\newblock In {\em A {F}estschrift for {E}rich {L}. {L}ehmann}, Wadsworth
  Statist./Probab. Ser., pages 157--184. Wadsworth, Belmont, CA, 1983.

\bibitem{hall2004generalized}
Alastair Hall.
\newblock {\em Generalized Method of Moments}.
\newblock Wiley Online Library, 2004.

\bibitem{hampel1968contributions}
Frank~Rudolf Hampel.
\newblock {\em Contributions to the theory of robust estimation}.
\newblock PhD thesis, University of California, Berkeley, 1968.

\bibitem{hansen1982large}
Lars~Peter Hansen.
\newblock Large sample properties of generalized method of moments estimators.
\newblock {\em Econometrica}, 50(4):1029--1054, 1982.

\bibitem{hazan_introduction_2023}
Elad Hazan.
\newblock Introduction to {Online} {Convex} {Optimization}, August 2023.
\newblock arXiv:1909.05207 [cs].

\bibitem{hopkins2020robust}
Sam Hopkins, Jerry Li, and Fred Zhang.
\newblock Robust and heavy-tailed mean estimation made simple, via regret
  minimization.
\newblock {\em Advances in Neural Information Processing Systems},
  33:11902--11912, 2020.

\bibitem{huber1964robust}
Peter~J Huber.
\newblock Robust estimation of a location parameter.
\newblock {\em The Annals of Mathematical Statistics}, 35(1):73--101, 1964.

\bibitem{huber1965robusttest}
Peter~J. Huber.
\newblock A robust version of the probability ratio test.
\newblock {\em Ann. Math. Statist.}, 36:1753--1758, 1965.

\bibitem{lai2016agnostic}
Kevin~A Lai, Anup~B Rao, and Santosh Vempala.
\newblock Agnostic estimation of mean and covariance.
\newblock In {\em 2016 IEEE 57th Annual Symposium on Foundations of Computer
  Science (FOCS)}, pages 665--674. IEEE, 2016.

\bibitem{lerman2025global}
Gilad Lerman, Kang Li, Tyler Maunu, and Teng Zhang.
\newblock Global convergence of iteratively reweighted least squares for robust
  subspace recovery.
\newblock {\em arXiv preprint arXiv:2506.20533}, 2025.

\bibitem{lermanmaunu2018fast}
Gilad Lerman and Tyler Maunu.
\newblock Fast, robust and non-convex subspace recovery.
\newblock {\em Inf. Inference}, 7(2):277--336, 2018.

\bibitem{loh2025review}
Po-Ling Loh.
\newblock A theoretical review of modern robust statistics.
\newblock {\em Annu. Rev. Stat. Appl.}, 12:477--496, 2025.

\bibitem{luenbergerlinear2008}
David~G. Luenberger and Yinyu Ye.
\newblock {\em Linear and {Nonlinear} {Programming}}.
\newblock Number 116 in International {Series} in {Operations} {Research} \&
  {Management} {Science}. Springer US, Boston, MA, third edition edition, 2008.

\bibitem{newey1994large}
Whitney~K Newey and Daniel McFadden.
\newblock Large sample estimation and hypothesis testing.
\newblock In {\em Handbook of Econometrics}, volume~4, pages 2111--2245.
  Elsevier, 1994.

\bibitem{nielsen2010quantum}
Michael~A Nielsen and Isaac~L Chuang.
\newblock {\em Quantum computation and quantum information}.
\newblock Cambridge university press, 2010.

\bibitem{pearson1894contributions}
Karl Pearson.
\newblock Contributions to the mathematical theory of evolution.
\newblock {\em Philosophical Transactions of the Royal Society of London. A},
  185:71--110, 1894.

\bibitem{prasad2020robustgradient}
Adarsh Prasad, Arun~Sai Suggala, Sivaraman Balakrishnan, and Pradeep Ravikumar.
\newblock Robust estimation via robust gradient estimation.
\newblock {\em J. R. Stat. Soc. Ser. B. Stat. Methodol.}, 82(3):601--627, 2020.

\bibitem{rohatgi2022robust}
Dhruv Rohatgi and Vasilis Syrgkanis.
\newblock Robust generalized method of moments: a finite sample viewpoint.
\newblock {\em Advances in Neural Information Processing Systems},
  35:15970--15981, 2022.

\bibitem{ronchetti1979robusttest}
Elvezio Ronchetti.
\newblock {\em Robusthatseigenschaften von Tests}.
\newblock PhD thesis, ETH Zürich, 1979.

\bibitem{ruskai2002quantum}
Mary~Beth Ruskai.
\newblock Inequalities for quantum entropy: a review with conditions for
  equality.
\newblock volume~43, pages 4358--4375. 2002.
\newblock Quantum information theory.

\bibitem{tian2025mmw}
Kevin Tian.
\newblock {CS395T}: Continuous algorithms, part viii: Matrix multiplicative
  weights, 2025.
\newblock Lecture notes.

\bibitem{zhang2025diagonally}
Liu Zhang, Oscar Mickelin, Sheng Xu, and Amit Singer.
\newblock Diagonally-weighted generalized method of moments estimation for
  {G}aussian mixture modeling.
\newblock {\em arXiv preprint arXiv:2507.20459}, 2025.

\bibitem{zhu2022quasigradients}
Banghua Zhu, Jiantao Jiao, and Jacob Steinhardt.
\newblock Robust estimation via generalized quasi-gradients.
\newblock {\em Inf. Inference}, 11(2):581--636, 2022.

\end{thebibliography}

\appendix

\makeatletter
\renewcommand{\theHsection}{appendix.\Alph{section}}
\renewcommand{\theHsubsection}{appendix.\Alph{section}.\arabic{subsection}}
\renewcommand{\theHsubsubsection}{appendix.\Alph{section}.\arabic{subsection}.\arabic{subsubsection}}
\renewcommand{\theHequation}{appendix.\Alph{section}.\arabic{equation}}
\makeatother

\section{Supplementary proofs}
\label{app-sec:robust-sgr-gmm}

\subsection{Supplementary proofs for fixed-center regret bound}
\label{app-sec:mw-mmw-rounds}
\printProofs[mw-mmw-rounds]

\subsection{Supplementary proofs for convergence of the fixed-center updates}
\label{app-sec:fixed-center-updates}
\printProofs[fixed-center-updates]

\subsection{Supplementary proofs for local finite-sample GMM analysis}
\label{app-sec:finite-sample-gmm}
\printProofs[finite-sample-gmm]

\end{document}